\documentclass[a4paper,10pt]{amsart}

\usepackage[utf8]{inputenc}
\usepackage[T1]{fontenc}
\usepackage{amsfonts}
\usepackage{amsthm}
\usepackage{amsmath}
\usepackage[english]{babel}
\usepackage[all]{xy}

\usepackage{ amssymb }
\usepackage{graphicx}
\usepackage{amscd}
\usepackage{latexsym}
\usepackage{hyperref}
\usepackage{mathrsfs}
\usepackage{enumerate}

\usepackage{color}

\usepackage{calc}

\newtheorem{theoremABC}{Theorem}

\theoremstyle{definition}
 \newtheorem{defi}{Definition}[section]

\theoremstyle{remark}
 \newtheorem{remark}[defi]{Remark}
 \newtheorem{exam}[defi]{Example}

\theoremstyle{plain}
\newtheorem{theo}[defi]{Theorem}
\newtheorem{prop}[defi]{Proposition}
\newtheorem{cor}[defi]{Corollary}
\newtheorem{lemma}[defi]{Lemma}

\newtheorem{assumption}[defi]{Assumption}
\newtheorem*{acknow}{Acknowledgments}

\def\e#1\e{\begin{equation}#1\end{equation}}
\def\ea#1\ea{\begin{align}#1\end{align}}

\numberwithin{equation}{section}

\newcommand{\zz}{\mathbb{Z}}

\newcommand{\rr}{\mathbb{R}}
\newcommand{\cc}{\mathbb{C}}
\newcommand{\hh}{\mathbb{H}}

\newcommand{\Z}[1]{\zz_{#1}}

\newcommand{\abs}[1]{\left\vert #1 \right\vert}

\newcommand{\red}{/\!\!/}

\newcommand{\Acal}{\mathcal{A}}
\newcommand{\Bcal}{\mathcal{B}}
\newcommand{\Ccal}{\mathcal{C}}
\newcommand{\Dcal}{\mathcal{D}}
\newcommand{\Ecal}{\mathcal{E}}
\newcommand{\Fcal}{\mathcal{F}}

\newcommand{\Hcal}{\mathcal{H}}
\newcommand{\Ical}{\mathcal{I}}
\newcommand{\Jcal}{\mathcal{J}}
\newcommand{\Kcal}{\mathcal{K}}
\newcommand{\Lcal}{\mathcal{L}}
\newcommand{\Mcal}{\mathcal{M}}
\newcommand{\Ncal}{\mathcal{N}}

\newcommand{\Pcal}{\mathcal{P}}
\newcommand{\Qcal}{\mathcal{Q}}
\newcommand{\Rcal}{\mathcal{R}}
\newcommand{\Scal}{\mathcal{S}}
\newcommand{\Tcal}{\mathcal{T}}
\newcommand{\Ucal}{\mathcal{U}}
\newcommand{\Vcal}{\mathcal{V}}
\newcommand{\Wcal}{\mathcal{W}}
\newcommand{\Xcal}{\mathcal{X}}

\newcommand{\Hund}{\underline{H}}
\newcommand{\Jund}{\underline{J}}
\newcommand{\Kund}{\underline{K}}
\newcommand{\Lund}{\underline{L}}
\newcommand{\Mund}{\underline{M}}
\newcommand{\Sund}{\underline{S}}
\newcommand{\Zund}{\underline{Z}}
\newcommand{\uund}{\underline{u}}

\newcommand{\M}{\mathscr{M}}
\newcommand{\Mg}{\mathscr{M}^\mathfrak{g}}

\newcommand{\Lscr}{\mathscr{L}}

\newcommand{\N}{\mathscr{N}}

\newcommand{\g}{\mathfrak{g}}

\newcommand{\Tel}{\mathrm{Tel}}

\newcommand{\Crit}{\mathrm{Crit}}

\newcommand{\ind}{\mathrm{ind}}

\renewcommand{\d}{\mathrm{d}}

\newcommand{\glag}{generalized Lagrangian correspondence}

\newcommand{\gint}{generalized intersection point}

\newcommand{\acs}{almost complex structure}

\newcommand{\arnaque}{\begin{flushright}
$\Box$
\end{flushright}}

\newcommand{\MW}{Manolescu and Woodward}
\newcommand{\WW}{Wehrheim and Woodward}

\newcommand{\KM}{Kronheimer and Mrowka}

\title[Equivariant Floer homology]{Equivariant Lagrangian Floer homology via cotangent bundles of $EG_N$}

\author{Guillem Cazassus}
\address{Centre for Quantum Mathematics,  University of Southern Denmark, Odense}
\email{g.cazassus@gmail.com}

\thanks{This work was funded by EPSRC grant reference EP/T012749/1.}
\begin{document}

\begin{abstract}
We provide a construction of equivariant Lagrangian Floer homology $HF_G(L_0, L_1)$, for a compact Lie group $G$ acting on a symplectic manifold $M$ in a Hamiltonian fashion, and a pair of $G$-Lagrangian submanifolds $L_0, L_1 \subset M$.

We do so by using symplectic homotopy quotients involving cotangent bundles of an approximation of $EG$. Our construction relies on Wehrheim and Woodward's theory of quilts, and the telescope construction.

We show that these groups are independent of the auxiliary choices involved in their construction, and are $H^*(BG)$-bimodules. In the case when $L_0 = L_1$, we show that their chain complex  $CF_G(L_0, L_1)$ is homotopy equivalent to the equivariant Morse complex of $L_0$.

Furthermore, if zero is a regular value of the moment map $\mu$ and if $G$ acts freely on $\mu^{-1}(0)$, we construct two ``Kirwan morphisms'' from $CF_G(L_0, L_1)$ to $CF(L_0/G, L_1/G)$ (respectively from $CF(L_0/G, L_1/G)$ to $CF_G(L_0, L_1)$).

Our construction applies to the exact and monotone settings, as well as in the setting of the extended moduli space of flat $SU(2)$-connections of a Riemann surface, considered in Manolescu and Woodward's work. Applied to the latter setting, our construction provides an equivariant symplectic side for the Atiyah-Floer conjecture.

\end{abstract}

\maketitle
\tableofcontents


\section{Introduction}
\label{sec:intro}

Lagrangian Floer homology is a group $HF(M; L_0, L_1)$\footnote{It is usually denoted $HF(L_0, L_1)$ when the symplectic manifold is fixed. As this will not be the case in this paper, we prefer to keep $M$ in the notation.}  associated with a pair of Lagrangians $L_0, L_1$ in a symplectic manifold $M$, provided these satisfy some assumptions. Depending on those assumptions, these groups are more or less complicated to define. A particularly difficult setting for defining these groups is when $M$ or $L_0, L_1$ have singularities. In practice, many interesting singular symplectic manifolds and Lagrangians arise as a symplectic reduction: if $M$ is a Hamiltonian $G$-manifold for some compact Lie group, and $L_0, L_1$ are $G$-Lagrangians, then unless the action is nice enough, $M\red G$ and $L_0/G, L_1/G$ might be singular.

For instance, this is the setting of the Atiyah-Floer conjecture in its initial formulation \cite{Atiyahfloer}: given an integral homology 3-sphere $Y$, its instanton homology $I_*(Y)$ should be isomorphic to a Lagrangian Floer homology 
\e
HF(\M(\Sigma); \Lscr(H_0), \Lscr(H_1))
\e 
in the Atiyah-Bott moduli space $\M(\Sigma)$ of flat SU(2)-connexions (which is to say, the SU(2)-character variety of the surface $\Sigma$) of a Heegaard splitting $ Y = H_0 \cup_{\Sigma} H_1$. Unfortunately, this moduli space is singular (as well as the Lagrangians $\Lscr(H_0), \Lscr(H_1)$). Indeed, the main source of difficulty with the conjecture is that the symplectic side $HF(\M(\Sigma); \Lscr(H_0), \Lscr(H_1))$ is not currently well-defined in general. Nevertheless, this conjecture has been studied extensively, and established in some settings where all moduli spaces are smooth \cite{DostoglouSalamon,Salamon_AF,Wehrheim_lbc_asd,
Wehrheim_asd_lbc,SalamonWehrheim,Duncan_thesis,Daemi_Fukaya_strategy,Xu_AF,DaemiFukayaLipyanskiy}.

Yet, Jeffrey \cite{jeffrey} and Huebschmann \cite{huebschmann} observed that the moduli space $\M(\Sigma)$ can be realized as the symplectic reduction of a smooth SU(2)-Hamiltonian space: (an open subset $\N$ of) the so-called extended moduli space. Moreover this space $\N$ contains smooth SU(2)-Lagrangians $L_0, L_1$ such that $\Lscr(H_0) = L_0 /G$ and $ \Lscr(H_1) = L_1 /G$. Manolescu and Woodward \cite{MW} defined symplectic instanton homology $HSI(Y)$ as a (non-equivariant) Lagrangian Floer homology in $\N$, and suggested that a good candidate for the symplectic side of the Atiyah-Floer conjecture would be an \emph{equivariant} version $HF_{SU(2)}(\N ; L_0, L_1)$ of $HSI(Y)$, as a substitute for $HF(\M(\Sigma); \Lscr(H_0), \Lscr(H_1))$.

Several versions of equivariant Floer homologies have appeared in the literature, in different settings:
\begin{itemize}
\item Austin and Braam \cite{AustinBraamMorse} defined an equivariant version of Morse homology  for an equivariant
Morse-Bott function by combining features of Morse and deRham homology.

\item In \cite{Viterbo} Viterbo suggested the definition of an equivariant version of symplectic homology for the reparametrization U(1)-action. This was implemented in \cite{BourgeoisOancea} via a Borel construction.
\item In \cite{MiR_thesis,MiR_Hamilt_GW}, \cite{CGS_sve}, Mundet i Riera and independently Cieliebak, Gaio and Salamon introduced the symplectic vortex equation, which among other things furnishes a possible way of approaching Atiyah-Floer type problems. It was used to build homology group associated with a Hamiltonian $G$-manifold  that could be called an equivariant Floer homology. In the closed string case (no Lagrangians), Frauenfelder \cite{Frauenfelder} defined ``moment Floer homology''.  For a pair of Lagrangians, Woodward  defined ``quasimap Floer homology'' \cite{Woodward_quasimap_FH}. This equation is particularly well-suited for relating equivariant invariants to invariants of the symplectic quotient \cite{TianXu_symplectic_GLSM,Woodwardqkir1,Woodwardqkir2,Woodwardqkir3,NguyenWoodwardZiltener}, in analogy with the Kirwan map \cite{Kirwan_thesis}.

\item In \cite{SeidelSmith}, Seidel and Smith defined a version of equivariant Lagrangian Floer homology for symplectic involutions ($G = \Z{2}$). Generalizations to finite group actions appeared in \cite{HLSflex,HLSflex_corr,HondaBao,ChoHong}.

\item In \cite{HLSsimplicial}, Hendricks, Lipshitz and Sarkar defined an equivariant Lagrangian Floer homology for Lie group actions. Their construction involves the theory of $(\infty,1)$ categories.

\item In \cite{KimLauZheng}, in the case of a single Lagrangian $L_0 = L_1$, Kim, Lau and Zheng defined an equivariant Lagrangian Floer homology using a Borel construction similar (but slightly different, see Remark~\ref{rem:comparison_KLZ}) to the one in this paper.

\item In gauge theory, \KM\  \cite{KMbook} defined several versions of monopole homology, corresponding to U(1)-equivariant theories. For instanton homology, a first construction was given in Austin-Braam \cite{AustinBraamEquivFloer}, and later  generalized by Miller Eismeier \cite{Miller_equiv}. Daemi and Scaduto also defined a version for knots \cite{DaemiScaduto_equiv_si}, related to singular instanton homology.

\end{itemize}

In this paper, we present another construction for equivariant Lagrangian Floer homology for Hamiltonian action of a compact Lie group.  Our construction  applies to \MW's setting, in which case it gives an equivariant symplectic side for the Atiyah-Floer conjecture. See Section~\ref{sec:Equivariant_HSI} for further details. 
Also, it is  as simple as possible, both algebraically and analytically. Algebraically, it is based on the telescope construction, all we need is contained in Section~\ref{sec:telescopes}, some of which might be new. Analytically, we work in a setting that allows us to use domain dependent perturbations, so transversality is standard \cite{FloerHoferSalamon}. In particular we don't need to achieve equivariant transversality, a usually delicate problem.

Our personal motivation for this construction is to recast the Donaldson polynomials as an extended Field theory, where equivariant Floer homology groups would play the role of 3-morphism spaces in the target 3-category, and generalized Donaldson polynomials for 4-manifolds with corners would take their values in these groups. We refer to \cite{ham} for more details. 

\subsection{Outline of the construction}\label{ssec:outline_constr} 
It is known since Floer that Morse homology corresponds to Lagrangian Floer homology in a cotangent bundle. This identification provides a dictionary between these two theories. Our strategy is to first reformulate the definition of equivariant homology in a Morse-theoretical way, and then use this dictionary to translate this construction to that Lagrangian setting.

Let $X$ be a closed smooth manifold acted on by a compact Lie group $G$. By definition, equivariant homology is the homology of the homotopy quotient
\e
H_*^G(X) = H_*(X\times_G EG),
\e
and using a finite dimensional smooth approximation $\lbrace EG_N\rbrace_N$ of 
\e
EG = \lim_{\overrightarrow{\ N\ }}{EG_N},
\e
with inclusions $i_N\colon EG_N \to EG_{N+1}$, one can write 
\e
H_*^G(X) =  \lim_{\overrightarrow{\ N\ }}H_*(X\times_G EG_N).
\e
Now pick a Morse function $f_N$ on $X\times_G EG_N$ for each $N$. One may now express the equivariant homology as a direct limit of Morse
homologies: 
\e
H_*^G(X) =  \lim_{\overrightarrow{\ N\ }}HM_*(X\times_G EG_N, f_N), 
\e
using Morse-theoretic pushforwards of 
\e
id_X\times_G i_N\colon X\times_G EG_N \to X\times_G EG_{N+1}.
\e

Now if $X$ is a $G$-manifold, then $M= T^*X$ is a Hamiltonian $G$-manifold, and its zero section $L_0 = L_1= 0_X\subset T^*X$ is a $G$-Lagrangian. In analogy with the isomorphism  $HF(T^*X; 0_X, 0_X) \simeq HM(X)$, one can define the equivariant Lagrangian Floer homology of $(T^*X; 0_X, 0_X)$ by:
\ea\label{eq:equiv_T*X}
HF^G(T^*X; 0_X, 0_X) &= HM_*^G(X) \\
&=  \lim_{\overrightarrow{\ N\ }}HF(T^* (X\times_G EG_N); 0_{X\times_G EG_N}, 0_{X\times_G EG_N}). \nonumber
\ea
It turns out that 
\ea
T^* (X\times_G EG_N) &= (T^* X \times T^* EG_N) \red G\text{, under which:}\\
0_{X\times_G EG_N} &= (0_{X}\times 0_{EG_N})/G.
\ea
Therefore (\ref{eq:equiv_T*X}) provides a general definition for $HF^G(M;L_0, L_1)$ for more general $G$-Hamiltonian triples $(M;L_0, L_1)$:
\e\label{eq:def_equiv_Floer_homol}
HF^G(M;L_0, L_1) = \lim_{\overrightarrow{\ N\ }}HF((M \times T^* EG_N) \red G; (L_0\times 0_{EG_N})/G,(L_1\times 0_{EG_N})/G),
\e
assuming $(M;L_0, L_1)$ satisfy standard assumptions (exactness or monotonicity, see Assumptions~\ref{ass:exact_setting} and \ref{ass:monotone_setting}) so that the Floer homologies are well-defined. Constructing the maps from $N$ to $N+1$ will involve quilts that generalize  the Morse-theoretic  pushforwards.

\subsection{Statement of results}\label{ssec:results} 
Throughout this paper we work with $\Z{2}$-coefficients to avoid sign discussions, but we believe that everything should work with $\zz$ coefficients as well, provided Lagrangians are endowed with equivariant relative Pin-structures. The following statement summarizes Propositions~\ref{prop:Lcal_moduli_space}, \ref{prop:abs_grading_CF_G},  \ref{prop:phi_htpy_commutes_alpha_beta},  \ref{prop:htpy_continuation}, \ref{prop:relation_kappa} and  \ref{prop:bimod_str}.

\begin{theoremABC}\label{th:construction_intro} The construction outlined in Section~\ref{ssec:outline_constr} can be implemented in the exact and monotone setting (Assumptions~\ref{ass:exact_setting}, \ref{ass:monotone_setting}): at the chain level, it defines a \emph{telescope}
\e
CF_G(M;L_0, L_1) = \Tel (CF_N, \alpha_N),
\e
where $CF_N$ is the Floer chain complex in $(M\times T^* EG_N )\red G$ and $\alpha_N \colon CF_N \to CF_{N+1}$ chain morphisms induced by the inclusions $i_N\colon EG_N \to EG_{N+1}$. The homotopy type of this telescope complex is independent of the auxiliary choices involved in the constructions: Hamiltonian perturbations, \acs s. Its associated homology group $HF_G(M;L_0,L_1)$ is an $H^*(BG)$-bimodule.

If furthermore, in the sense of Definition~\ref{def:G-gradings}, $M$ admits a $\Z{n}$ Maslov $G$-cover $\widehat{\Lcal}_G$ of its Lagrangian Grassmannian bundle, and $L_0$ and $L_1$ admit a $(\widehat{\Lcal}_G, G)$-grading, then $CF_G(M;L_0, L_1)$ has an absolute grading over $\Z{n}$.
\end{theoremABC}

The next three theorems are likely to hold for the other constructions of equivariant Lagrangian Floer homology, but are new to our best knowledge.

\vspace*{.2cm}

In the case when the Lagrangians coincide, we prove the following, using the Piunikhin-Salamon-Schwarz construction \cite{PSS}:

\begin{theoremABC}\label{th:Floer_vs_Morse_intro}
Let $L\subset M$  satisfy either Assumption~\ref{ass:exact_setting} or \ref{ass:monotone_setting}, with $L_0 = L_1 = L$. Then $CF_G(M; L,L))$ is homotopy equivalent to $CM_G(L)$, and the resulting homology groups are isomorphic  as $H^*(BG)$-bimodules. In particular, if $X$ is a smooth, compact $G$-manifold, then $CM_G(X) \simeq CF_G(T^* X; 0_X, 0_X)$.
\end{theoremABC}

When the symplectic quotient is nice enough, we define in Section~\ref{sec:Kirwan_maps} ``Kirwan morphisms'' relating equivariant Floer homology with Floer homology of the quotient:

\begin{theoremABC}\label{th:Kirwan_maps_intro} If the action of $G$ on $M$ is regular (in the sense of Definition~\ref{def:reduction}, so that the symplectic quotient $M\red G$ is smooth), then the equivariant Floer complex of $(M;L_0,L_1)$ is related to the non-equivariant Floer complex of the quotients $(M\red G;L_0/G,L_1/G)$ by two morphisms:
\ea
K &\colon CF_G(M;L_0,L_1)\to CF(M\red G;L_0/G,L_1/G)\\
K' &\colon CF(M\red G;L_0/G,L_1/G) \to CF_G(M;L_0,L_1)
\ea
\end{theoremABC}

Finally, we apply our construction to \MW's setting. The following statement corresponds to  Theorem~\ref{th:equiv_HSI}:

\begin{theoremABC}\label{th:equiv_HSI_intro} 
The construction outlined in Section~\ref{ssec:outline_constr} can also be implemented in \MW's setting, i.e. when $M = \N(\Sigma')$ is the open subset of the extended moduli space involved in \cite{MW}. We call the corresponding group $HSI_G(Y)$ equivariant symplectic instanton homology. It is a relatively $\Z{8}$-graded $H^*(BSU(2))$-module,  and in the case when $Y$ is a rational homology sphere, an absolute $\Z{8}$-grading can be fixed canonically.
\end{theoremABC}

We believe $HSI_G(Y)$  is independent of the Heegaard splitting of $Y$, and outline a proof of this in Remark~\ref{rem:indep_Heegaard_splitting}. Furthermore, it provides an equivariant symplectic side for the Atiyah-Floer conjecture: we believe it should be isomorphic to a suitably defined $SU(2)$-equivariant version of instanton homology  \cite{Miller_equiv}.

\subsection{Organization of the paper}\label{ssec:orga} 

In Section~\ref{sec:telescopes}, we start by setting the algebraic framework of telescopes that we will be using in our constructions.

In Section~\ref{sec:equiv_Morse}, we provide more details about the Morse theoretical construction of equivariant homology outlined above.

In Section~\ref{sec:equiv_Floer}, after reviewing some standard material about Hamiltonian actions and Lagrangian Floer homology, we set our working assumptions and construct $CF_G(M,L_0, L_1)$.

In Section~\ref{sec:continuation_maps}, we define continuation maps to prove independence of perturbations.

In Section~\ref{sec:Floer_vs_Morse}, we compute the case $L_0= L_1$, using a PSS construction.

In Section~\ref{sec:module_str}, we define the bimodule structure on $HF_G(M,L_0, L_1)$.

In Section~\ref{sec:Kirwan_maps}, we construct the Kirwan maps of Theorem~\ref{th:Kirwan_maps_intro}.

Finally, in Section~\ref{sec:Equivariant_HSI} we focus on \MW's setting and define equivariant symplectic instanton homology of a 3-manifold.

\begin{acknow}I would like to thank Ciprian Manolescu and Chris Woodward for suggesting this project many years ago, as well as Frédéric Bourgeois, Paolo Ghiggini, Dominic Joyce, Paul Kirk, Artem Kotelskiy, Ciprian Manolescu, Eckhard Meinrenken, Mike Miller Eismeier,  Alex Oancea, James Pascaleff, Alex Ritter, Matt Stoffregen, Claude Viterbo, Chris Woodward, Guangbo Xu and Wai-Kit Yeung for helpful conversations. I also thank Julio Sampietro for pointing out a minor mistake in the definition of $\Lambda_N$ in a previous version. I would also like to thank the anonymous referee, whose careful comments contributed to improving the exposition.
\end{acknow}

\section{Telescopes}
\label{sec:telescopes}

As in symplectic homology \cite{Viterbo,BourgeoisOancea} or Wrapped Floer homology \cite{AbouzaidSeidel}, we will define the equivariant chain complex as a homotopy colimit. The telescope construction is a nice model for this homotopy colimit. We first review its construction at the level of spaces, as it sheds light on the chain level construction that we will be using to define the equivariant Morse and Floer complexes.

\subsection{Spaces}\label{ssec:spaces}
\begin{defi}[Telescopes]\label{def:Telescopes_topol}
Let $(X_N, a_N)_{N\geq 0}$ be a sequence of spaces, with \emph{increment} maps 
\[
a_N\colon X_N \to X_{N+1}.
\]
The \emph{mapping telescope} of this sequence is defined by
\[
\Tel(X_N, a_N) = \left(\mathbin{\coprod}_{N} X_N \times I_t  \right) /(x,1)\sim (a_N(x), 0)
\]
 with $I_t = [0,1]$. The subscript $t$ is here to indicate the variable we will use for elements in $I_t$.

This telescope construction is endowed with a shift map 
\[
a\colon \Tel(X_N, a_N) \to \Tel(X_N, a_N)
\] 
defined by $a(x,t) = (a_N(x),t)$. The shift is a homotopy equivalence, as it is connected to the identity through the path of maps $(a^u)_{0\leq u\leq 1}$ defined by:
\[
a^u(x,t)= \begin{cases}(x, t+u)&\text{ if }t\leq 1-u ,\\
(a_N(x), t+u-1) &\text{ if } t\geq 1-u.
\end{cases}
\]
\end{defi}
\begin{defi}[Maps between telescopes]\label{def:Maps_between_Telescopes}
Let $(X_N, a_N)_{N\geq 0}$ and $(Y_N, b_N)_{N\geq 0}$ be such sequences of spaces. Consider a sequence of maps 
\[
f_N\colon X_N \to Y_N
\]
that commute with the increments $a_N, b_N$ up to homotopy, in the sense that there exist  maps $k_N \colon X_N \times I_t \to Y_{N+1}$ such that
\begin{align*}
k_N(\cdot, 0) &= b_N f_N \\
k_N(\cdot, 1) &=  f_{N+1} a_N .
\end{align*}
In this setting one can define a map
\[
\Tel(f_N,k_N) \colon \Tel(X_N)\to \Tel(Y_N)
\] 
by setting
\e 
\Tel(f_N, k_N) (x,t) = \begin{cases}
(f_N(x), 2t) &\text{ if } 0\leq t \leq \frac{1}{2},\\
(k_N(x, 2t-1),0) &\text{ if } \frac{1}{2}\leq t \leq 1.
\end{cases}\label{eq:map_on_telescopes}
\e
\end{defi}
\begin{proof}[Proof that the map is well-defined]
These two quantities agree when $t=\frac{1}{2}$, since 
\[
(f_N(x), 1) =  (b_N f_N(x), 1) = (k_N(x, 0),0).
\]
Furthermore, $\Tel(f_N, k_N)$ maps both $(x,1)$ and $(a_N(x),0)$ to the same element $(f_{N+1} a_N(x),0)$.

\end{proof}
\begin{remark} One could have used other formulas here, for example by setting
\[
\Tel(f_N, k_N) (x,t) = (k_N(x,t),t).
\]
but in view of the next section, the one we use has the advantage of being a cellular map.
\end{remark}

\begin{defi}[Homotopies between telescopes]\label{def:Homotopies_between_Telescopes}
Consider now two sequences of maps
\[
f_N^0, f_N^1 \colon X_N \to Y_N
\]
that are homotopic and that  commute with the increments $a_N, b_N$ up to homotopy: we are given two maps
\begin{align*}
f_N &\colon X_N\times I_u \to Y_N,\\
k_N &\colon X_N \times I_t\times I_u \to Y_{N+1},
\end{align*}
satisfying:
\begin{align*}
 f_N(\cdot, 0) &= f_N^0 ,\\
f_N(\cdot, 1) &= f_N^1 ,\\
k_N(x,0,u) &= b_N f_N(x, u), \\
k_N(x,1,u) &= f_{N+1}(a_N(x), u)  .
\end{align*}
Notice that this  is the same setting  as in the previous paragraph, replacing $X_N$ by $Z_N = X_N \times I_u$. Therefore we get a map between telescopes:
\[
\Tel(f_N) \colon \Tel(Z_N)\to \Tel(Y_N)
\]
and since $\Tel(Z_N) =  \Tel(X_N) \times I_u$, one can think of $\Tel(f_N)$ as a homotopy between  
\[
\Tel(f_N^0), \Tel(f_N^1) \colon \Tel(X_N)\to \Tel(Y_N).
\]
\end{defi}

Later on, when there is no ambiguity, we will occasionally drop indexing subscripts for readability.

\begin{prop}[Products]\label{prop:product_telescopes_spaces}
Let $(X_m, a_m)$ and $(Y_n, b_n)$ be two sequences of spaces, then the map
\e
\Phi \colon \Tel(X_m, a_m) \times \Tel(Y_n, b_n) \to \Tel(X_m\times Y_m, a_m\times b_m)
\e
defined, with $(x, u) \in X_m \times [0,1)$ and $(y,v) \in Y_n \times [0,1)$ and denoting 
\ea 
a^k &= a\circ \cdots \circ a \colon X_m\to X_{m+k}, \\
b^k &= b\circ \cdots \circ b \colon Y_n\to Y_{n+k}, 
\ea 
by 
\e
\Phi ((x,u), (y,v)) = \begin{cases} ((x, y), \max(u,v) ) & \text{if }m=n \\
 ((a^{n-m} x, y), v ) & \text{if }m<n \\
 (( x, b^{m-n}  y), u) & \text{if }m>n \\
\end{cases}
\e
is a homotopy equivalence.
\end{prop} 
\begin{proof} Left to the reader (see Figure~\ref{fig:product_telescopes}).
\end{proof}

\begin{figure}[!h]
    \centering
    \def\svgwidth{.50\textwidth}
    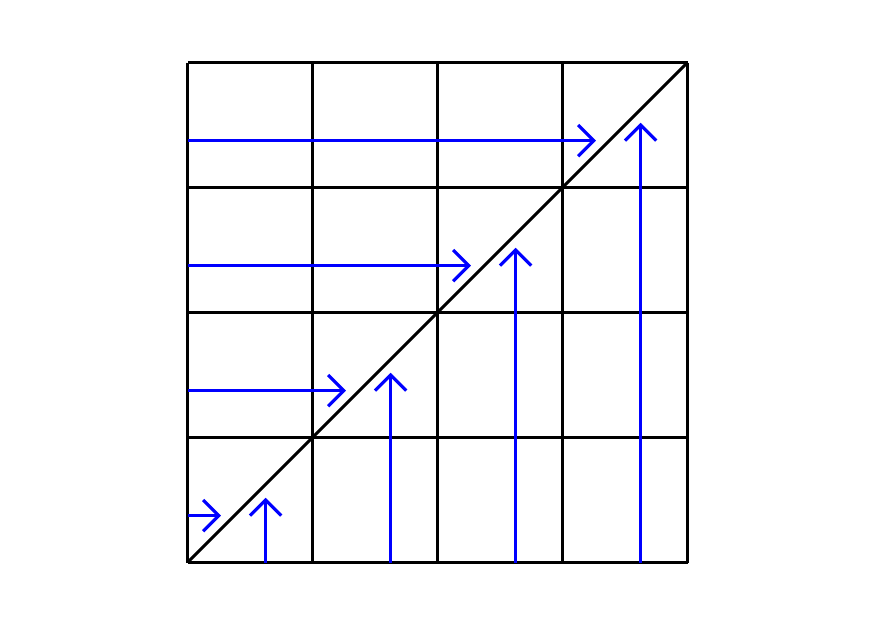
      \caption{Map from a product of telescopes to the telescope of products.}
      \label{fig:product_telescopes}
\end{figure}

\subsection{Chain complexes}\label{ssec:chain_cpx}
This subsection is the algebraic transcription of the previous one.

\begin{defi}[Telescopes of chain complexes]\label{def:telescope_chain_cpx}
Let $(C_N, \alpha_N)_{N\geq 0}$ be a sequence of chain complexes over $\Z{2}$, with \emph{increment} chain morphisms 
\[
\alpha_N\colon C_N \to C_{N+1}.
\]
One can think of $C_N$ as the cellular complex of $X_N$, equipped with some cellular decomposition.
The \emph{telescope} of this sequence is defined by
\[
\Tel(C_N, \alpha_N) = \bigoplus_{N} C_N \oplus q C_N
\]
 with $q$ a formal variable of degree one. If $\sigma$ is a cell of $X_N$ the $C_N$ summand in $\Tel(C_N, \alpha_N)$ corresponds to the cells $\sigma \times \lbrace 0\rbrace$, while the $q C_N$ summand corresponds to the cells $\sigma \times [ 0,1]$. It is endowed with the differential 
\[
\delta(a+ qb) = da + q db + \alpha_N b -b 
\]
Sometimes we may just denote it $\Tel(C_N)$ when there is no ambiguity.
\end{defi}

\begin{defi}[Morphisms between telescopes]\label{def:Morphisms_between_Telescopes}
Let $(C_N, \alpha_N)_{N\geq 0}$ and $(D_N, \beta_N)_{N\geq 0}$ be two sequences as before, consider a sequence of chain morphisms 
\[
\varphi_N\colon C_N \to D_N
\]
that commute with the increment $\alpha_N, \beta_N$ up to a homotopy $\kappa_N \colon C_N  \to D_{N+1}$: 
\[
 \varphi_{N+1} \alpha_N  - \beta_N \varphi_N = d \kappa_N + \kappa_N d
\]

\[
\xymatrix{\cdots \ar[r] & C_N \ar[d]^{\varphi_N} \ar[r]^{\alpha_N} \ar@{-->}[dr]^{\kappa_N} & C_{N+1} \ar[d]^{\varphi_{N+1}}\ar[r]^{} & \cdots \\
\cdots \ar[r] & D_N  \ar[r]^{\beta_N}  & D_{N+1} \ar[r]^{} & \cdots \\
  } .
\]

In this setting one can define a map $\Tel(\varphi_N, \kappa_N) \colon \Tel(C_N)\to \Tel(D_N)$ by
\[
\Tel(\varphi_N, \kappa_N ) (a+qb) = \varphi_N(a)+ \kappa_N(b) +q \varphi_N(b).
\]
\end{defi}
\begin{remark}If $a$ represents a cell in $X_N$, then $\Tel(f_N, k_N)$ will map it to $(f_N\circ a, 0)$. If $b$ represents a cell in $X_N$, then  $q b$ corresponds to the cell $b\times I_t$ in $\Tel(X_N)$ which by (\ref{eq:map_on_telescopes}) is mapped to the union of $(f_N \circ b)\times I_t$ and $k_N(b\times I_t) \times \left\lbrace 0 \right\rbrace$.
\end{remark}

One has $\delta \Tel(\varphi_N, \kappa_N)= \Tel(\varphi_N, \kappa_N) \delta$:
\begin{align*}
\delta \Tel(\varphi_N, \kappa_N) (a+qb)&= \delta (  \varphi_N(a)+ \kappa_N(b) +q \varphi_N(b)) \\
&= d \varphi a + \left[ d \kappa b  + \beta \varphi b\right]  + q d\varphi b - \varphi b \\
&=  \varphi d a + \left[  \kappa d b  +  \varphi \alpha b\right]  + q \varphi d b - \varphi b \\
&=  \Tel(\varphi_N, \kappa_N) \delta (a+qb) .
\end{align*}

\begin{prop}[Compositions of telescopes]\label{prop:compo_telescopes}
Let $(\varphi_N, \kappa_N)$ be a sequence of chain maps and homotopies as above between the sequences of chain complexes $(C_N, \alpha_N)$ and $(D_N, \beta_N)$. Likewise, let $(\chi_N, \lambda_N)$ be a sequence of chain maps and homotopies going from $(D_N, \beta_N)$ to $(E_N, \gamma_N)$. Then $\chi_N \varphi_N$ commutes with $\alpha_N$ and $\gamma_N$ up to the homotopy $\chi_{N+1}\kappa_N +\lambda_N \varphi_N$.
\[
\xymatrix{\cdots \ar[r] & C_N \ar[d]^{\varphi_N} \ar[r]^{\alpha_N} \ar@{-->}[dr]^{\kappa_N}  
& C_{N+1} \ar[d]^{\varphi_{N+1}}\ar[r]^{} & \cdots \\
\cdots \ar[r] & D_N \ar[d]^{\chi_N} \ar@{-->}[dr]^{\lambda_N} \ar[r]^{\beta_N}  & D_{N+1} \ar[d]^{\chi_{N+1}} \ar[r]^{} & \cdots \\
\cdots \ar[r] & E_N  \ar[r]^{\gamma_N}  & E_{N+1} \ar[r]^{} & \cdots \\
  } .
\] 
Define their composition to be the set of maps:
\[
(\chi_N, \lambda_N) \circ (\varphi_N, \kappa_N) = ( \chi_{N+1} \varphi_N, \chi_N\kappa_N +\lambda_N \varphi_N).
\]
The operation of taking telescopes is compatible with composition as follows: 
\[
\Tel\left((\chi_N, \lambda_N) \circ (\varphi_N, \kappa_N)  \right)= \Tel(\chi_N, \lambda_N) \circ \Tel(\varphi_N, \kappa_N).
\]
\end{prop}

\begin{proof}

Commutativity up to homotopy:
\begin{align*}
 \chi \varphi \alpha - \gamma\chi \varphi &=   \chi (\beta \varphi + d \kappa + \kappa d) - (\chi \beta + d \lambda + \lambda d )\varphi \\
&= d(\chi\kappa +\lambda \varphi) + (\chi\kappa +\lambda \varphi) d .
\end{align*}

Compatibility with telescopes:
\begin{align*}
\Tel(\chi_N, \lambda_N) \circ \Tel(\varphi_N, \kappa_N)(a+qb)&= \Tel(\chi_N, \lambda_N)(\varphi a + \kappa b + q \varphi b) \\
&=  \chi\varphi a + \chi\kappa b + \lambda \varphi b + q \chi\varphi b\\
&= \Tel((\chi_N, \lambda_N) \circ (\varphi_N, \kappa_N))(a+qb).
\end{align*}
\end{proof}

\begin{defi}[Homotopies between telescopes]\label{def:Homotopies_between_Telescopes_alg}
Consider now two sequences $(\varphi_N^0, \kappa_N^0)$ and $(\varphi_N^1, \kappa_N^1)$ from $(C_N, \alpha_N)$ to $(D_N, \beta_N)$ as before (i.e. for $u=0,1$,  $\varphi_{N+1}^u \alpha_N  - \beta_N \varphi_N^u = d \kappa_N^u + \kappa_N^u d$). 

We say that $(\varphi_N^0, \kappa_N^0)$ and $(\varphi_N^1, \kappa_N^1)$ are \emph{homotopic} if there exists two sequences of maps
\begin{align*}
\phi_N&\colon C_N\to D_{N}  ,\\
\kappa_N&\colon C_N\to D_{N+1}  ,
\end{align*}
such that
\begin{align*}
\varphi^1 - \varphi^0 &= d \phi + \phi d ,\\
\kappa^1 - \kappa^0 + \phi \alpha + \beta \phi &= d \kappa + \kappa d .
\end{align*}
Let then $\Tel(\phi, \kappa)\colon \Tel(C,\alpha) \to \Tel(D,\beta)$ be defined by
\[
\Tel(\phi, \kappa)(a+qb) = \phi a + \kappa b + q\phi b .
\]
\end{defi}

\begin{prop}\label{prop:homotopy_telescopes}
The morphism $\Tel(\phi, \kappa)$ defined above gives a homotopy between $\Tel(\varphi^0, \kappa^0)$ and $\Tel(\varphi^1, \kappa^1)$, in the sense that
\[
\Tel(\varphi^0, \kappa^0) -\Tel(\varphi^1, \kappa^1) = \delta \Tel(\phi, \kappa)+ \Tel(\phi, \kappa) \delta .
\]
\end{prop}

\begin{proof}
\begin{align*}
&(\delta \Tel(\phi, \kappa)+ \Tel(\phi, \kappa) \delta)(a+qb) \\
&= \delta(\phi a + \kappa b + q\phi b) + \Tel(\phi, \kappa)(da +\alpha b + b + qdb) \\
&= \left[ d \phi a + d \kappa b + \beta \phi b + \phi b + qd\phi b \right]  + \left[ \phi d a + \phi \alpha b + \phi b + \kappa d b + q \phi db \right]  \\
&= (d\phi + \phi d )a + (\phi \alpha + \beta \phi + d \kappa + \kappa d )b + q (\phi d + d \phi)b  \\
&= (\varphi^1 - \varphi^0 )a + (\kappa^1 - \kappa^0 )b + q (\varphi^1 - \varphi^0 )b  \\
&= \left[\Tel(\varphi^0, \kappa^0) -\Tel(\varphi^1, \kappa^1)\right](a+qb).
\end{align*}
\end{proof}

Together with Proposition~\ref{prop:compo_telescopes} this implies:

\begin{cor}\label{cor:homotopy_equiv_telescopes}
Let $(\varphi, \kappa)$ and $(\chi, \lambda)$ be sequences of chain maps and homotopies going respectively from $(C,\alpha)$ to $(D, \beta)$ and from $(D, \beta)$  to  $(C,\alpha)$, and such that $(\chi, \lambda) \circ (\varphi, \kappa)$ and $(\varphi, \kappa)\circ (\chi, \lambda)$ are respectively homotopic to $(id_C,0)$ and $(id_D,0)$. Then $\Tel(\varphi, \kappa)$ and $\Tel(\chi, \lambda)$ are homotopy equivalences between  $\Tel(C,\alpha)$ and $\Tel(D, \beta)$.
\end{cor}

\begin{prop}[Products]\label{prop:product_telescopes_chain_cpx}
Let $(C_m, \alpha_m)$ and $(D_n, \beta_n)$ be two sequences of chain complexes, then the map
\e
\Phi \colon \Tel(C_m, \alpha_m) \otimes \Tel(D_n, \beta_n) \to \Tel(C_m\otimes D_m, \alpha_m\otimes \beta_m)
\e
defined, with $(x+ q x') \in C_m \oplus q C_m$ and $(y+qy') \in D_n \oplus q D_n$ and denoting 
\ea
\alpha^k = \alpha\circ \cdots \circ \alpha \colon C_m\to C_{m+k}, \\
\beta^k = \beta\circ \cdots \circ \beta \colon D_n\to D_{n+k}, \\
\ea 
by
\e
\Phi ((x+ q x')\otimes (y+ q y')) = \begin{cases} (x\otimes y) + q (x\otimes y' +x'\otimes y ) & \text{if }m=n \\
\alpha^{n-m} x \otimes y + q(\alpha^{n-m} x \otimes y')  & \text{if }m<n \\
 x  \otimes \beta^{m-n} y  + q (x' \otimes \beta^{m-n} y) & \text{if }m>n \\
\end{cases}
\e
is a morphism of chain complexes.
\end{prop} 
\begin{proof} Compute first:
\ea
&(Id\otimes \delta + \delta\otimes Id ) ((x+ q x')\otimes (y+ q y')) \nonumber\\
=& (x+ q x')\otimes (d y+ qd y' + \beta y' + y') + (d x+ q d x' + \alpha x' + x')\otimes (y+ q y') . \nonumber   
\ea
Before applying $\Phi$, notice that this is a sum of terms of degree $(m, n)$, except $(x+ q x')\otimes \beta y'$ and $\alpha x' + \otimes (y+ q y') $, which are of degree  $(m, n+1)$ and  $(m+1, n)$ respectively. We will treat the cases $m=n$ and $m<n$ separately (the case $m>n$ is analogous).

\textbf{Case $m=n$:}
\ea
&\Phi \circ (Id\otimes \delta + \delta\otimes Id ) ((x+ q x')\otimes (y+ q y')) \nonumber\\
=& x\otimes (d y  + y' ) + q [x'\otimes (d y  + y' ) + x\otimes d y'] +  \alpha x\otimes \beta y' \nonumber\\
+& (d x  + x') \otimes y + q [ (d x  + x') \otimes y' + d x' \otimes y  ] + \alpha x'\otimes \beta y\nonumber \\
=& x\otimes (d y  + y' ) + q [x'\otimes d y   + x\otimes d y'] +  \alpha x\otimes \beta y' \nonumber\\
+& (d x  + x') \otimes y + q [ d x  \otimes y' + d x' \otimes y  ] + \alpha x'\otimes \beta y, \nonumber 
\ea
which is equal to:
\ea
& \delta \Phi  ((x+ q x')\otimes (y+ q y'))\nonumber\\
=& \delta (  (x\otimes y) + q (x\otimes y' +x'\otimes y ) )\nonumber\\
=& d x\otimes  y + x\otimes  d y + q ( dx\otimes  y' + x\otimes  dy' + dx'\otimes  y + x'\otimes  dy ) \nonumber\\
+& (\alpha \beta) (x\otimes y' +x'\otimes y) + x\otimes y' +x'\otimes y  \nonumber\\
=& d x \otimes y + x \otimes d y + q ( dx\otimes  y' + x\otimes  dy' + dx'\otimes  y + x'\otimes  dy ) \nonumber\\
+&   \alpha x\otimes \beta y' + \alpha x'\otimes \beta y + x\otimes y' +x'\otimes y .\nonumber
\ea

\textbf{Case $m<n$:} Let $k= n-m$, notice that $\Phi ( \alpha x'  \otimes (y+ q y') ) = \alpha^k x'  \otimes (y+ q y')$, whether $k=1$ or $k>1$.

\ea
&\Phi \circ (Id\otimes \delta + \delta\otimes Id ) ((x+ q x')\otimes (y+ q y')) \nonumber\\
=& (\alpha^k x)\otimes (d y+ y' +qd y' )  +  (\alpha^{k+1} x)\otimes (
 \beta y')    \nonumber\\
+& (\alpha^k d x+ \alpha^k x' )\otimes (y+ q y')  + ( \alpha^k x' )\otimes (y+ q y') \nonumber\\
=& (\alpha^k x)\otimes (d y+ y' +qd y' )  +  (\alpha^{k+1} x)\otimes (
 \beta y')    \nonumber\\
+& (\alpha^k d x)\otimes (y+ q y') . \nonumber
\ea
On the other hand,

\ea
& \delta \Phi  ((x+ q x')\otimes (y+ q y'))\nonumber \\
=& \delta\left[  (\alpha^k x)\otimes (y+ q y')\right] \nonumber\\
=& d (\alpha^k x\otimes  y) + \alpha\beta (\alpha^k x\otimes  y') + \alpha^k x\otimes  y' +q d(\alpha^k x\otimes  y') \nonumber\\
=& d \alpha^k x\otimes  y +\alpha^k x \otimes d y + \alpha^{k+1} x \otimes \beta y' \nonumber\\
 +& \alpha^k x \otimes y' +q d\alpha^k x \otimes y' +q \alpha^k x\otimes  d y', \nonumber
\ea
which, since $\alpha d = d \alpha$, are equal.
\end{proof}
\begin{remark}The morphism $\Phi$ is probably a homotopy equivalence as well, but we will not need that.
\end{remark}

\section{Equivariant Morse homology}
\label{sec:equiv_Morse}

We begin with a quick review of Morse homology to set our notations, and refer for example to \cite{AudinDamian} for more details.

\subsection{Morse homology}\label{ssec:Morse_homology}
Let $X$ be a compact smooth manifold of dimension $n$, and $f\colon M\to \rr$ a Morse function. Each critical point $x$ has a Morse index $\ind(x)$. Denote respectively the set of critical points and index $k$ critical points by $\Crit(f)$ and $\Crit_k(f)$.

\begin{defi}\label{def:pseudo_grad} A \emph{pseudo-gradient} for $f$ is a vector field $v\in\mathfrak{X}(X)$ on $X$ such that for all $x\in X\setminus \Crit(f)$, $d_xf. v<0$; and such that in a Morse chart near a critical point, $v$ is the negative gradient of $f$ for the standard metric on $\rr^n$. Denote by $\mathfrak{X}(X,f)\subset\mathfrak{X}(X)$ the space of pseudo-gradients for $f$. This is a convex (hence contractible) space.
\end{defi}

\begin{defi}\label{def:stable_unstable_mfd}Let $x\in \Crit_k(f)$   and  $v\in \mathfrak{X}(X,f)$. Define the \emph{stable} (resp.  \emph{unstable}) submanifold of $x$ by:
\ea
S_x &= \left\lbrace y\in X :\lim_{t\to +\infty} \phi_t^v(y)=x \right\rbrace,\\
U_x &= \left\lbrace y\in X :\lim_{t\to -\infty} \phi_t^v(y)=x \right\rbrace .
\ea
with $\phi_t^v$ the flow at time $t$ of $v$. The subsets $S_x$ and $U_x$ are smooth (non-proper) submanifolds diffeomorphic respectively to $\rr^k$ and $\rr^{n-k}$.  
\end{defi}

\begin{defi}\label{def:Palais_Smale_cond} A pseudo-gradient $v$ is \emph{Palais-Smale} if for any pair $x,y$ of critical points, $S_x$ intersects $U_y$ transversally.
\end{defi}

\begin{defi}\label{def:Morse_complex} Assume $v\in\mathfrak{X}(X,f)$  is Palais-Smale. Define the \emph{Morse complex} 
\[
CM_*(X,f),\ \partial\colon CM_*(X,f)\to CM_{*-1}(X,f) 
\] 
by
\ea
CM_k(X,f) &= \bigoplus_{x\in \Crit_k(f)} \Z{2}\, . x\\
\partial x &= \sum_{y\in \Crit_{k-1}(f)} \#\big ( (U_x\cap S_y)/\rr \big) \, . y ,
\ea
where $\rr$ acts on $U_x\cap S_y$ by the flow of $v$. It is well-known that $\partial^2 = 0$ and that $HM_*(X,f) = H_*(X, \Z{2})$. Intuitively, $x\in  \Crit_k(f)$ corresponds to a $k$-chain obtained by triangulating $U_x$.
\end{defi}

\subsection{Pushforwards on Morse homology}
\label{ssec:pushforwards} Let $F\colon X \to Y$ be a differentiable map between two smooth compact manifolds, then $F$ induces a pushforward in homology $F_*\colon H_*(X)\to H_*(Y)$. We now recall a Morse theoretic construction of such a map, and refer for example to \cite[Section~2.8]{KMbook} for more details. 

Endow $X$ and $Y$ with two Morse functions $f\colon X\to \rr$, $g\colon Y\to \rr$, and pseudo-gradients $\nabla f, \nabla g$. Let $x\in \Crit_k(f)$ be a generator of $CM_k(X,f)$ (say $k\geq 1$ for the following discussion). Heuristically, $x$ corresponds to the $k$-chain of its unstable manifold $U_x$, therefore its image $F_*x$ should correspond to $F(U_x)$, which is a priori unrelated to $g$. Apply the flow of $\nabla g$ to it: for $t$ large enough, most points will fall down to local minimums, except those points lying in a stable manifold $S_y$ of a critical point $y$ of index $l\geq 1$. If $k=l$, then a small neighborhood of those points will concentrate to $U_y$, which now corresponds to a generator of $CM(Y,g)$.

Therefore the previous discussion motivates the following definition. Assume that $f$, $g$ and two pseudo-gradients $\nabla f, \nabla g$ are chosen so that, for any critical points $x$ and $y$ of $f$ and $g$ respectively, the graph $\Gamma(F)$ intersects $U_x\times S_y$ transversely in $X\times Y$. Define then 
\ea \nonumber
F_* &\colon CM_*(X,f)\to CM_*(Y,g)\text{ by} \\
F_* x &= \sum_{y\in \mathrm{Crit}_{\mathrm{ind}\,x}} \# \left( \Gamma (F) \cap U_x \times S_y \right) y. \label{eq:Morse_pushforwards}
\ea
This map induces the actual pushforward in homology \cite[prop.~2.8.2]{KMbook}.

In order to translate this construction to Floer theory, it is convenient to think of $\Gamma (F) \cap U_x \times S_y$ as a moduli space of \emph{grafted flow lines} from $x$ to $y$: By this we mean a pair of flow lines $(\gamma_-,\gamma_+)$, with
\begin{align*}
\gamma_- &\colon \rr_- \to M,\ \gamma_-'(t) =  \nabla f(\gamma_-(t)), \\
\gamma_+ &\colon \rr_+ \to N,\ \gamma_+'(t) =  \nabla g(\gamma_+(t))
\end{align*}
such that 
\begin{align*}
&\lim_{t\to -\infty}{\gamma_-(t)} = x, \\ 
&\lim_{t\to +\infty}{\gamma_+(t)} = y, \\
&F\left(\gamma_-(0)\right) = \gamma_+(0). 
\end{align*}
These are Morse counterparts of quilts, as we shall see. Denoting $\Mcal(x,y)$ the moduli space of such grafted lines, the identification with $\Gamma (F) \cap U_x \times S_y$ is given by:
\e
(\gamma_-,\gamma_+)\mapsto (\gamma_-(0),\gamma_+(0)).
\e

\begin{figure}[!h]
    \centering
    \def\svgwidth{.50\textwidth}
\begingroup%
  \makeatletter%
  \providecommand\color[2][]{%
    \errmessage{(Inkscape) Color is used for the text in Inkscape, but the package 'color.sty' is not loaded}%
    \renewcommand\color[2][]{}%
  }%
  \providecommand\transparent[1]{%
    \errmessage{(Inkscape) Transparency is used (non-zero) for the text in Inkscape, but the package 'transparent.sty' is not loaded}%
    \renewcommand\transparent[1]{}%
  }%
  \providecommand\rotatebox[2]{#2}%
  \newcommand*\fsize{\dimexpr\f@size pt\relax}%
  \newcommand*\lineheight[1]{\fontsize{\fsize}{#1\fsize}\selectfont}%
  \ifx\svgwidth\undefined%
    \setlength{\unitlength}{419.52755906bp}%
    \ifx\svgscale\undefined%
      \relax%
    \else%
      \setlength{\unitlength}{\unitlength * \real{\svgscale}}%
    \fi%
  \else%
    \setlength{\unitlength}{\svgwidth}%
  \fi%
  \global\let\svgwidth\undefined%
  \global\let\svgscale\undefined%
  \makeatother%
  \begin{picture}(1,0.27027027)%
    \lineheight{1}%
    \setlength\tabcolsep{0pt}%
    \put(0,0){\includegraphics[width=\unitlength,page=1]{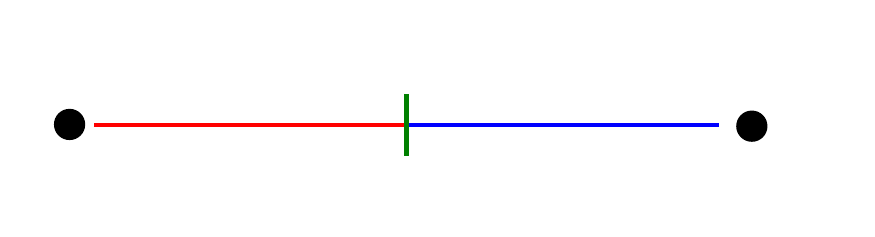}}%
    \put(0.22637739,0.02594272){\color[rgb]{1,0,0}\makebox(0,0)[lt]{\lineheight{1.25}\smash{\begin{tabular}[t]{l}$M$\end{tabular}}}}%
    \put(0.44096268,0.02197347){\color[rgb]{0,0.50196078,0}\makebox(0,0)[lt]{\lineheight{1.25}\smash{\begin{tabular}[t]{l}$F$\end{tabular}}}}%
    \put(0.63167242,0.01998869){\color[rgb]{0,0,1}\makebox(0,0)[lt]{\lineheight{1.25}\smash{\begin{tabular}[t]{l}$N$\end{tabular}}}}%
    \put(0.21052847,0.19085126){\color[rgb]{1,0,0}\makebox(0,0)[lt]{\lineheight{1.25}\smash{\begin{tabular}[t]{l}$\nabla f$\end{tabular}}}}%
    \put(0.61170247,0.19192292){\color[rgb]{0,0,1}\makebox(0,0)[lt]{\lineheight{1.25}\smash{\begin{tabular}[t]{l}$\nabla g$\end{tabular}}}}%
    \put(0,0){\includegraphics[width=\unitlength,page=2]{grafted_line.pdf}}%
  \end{picture}%
\endgroup%

      \caption{A grafted line.}
      \label{fig:grafted_line}
\end{figure}

\subsection{The equivariant Morse complex}
\label{ssec:equiv_Morse_complex}

Let $X$ be a closed $G$-manifold, and $EG_N$ finite dimensional smooth approximations of $EG$, with inclusions 
\e
i_N \colon EG_N \to EG_{N+1}.
\e 
Fix Morse functions $f_N$ on $X_N = X\times_G EG_N$. Let the increment maps be
\e
j_N = id_X\times_G i_N\colon X_N\to X_{N+1} ,
\e
and define the \emph{equivariant Morse complex} as
\e\label{eq:equiv_Morse_complex}
CM^G_*(X, \lbrace f_N \rbrace) = \Tel \Big(CM(X_N,f_N), (j_{N})_* \Big). 
\e
Let also the \emph{equivariant Morse cochain complex} $CM_G^*(X, \lbrace f_N \rbrace)$ be its dual complex. From the facts that Morse homology corresponds to singular homology, and that Morse pushforwards induce usual pushforwards, one gets:

\begin{prop}\label{prop:equiv_Morse_is_equiv_homol} Denoting $HM^G_*$ and $CM_G^*$ the homology groups associated with the above complexes, one has 
\ea
HM^G_*(X, \lbrace f_N \rbrace) &= H^G_*(X), \\
HM_G^*(X, \lbrace f_N \rbrace) &= H_G^*(X).
\ea
\end{prop}

\section{Equivariant Lagrangian Floer homology}
\label{sec:equiv_Floer}
We first recall some basic facts about Hamiltonian actions.

\subsection{Hamiltonian actions}
\label{ssec:Hamiltonian_actions}

\begin{defi}\label{def:ham_action}(Hamiltonian manifold)
Let $G$ be a Lie group. A  Hamiltonian $G$-manifold $(M,\omega,\mu)$ is a symplectic manifold  $(M,\omega)$  endowed with a  $G$-action by symplectomorphisms, induced by a moment map $\mu \colon M\to \mathfrak{g}^*$. The moment map is $G$-equivariant with respect to this action and  the coadjoint representation on $\mathfrak{g}^*$, and satisfies the following equation: 
\e
\iota_{X_\eta} \omega = d \langle \mu, \eta\rangle ,
\e
for each $\eta \in \mathfrak{g}$, where $X_\eta$ denotes the vector field on $M$ induced by the infinitesimal action, i.e. 
\e
X_\eta (m)= \frac{d}{dt}_{|t=0} (e^{t\eta} m).
\e
In other words, $X_\eta$ is the symplectic gradient of the function $\langle \mu, \eta\rangle$.

\end{defi}

\begin{exam}[Induced Hamiltonian action on a cotangent bundle]\label{exam:ham_action_cotangent_bundle}
A smooth action of $G$ on $X$ lifts to a Hamiltonian action on $T^* X$ defined by 
\e
g\cdot (q,p) = (gq, p\circ (D_q L_g)^{-1}), 
\e
with $g\in G$, $(q,p)\in T^*X$. Its moment map $\mu \colon T^*X \to \mathfrak{g}^*$ is defined by 
\e
\mu(q,p) \cdot \xi = p(X_\xi(q)). 
\e
\end{exam}

\begin{defi}[Weinstein correspondence \cite{Weinstein_sg}]\label{def:Weinstein_corresp}
The data of both the action and the moment map can be conveniently packaged as a Lagrangian submanifold 
\e
\Lambda_G(M)\subset T^*G\times M^-\times M,
\e 
that we will refer to as the \emph{Weinstein correspondence}, defined as:
\e
 \Lambda_G(M) = \lbrace ((q,p),m,m') : m' = q.m,\ R_{g^{-1}}^*p = \mu(m)  \rbrace.
  \e
When $M = T^*X$ is a cotangent bundle and the action and the moment maps are the ones canonically induced from a smooth action on the base $X$, $\Lambda_G(M)$ corresponds to the conormal bundle of the graph of the action  
\e
\Gamma_G(X) \subset G\times X\times X ,
\e
where one identifies $T^* X$ with $(T^* X)^-$ via $(q,p)\mapsto (q,-p)$. 
\end{defi}

\begin{defi}\label{def:reduction}(Symplectic quotient) 
If $(M,\omega,\mu)$ is a Hamiltonian manifold, its \emph{symplectic quotient} (or \emph{reduction}) is defined as 
\e
M\red G  = \mu^{-1}(0) / G . 
\e
When $0$ is a regular value for $\mu$, and $G$ acts freely and properly on $\mu^{-1}(0)$, $M\red G $ is also a symplectic manifold. In this case, we will say that the action is \emph{regular}. 
\end{defi}

\begin{prop}\label{prop:canon_corresp}(Canonical Lagrangian correspondence between $M$ and $M\red G$) If the action of $G$ on $M$ is regular in the sense of Definition~\ref{def:reduction}, the image of the map 
\e
\iota \times \pi  \colon \mu^{-1}(0) \to  M^- \times M\red G
\e 
is a Lagrangian correspondence, where $\iota$ and $\pi$ denote respectively the inclusion and the projection.

\end{prop}

\begin{defi}\label{def:G-lagr}($G$-Lagrangian)
A $G$-Lagrangian of a Hamiltonian $G$-manifold $M$ is a Lagrangian submanifold $L\subset M$ that is both contained in the zero level $\mu^{-1}(0)$, and  $G$-invariant. When the $G$-action is regular, the $G$-Lagrangians of $M$ are in one-to-one correspondence with the Lagrangians on $M\red G$ (though a Lagrangian in $M$ need not be a $G$-Lagrangian to induce a Lagrangian on $M\red G$). 

\end{defi}

\subsection{Symplectic homotopy quotients}
\label{ssec:Symplectic_homotopy_quotients}

Let now $M$ be a Hamiltonian $G$-manifold,  with moment map $\mu_M$, and $L_0$, $L_1$ a pair of $G$-Lagrangians. Let 
\ea
T_N &= T^*EG_N,\text{ and let} \\
0_N &\subset T_N
\ea denote its zero section. For any $N$, set 
\e
M_N = (M \times T_N) \red G,
\e
 which is a smooth symplectic manifold. Indeed, it follows from the fact that $G$ acts freely on $EG_N$ that zero is a regular value of the moment map, and that $G$ acts freely on its zero level in $M \times T_N$. 
Likewise, for $i=0,1$, 
\e
L_i^N = L_i \times_G 0_N
\e
is a smooth Lagrangian submanifold of $M_N$.

Even if $M$ is compact, $M_N$ is never compact (as soon as $EG_N\neq G$), and will not always be convex at infinity in general. The following will be important for ensuring compactness of the moduli spaces of pseudo-holomorphic curves in the monotone setting:

\begin{prop}\label{prop:fibration_homotopy_quotient}
Let $M$ be a G-Hamiltonian manifold with moment map $\mu_M$, and $L\subset M$ a $G$-Lagrangian. Let $E$ be a closed manifold on which $G$ acts freely, and $B=E/G$. 

Equip $E$ with a $G$-invariant Riemannian metric, inducing an equivariant almost complex structure $J_E$ on $T^*E$ and an almost complex structure $J_B$ on $T^* B$. Assume also that $M$ is equipped with an equivariant almost complex structure $J_M$.

Let $Q = (M\times T^*E)\red G$. Then $Q$ inherits a symplectic structure $\omega_Q$ and a compatible \acs\ $J_Q$. The quotient $Q$ fibers over $T^*B$, with fibers $M$, and this fibration restricts to a fibration of $(L\times 0_E)/G$ over $0_B$, with fibers $L$.
\e
\xymatrix{
 & M \ar@{^{(}->}^{\iota_\beta}[r]& Q \ar@{->>}^{\pi}[d]\\
L\ \  \ar@{^{(}->}[r] \ar@{^{(}->}[ru]& (L\times 0_E)/G \ar@{->>}[d] \ar@{^{(}->}[ru] & T^* B\\
 & 0_B \ar@{^{(}->}[ru]& 
}
\e
The projection $Q\to T^*B$ is $(J_Q, J_B)$-holomorphic, and 
 for each $\beta\in T^*B$, the inclusion $\iota_\beta\colon M\hookrightarrow Q$ to the fiber over $\beta$ satisfies:
\e
\iota_\beta^* \omega_Q = \omega_M .
\e
\end{prop}

\begin{proof}

For $q\in E$, let $O_q = T_q (G\cdot q) \subset E$ be the tangent space of the $G$-orbit passing through $q$. With $\mu_{E}\colon T^*E\to \g^*$ the moment map, one has 
\e
(\mu_{E})^{-1}(0) = \left\lbrace (q,p) \ |\ p_{|O_q}=0  \right\rbrace .
\e
From the orthogonal splitting
\e
T_q E =  O_q \oplus O_q^{\perp},
\e
one gets a dual splitting
\e
T_q^* E =  O_q^* \oplus (O_q^{\perp})^*,
\e
and with $(O^{\perp})^* = \bigcup_q (O_q^{\perp})^*$, one has $(\mu_{E})^{-1}(0) = (O^{\perp})^*$, and a projection 
\e 
T^*E\to (\mu_{E})^{-1}(0)
\e 
defined fiberwise by orthogonal projection to $(O_q^{\perp})^*$. Composing this projection with the quotient projection 
\e
(\mu_{T^*E})^{-1}(0) \to (T^*E)\red G \simeq T^* B,
\e
one obtains a projection $T^*E\to T^* B$.

Notice that $J_{E}$ descends to an almost complex structure $J_{B}$ on $T^*B$, and under the identification $(T^*E)\red G \simeq T^* B$ corresponds to the \acs\ induced by the Riemannian metric on $B$. With respect to these \acs s, the projection $T_N\to B_N$ is pseudo-holomorphic.

(This map corresponds to the projection to the GIT quotient $T^*E \to T^*E / G^\cc$, by the Kempf-Ness theorem and the fact that all the points of $T^*E$ are semi-stable.)

One then gets a projection 
\e
\mu_{diag}^{-1}(0) \hookrightarrow M\times T^*E  \to T^*E  \to T^*B
\e
which is $G$-invariant for the diagonal action, and therefore a projection
\e
\pi \colon Q  \to T^*B
\e
whose fiber is diffeomorphic  to $M$. Indeed, given a point $\beta = [(q_0,p_0)]\in T^*B$, the fiber $\pi^{-1}(\beta)$ can be parametrized by the map $\iota_\beta\colon M\to Q= (M\times T^*E)\red G$ defined by
\e
 m\mapsto [m, (q_0,p_0-\mu(m))],
 \e
 where we view $\mu(m)$ as an element of $O_q^*$. In order to show that $(\iota_\beta)^*\omega_Q  = \omega_M$, let $m\in M$ and $v,w\in T_m M$. One has:
\e
(\iota_\beta)^*(\omega_Q)_m (v,w) = \omega_M(v,w) + \omega_E (-d_m\mu .v,-d_m\mu .w).
\e 
And since both $-d_m\mu .v,-d_m\mu .w$ lie on the fiber direction, which is Lagrangian for $\omega_E$, the second summand vanishes.

\end{proof}

\subsection{Increment correspondences}\label{ssec:incr_corr}
Recall that the approximation of $EG$ comes with inclusions $i_N\colon EG_N \to EG_{N+1}$. These induce Lagrangian \emph{increment correspondences} $\Lambda_N\subset (M_N)^-\times M_{N+1}$ defined as follows. Inside $M^-\times T_N^- \times M \times T_{N+1}$, consider (modulo implicitly exchanging the second and third factors) 
\e
\widehat{\Lambda}_N = \Delta_M \times N_{\Gamma(i_N)},
\e
where $\Delta_M\subset M^-\times M$ is the diagonal, and $N_{\Gamma(i_N)} \subset (T_N)^-\times T_{N+1}$ is given by:
\e
N_{\Gamma(i_N)} = \left\lbrace (q,p),(q',p')\ :\ q' = i_N(q),\ p= p'\circ \d i_N  \right\rbrace.
\e
Notice that, under the identification of $T_N$ with $(T_N)^-$ via $(q,p)\mapsto (q,-p)$, $N_{\Gamma(i_N)}$ corresponds to the conormal bundle of the graph of $i_N$. The group $G\times G$ acts on $M^-\times T_N^- \times M \times T_{N+1}$ by:
\e
(g_1, g_2).(m,t,m',t') = (g_1 m,g_1 t,g_2 m',g_2 t')
\e
and with moment map $\Phi = (-\Phi_1, \Phi_2)$, with 
\ea
\Phi_1 (m,t,m',t') &= \mu_M(m) + \mu_{T_{N}}(t), \\
\Phi_2 (m,t,m',t') &= \mu_M(m') + \mu_{T_{N+1}}(t'). 
\ea
Notice that $\widehat{\Lambda}_N$ is a $G$-Lagrangian for the diagonal action of this $G\times G$-action (this follows from the equivariance of $i_N$), but not a $(G\times G)$-Lagrangian\footnote{This fact was pointed out to us by Julio Sampietro.}.

However, $\widehat{\Lambda}_N$ intersects $\Phi^{-1}(0)$ cleanly. Let $I_N =\widehat{\Lambda}_N \cap \Phi^{-1}(0)$ be this intersection. Indeed, it follows from the fact that $\widehat{\Lambda}_N$ is a $G^{diag}$-Lagrangian (in particular contained in $(-\Phi_1+ \Phi_2)^{-1}(0)$) that $I_N =\widehat{\Lambda}_N \cap \Phi_1^{-1}(0) =\widehat{\Lambda}_N \cap \Phi_2^{-1}(0)$, and one can see that these two intersections are transverse, which implies clean intersection with $\Phi^{-1}(0) = \Phi_1^{-1}(0)\cap \Phi_2^{-1}(0)$.

It follows that $\widetilde{\Lambda}_N := (G\times G) . I_N$  (which also equals $(G\times \lbrace 0 \rbrace) . I_N= (\lbrace 0 \rbrace\times G) . I_N$) is a $(G\times G)$-Lagrangian, which gives our desired increment correspondence:
\e
\Lambda_N = \widetilde{\Lambda}_N /(G\times G) \subset M_N ^- \times M_{N+1}.
\e

\subsection{Exact setting}\label{ssec:exact_setting}
The simplest setting for defining Lagrangian Floer homology is the exact and convex at infinity one, which we recall quickly:

\begin{defi}\label{def:exact_mfd_lag}
A symplectic manifold $(M,\omega)$ is \emph{exact} if $\omega = d\lambda$.

A Lagrangian $\iota\colon L\hookrightarrow M$ inside an exact symplectic manifold $(M,\lambda)$ is \emph{exact} if $\iota^* \lambda = df$.
\end{defi}

\begin{defi}\label{def:convex_at_infinity}
An exact symplectic manifold $(M,\lambda)$ is \emph{convex at infinity} if it is isomorphic to the positive symplectization of a contact manifold outside a compact subset.

A Lagrangian $L$ inside a convex symplectic manifold $(M,\lambda)$  is \emph{cylindrical at infinity} if it is of the form $\rr_+\times \Lambda$ at infinity, where $\Lambda$ is a 
Legendrian.
\end{defi}

\begin{exam}A cotangent bundle $T^*X$ with its Liouville form $\lambda = p\, dq$ is exact and convex at infinity. If $Z\subset X$ is a smooth submanifold, then its conormal bundle $N_Z \subset T^* X$ is exact ($\iota^* \lambda = 0$) and cylindrical at infinity.
\end{exam}

To ensure that the symplectic homotopy quotients will have such structures, one can impose equivariant analogues as follows.

\begin{defi}\label{def:G-exact_mfd_lag}
A $G$-Hamiltonian manifold $(M,\omega)$ is \emph{equivariantly exact} (or \emph{$G$-exact}) if $\omega = d\lambda$, with $\lambda$ a $G$-invariant 1-form.

If $G$ is connected, this is equivalent to saying that 
\e
\forall \xi \in \g,\ \mathcal{L}_{X_\xi}\lambda = 0,
\e 
i.e. the form $\iota_{X_\xi} \lambda + \left\langle \mu, \xi \right\rangle $ is closed. 

A $G$-Lagrangian $\iota\colon L\hookrightarrow M$ in a $G$-exact symplectic manifold $(M,\lambda)$ is \emph{$G$-exact} if $\iota^* \lambda = df$, with $f$ a $G$-invariant smooth function.
\end{defi}
\begin{remark}If $G$ is compact, any exact $G$-Hamiltonian manifold can be made $G$-exact: one can make the primitive $\lambda$ equivariant by averaging over $G$:
\e
\frac{1}{\mathrm{Vol}G} \int_{g\in G} \varphi_g^* \lambda ,  
\e
with $\varphi_g\colon M\to M$ the multiplication by $g$. The same is true for Lagrangians.
\end{remark}

Recall the contact analogues of Hamiltonian actions and symplectic reductions. These appeared in \cite{Albert}, see also \cite[sec.~7.7]{Geiges_book} and the references therein.

\begin{defi}\label{def:contact_Hamiltonian_actions}
Let $(X, \xi=\ker \lambda)$ be a contact manifold. A contact $G$-Hamiltonian action is an action by contactomorphisms, with a moment map $\mu\colon X\to \g^*$ such that 
\e
\iota_{X_\xi}\lambda + \langle \mu, \xi\rangle = 0,\ \forall \xi\in \g .
\e

Say that such an action is \emph{regular} if $0\in \g^*$ is a regular value of $\mu$ and $G$ acts freely on $\mu^{-1}(0)$.

When this is the case, the \emph{contact reduction} $X\red G =\mu^{-1}(0)/G$ inherits a contact form.

A Legendrian $\Lambda\subset X$ is a $G$-Legendrian if it is $G$-invariant and contained in $\mu^{-1}(0)$.
\end{defi}

\begin{prop}\label{prop:symplectization_G-mfd} Let $(X, \lambda, \mu)$ be a contact $G$-manifold. Its symplectization $SX = (\rr\times X, e^t \lambda)$ is $G$-Hamiltonian with moment map $\Phi(t,x) = e^t \mu(x)$, and is $G$-exact. Furthermore, if the $G$-action on $X$ is regular, the same is true for the $G$-action on $SX$, and $(SX)\red G$ is isomorphic to $S(X\red G)$ as an exact symplectic manifold.

If $\Lambda\subset X$ is a $G$-Legendrian, then $S\Lambda = \rr\times \Lambda\subset SX$ is a $G$-exact Lagrangian, $(S\Lambda)\red G \simeq S(\Lambda\red G)$.
\end{prop}

\begin{defi}\label{def:G-convex_at_infinity}
A $G$-exact symplectic manifold $(M,\lambda)$ is \emph{$G$-convex at infinity} if it is isomorphic (as $G$-exact Hamiltonian manifolds) to the positive symplectization of a contact $G$-manifold outside a compact subset.

A $G$-Lagrangian $L$ inside a $G$-convex symplectic manifold $(M,\lambda)$  is \emph{$G$-cylindrical at infinity} if it is of the form $\rr_+\times \Lambda$ at infinity, where $\Lambda$ is a  $G$-Legendrian.
\end{defi}

\begin{prop}\label{prop:cotangent_bundle_equi_exact}
Let $X$ be a $G$-manifold, then its cotangent bundle $TX$ is $G$-exact and $G$-convex at infinity. If $Z\subset X$ is a $G$-invariant submanifold, then its conormal bundle $N_Z \subset T^*X$ is $G$-exact. Moreover, the Liouville form restricts to zero on $N_Z$.
\end{prop}

In defining $HF^G(M;L_0, L_1)$ we will assume:
\begin{assumption}\label{ass:exact_setting} The symplectic manifold $(M,\omega =  \d \lambda)$ is $G$-exact and $G$-convex at infinity. The Lagrangians  $L_0, L_1 \subset M$ are $G$-exact, $G$-cylindrical at infinity, and disjoint outside a compact subset.
\end{assumption}

It then follows from the above discussion that for each $N$, $M_N$ is exact and convex at infinity, and that $L_0^N, L_1^N \subset M_N$ are exact, cylindrical at infinity and disjoint outside a compact subset. Furthermore, one can show that $\Lambda_N \subset (M_N)^- \times M_{N+1}$ is  exact and cylindrical at infinity as well (but we will not use this).

\subsection{Monotone setting}\label{ssec:monotone_setting}

The monotone setting is another simple case where Lagrangian Floer homology is easy to define, see for example \cite[sec.~2.1]{MW}. 

If $L\subset M$ is a Lagrangian submanifold, let $\mu_L\colon \pi_2(M,L) \to \zz$ denote the Maslov index. If $u$ is a disc in $M$ with boundary in $L$, $v$ a sphere in $M$, and $u\#v$ is the disc corresponding to the connected sum  at an interior point (see Figure~\ref{fig:internal_connected_sum}), recall that
\e\label{eq:Maslov_sphere_bubbling}
\mu_L(u\#v) = \mu_L(u) + 2 c_1(v).
\e

\begin{figure}[!h]
    \centering
    \def\svgwidth{.50\textwidth}
\begingroup%
  \makeatletter%
  \providecommand\color[2][]{%
    \errmessage{(Inkscape) Color is used for the text in Inkscape, but the package 'color.sty' is not loaded}%
    \renewcommand\color[2][]{}%
  }%
  \providecommand\transparent[1]{%
    \errmessage{(Inkscape) Transparency is used (non-zero) for the text in Inkscape, but the package 'transparent.sty' is not loaded}%
    \renewcommand\transparent[1]{}%
  }%
  \providecommand\rotatebox[2]{#2}%
  \newcommand*\fsize{\dimexpr\f@size pt\relax}%
  \newcommand*\lineheight[1]{\fontsize{\fsize}{#1\fsize}\selectfont}%
  \ifx\svgwidth\undefined%
    \setlength{\unitlength}{215.43307087bp}%
    \ifx\svgscale\undefined%
      \relax%
    \else%
      \setlength{\unitlength}{\unitlength * \real{\svgscale}}%
    \fi%
  \else%
    \setlength{\unitlength}{\svgwidth}%
  \fi%
  \global\let\svgwidth\undefined%
  \global\let\svgscale\undefined%
  \makeatother%
  \begin{picture}(1,0.90789474)%
    \lineheight{1}%
    \setlength\tabcolsep{0pt}%
    \put(0,0){\includegraphics[width=\unitlength,page=1]{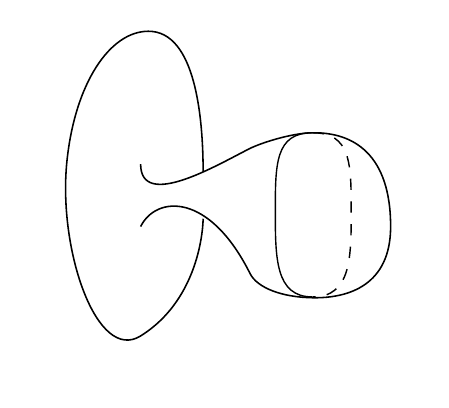}}%
    \put(0.26890061,0.59890399){\color[rgb]{0,0,0}\makebox(0,0)[lt]{\lineheight{1.25}\smash{\begin{tabular}[t]{l}$u$\end{tabular}}}}%
    \put(0.69627193,0.15940231){\color[rgb]{0,0,0}\makebox(0,0)[lt]{\lineheight{1.25}\smash{\begin{tabular}[t]{l}$v$\end{tabular}}}}%
    \put(0.05844117,0.15704054){\color[rgb]{0,0,0}\makebox(0,0)[lt]{\lineheight{1.25}\smash{\begin{tabular}[t]{l}$L$\end{tabular}}}}%
  \end{picture}%
\endgroup%

      \caption{Interior connected sum $u\#v$.}
      \label{fig:internal_connected_sum}
\end{figure}
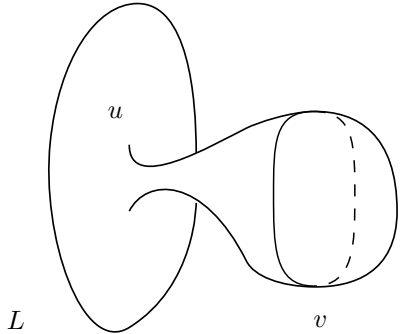

\begin{defi}\label{def:monotone}Let $\kappa>0$. A symplectic manifold $(M, \omega)$ is \emph{$\kappa$-monotone} if 
\e
[\omega]_{|\pi_2(M)} = \kappa\, c_1(TM).
\e
A Lagrangian submanifold $L\subset M$ is \emph{$\kappa$-monotone} if \e
2[\omega]_{|\pi_2(M,L)} = \kappa\, \mu_L ,
\e
where $\mu_L$ is the Maslov index.
\end{defi}

If $L\subset M$ is $\kappa$-monotone and $M$ is connected, then $M$ is $\kappa$-monotone. The converse is true if $L$ is simply connected.

This allows us to control the energy of pseudo-holomorphic curves to ensure compactness of the moduli spaces involved in the construction of Floer homology. Still, as Maslov index two discs can obstruct the differential squaring to zero, one usually makes an extra assumption on minimal Maslov numbers.

\begin{defi}\label{def:minimal_Chern_Maslov}
The \emph{minimal Chern number} $N_M$ of $M$ is the positive generator of the image of $c_1(TM)\colon \pi_2(M)\to \zz$.

The \emph{minimal Maslov number} $N_L$ of $L$ is the positive generator of the image of $\mu_L \colon \pi_2(M,L)\to \zz$.
\end{defi}

If $M$ is connected, $N_L$ divides $2 N_M$, and if $L$ is simply connected, $N_L = 2 N_M$.

Floer homology is well-defined under the following:
\begin{assumption}\label{ass:monotone_setting} The symplectic manifold $M$ is compact $\kappa$-monotone, and the Lagrangians $L_0, L_1$ are $\kappa$-monotone and such that $N_{L_0},N_{L_1}\in A \zz$, with $A\geq 3$.
\end{assumption}

Except for compactness, this setting transports to symplectic homotopy quotients, as we shall see.
 
\begin{lemma}\label{lem:Maslov_index_quotient}
Let $(M,\omega)$ be a symplectic manifold, endowed with a regular Hamiltonian action of a compact Lie group $G$ with moment $\mu\colon M\to \mathfrak{g}^* $, and $L\subset \mu^{-1}(0)$ a $G$-Lagrangian. Let $u\colon (D^2, \partial D^2) \to (M\red G, L/G)$ be a disc, and $\overline{u}\colon (D^2, \partial D^2) \to (\mu^{-1}(0), L)$ a lift of $u$ (which exists since $D^2$ is contractible). Then the Maslov indices are equal:

\e
\mu_{L/G}(u) = \mu_L (\overline{u}).
\e
\end{lemma}
\begin{proof}
Let $\g^{\cc} = \g\oplus i \g$, and $J$ a $G$-invariant almost-complex structure on $M$ compatible with $\omega$.
Since the action is regular, we get an injective map, with $Z=\mu^{-1}(0)$
\ea
Z\times \g^{\cc} &\to TM_{|Z} \\
(x, v+iw) &\mapsto X_v(x) + J_x X_w .
\ea 

Denote by $\underline{\g}^{\cc}$ the (trivial) sub-bundle of $TM_{|Z}$ corresponding to the image of this map. Notice that the sub-bundle $\underline{\g} \subset \underline{\g}^{\cc}$ defined as the image of $Z\times \g$ corresponds to the foliation of $Z$ by orbits.

Let $V$ denote the orthogonal complement of $\underline{\g}^{\cc}$ in $TM_{|Z}$ (with respect to either the symplectic structure or the Riemannian metric induced by $J$), and $W = V\cap TL$.

Let now $u$ and $\overline{u}$ be as in the statement. One has 
\e
\mu_L(\overline{u}) = \mu( \overline{u}^* (\underline{\g}^{\cc}, \underline{\g})) + \mu( \overline{u}^* (V,W)),
\e 
but $\mu( \overline{u}^* (\underline{\g}^{\cc}, \underline{\g})) =0$ as $\underline{\g}$ is a constant sub-bundle of $\underline{\g}^{\cc}$. Moreover, 
\e
\mu( \overline{u}^* (V,W))= \mu_{L/G}(u) ,
\e 
since
\ea
V&\simeq \pi^* T(M\red G)\text{, and} \\ 
W&\simeq \pi^* T(L/G),
\ea
with $\pi\colon Z\to M\red G$ the projection.
\end{proof}

Recall from \cite[Lemma~4.4]{MW} that (under some assumptions), the symplectic quotient of a $\kappa$-monotone symplectic manifold is again $\kappa$-monotone. Here is a relative version that follows from the previous lemma:

\begin{prop}\label{prop:monotone_reduction}
Let $(M,\omega)$ be a symplectic manifold, endowed with a regular Hamiltonian action of a compact Lie group $G$ with moment $\mu\colon M\to \mathfrak{g}^* $, and $L\subset \mu^{-1}(0)$ a $G$-Lagrangian. If $L$ is $\kappa$-monotone, then $L/G$ is also $\kappa$-monotone. Moreover, the minimal Maslov number $N_{L/G}$ is a multiple of $N_L$.
\end{prop}
\begin{proof}
Let $u$ be a disc in $(M\red G, L/G)$ an $\overline{u}$ a lift to $(M,L)$. The first statement follows from:
\e
\kappa \mu_{L/G}(u)= \kappa \mu_L(\overline{u}) = 2 \omega (\overline{u})= 2\omega_{M\red G}(u).
\e
The statement about minimal Maslov numbers follows from the fact that any disc can be lifted, and from Lemma~\ref{lem:Maslov_index_quotient}.
\end{proof}

\begin{lemma}\label{lem:Diagonal_momo_mini_Maslov}
If $M$ is $\kappa$-monotone, then the diagonal $\Delta_M\subset M^-\times M$ is $\kappa$-monotone. Furthermore, $N_{\Delta_M} = 2N_M$.
\end{lemma}
\begin{proof}
A disc 
\e
u= (u_1, u_2)\colon (D^2, \partial D^2)\to (M^-\times M,\Delta_M)
\e 
gives rise to a sphere 
\e
v= u_1  \cup_{\partial D^2} u_2 \colon S^2 = (D^2)^- \cup_{\partial D^2}D^2 \to M,
\e
and conversely any sphere in $M$ gives a disc in $(M^-\times M,\Delta_M)$. The claim follows from $\mu_{\Delta_M}(u) = 2 c_1(v)$. Indeed, $u$ is homotopic to $\alpha\#\beta$ (see Figure~\ref{fig:Maslov_vs_Chern}), with 
\ea
\alpha = (u_1, u_1)&\colon D^2 \to\Delta_M\text{, and} \\
\beta = (u_1, u_1)\cup_{\partial D^2} (u_1, u_2)&\colon (D^2)^- \cup_{\partial D^2}D^2 \to M.
\ea
Then  (\ref{eq:Maslov_sphere_bubbling}) gives $\mu_{\Delta_M}(u) =\mu_{\Delta_M}(\alpha) +2c_1(\beta)$, but $\mu_{\Delta_M}(\alpha)=0$ as $\alpha\subset \Delta_M$, and $\beta= (u_1\cup_{\partial D^2} u_1) \times u$ so  $c_1(\beta) = c_1(u_1\cup_{\partial D^2} u_1) + c_1(u)$, but as $u_1\cup_{\partial D^2} u_1$ is homotopic to a constant map, $c_1(u_1\cup_{\partial D^2} u_1)=0$.

\begin{figure}[!h]
    \centering
    \def\svgwidth{.50\textwidth}
\begingroup%
  \makeatletter%
  \providecommand\color[2][]{%
    \errmessage{(Inkscape) Color is used for the text in Inkscape, but the package 'color.sty' is not loaded}%
    \renewcommand\color[2][]{}%
  }%
  \providecommand\transparent[1]{%
    \errmessage{(Inkscape) Transparency is used (non-zero) for the text in Inkscape, but the package 'transparent.sty' is not loaded}%
    \renewcommand\transparent[1]{}%
  }%
  \providecommand\rotatebox[2]{#2}%
  \newcommand*\fsize{\dimexpr\f@size pt\relax}%
  \newcommand*\lineheight[1]{\fontsize{\fsize}{#1\fsize}\selectfont}%
  \ifx\svgwidth\undefined%
    \setlength{\unitlength}{419.52755906bp}%
    \ifx\svgscale\undefined%
      \relax%
    \else%
      \setlength{\unitlength}{\unitlength * \real{\svgscale}}%
    \fi%
  \else%
    \setlength{\unitlength}{\svgwidth}%
  \fi%
  \global\let\svgwidth\undefined%
  \global\let\svgscale\undefined%
  \makeatother%
  \begin{picture}(1,0.70945946)%
    \lineheight{1}%
    \setlength\tabcolsep{0pt}%
    \put(0,0){\includegraphics[width=\unitlength,page=1]{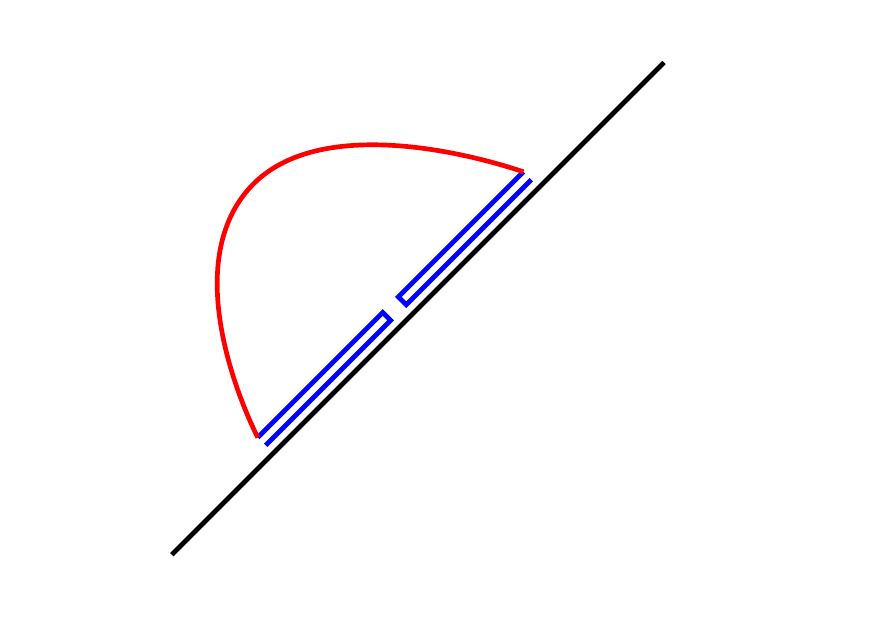}}%
    \put(0.50631496,0.3381099){\color[rgb]{0,0,1}\makebox(0,0)[lt]{\lineheight{1.25}\smash{\begin{tabular}[t]{l}$(u_1, u_1)$\\\end{tabular}}}}%
    \put(0.19636036,0.55684159){\color[rgb]{1,0,0}\makebox(0,0)[lt]{\lineheight{1.25}\smash{\begin{tabular}[t]{l}$(u_1, u_2)$\\\end{tabular}}}}%
    \put(0.75978322,0.55750282){\color[rgb]{0,0,0}\makebox(0,0)[lt]{\lineheight{1.25}\smash{\begin{tabular}[t]{l}$\Delta_M$\end{tabular}}}}%
  \end{picture}%
\endgroup%

      \caption{$u$ is homotopic to $\alpha \# \beta$.}
      \label{fig:Maslov_vs_Chern}
\end{figure}
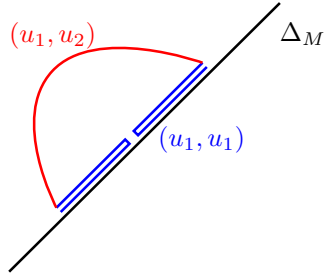

\end{proof}

\begin{lemma}\label{lem:Maslov_conormal}
Let $Z\subset X$ with $\pi_2(X,Z)=0$, then any disc 
\e
u\colon (D^2, \partial D^2)\to (T^*X, N_Z)
\e 
is homotopic to a constant disc and therefore has Maslov index zero.
\end{lemma}
\arnaque

Then let $(M;L_0, L_1)$ be a $G$-Hamiltonian manifold with a pair of $G$-Lagrangians satisfying Assumption~\ref{ass:monotone_setting}. Assume that $N$ is large enough, so that $EG_N$ and $EG_{N+1}$ are highly connected.

By Lemma~\ref{lem:Maslov_conormal} it follows that $L_i \times 0_N\subset M\times T_N$ is $\kappa$-monotone and $N_{L_i \times 0_N} = N_{L_i }$. By Lemmas~\ref{lem:Maslov_index_quotient} and Proposition~\ref{prop:monotone_reduction}, $L_i^N$ is monotone and $N_{L_i^N}\in A\zz$. 

\subsection{Construction}\label{ssec:construction}

Fix first a $G$-invariant \acs\ on $M$ and a $G$-invariant Riemannian metric on each $EG_N$, so to get  \acs s $J_{M_N}, J_{B_N}$ on $M_N$ and $B_N= (T_N)\red G$ respectively, as in Section~\ref{ssec:Symplectic_homotopy_quotients}.

Let $\Jcal(M_N)$ be the space of compatible \acs s on $M_N$ agreeing with our fixed $J_{M_N}$ outside the preimage by $\pi_N$ of a compact subset of $B_N$ containing $0_{B_N}$, and 
\e
\Jcal_t(M_N) = C^\infty ([0,1],\Jcal(M_N)).
\e

Let $\Hcal(M_N)$ be the space of compactly supported smooth functions on $M_N$, and  $\Hcal_t(M_N) = C^\infty ([0,1],\Hcal(M_N))$. If $H_t \in \Hcal_t(M_N)$, let $X_{H_t}$ denote its symplectic gradient, and $\phi_{H_t}^t$ its flow.

\begin{defi}[Perturbed (generalized) intersection points]\label{def:perturbed_gints}
Let $\Ical_{H_t}(L_0^N, L_1^N)$ be the set of $H_t$-perturbed intersection points, i.e. Hamiltonian chords $\gamma\colon [0,1]\to M_N$ of $X_{H_t}$  with $\gamma(0)\in L_0^N$ and $\gamma(1)\in L_1^N$. These are in one to one correspondence with $\phi_{H_t}^1(L_0^N) \cap L_1^N$.

More generally, if $\Lund$ is a cyclic \emph{\glag}\ from $M_0$ to itself, i.e. a sequence of Lagrangian correspondences
\e
\underline{L}  = \xymatrix{ M_0 \ar[r]^{L_{0}} & M_1 \ar[r]^{L_{1}} & \cdots \ar[r]^{L_{k-1}}& M_k = M_0,}
\e
$\delta_1, \ldots , \delta_{k-1}$ are positive integers, and $\Hund = (H^{1}_t, \ldots, H^{k-1}_t)$ a sequence of Hamiltonians in $M_1, \ldots , M_{k-1}$ respectively, let $\Ical_{\Hund}(\Lund)$ be the set of \emph{$\Hund$-perturbed \gint s}, i.e. Hamiltonian chords $\gamma_i\colon [0,\delta_i]\to M_i$ such that for all $i=1, \ldots, k-2$, with $\gamma_k := \gamma_1$,
\e
(\gamma_i(\delta_i), \gamma_{i+1}(0))\in L_i,
\e
with $L_i' = (id\times \phi_{H^i_t}^{\delta_i})(L_i)$. The map $(\gamma_1, \ldots , \gamma_{k-1})\mapsto (\gamma_1(\delta_1), \ldots , \gamma_{k-1}(\delta_{k-1}))$ gives a bijection between $\Ical_{\Hund}(\Lund)$ and the intersection of
\ea
\left(\prod \Lund'\right) _0 :=&L_{0}'\times L_2' \times \cdots\text{, and} \\
\left(\prod \Lund'\right) _1 :=&L_{1}'\times L_3' \times \cdots\text{ in} \\
\prod \Mund  :=& M_0^- \times M_1 \times \cdots \times M_{k-1}^\pm.
\ea
\end{defi}

\begin{defi}\label{def:Floer_datum}
As in \cite{Seidel_book}, call a pair $(H_t, J_t)\in \Hcal_t(M_N)\times \Jcal_t(M_N)$ such that $\phi_{H_t}^1(L_0^N)$ and $L_1^N$ intersect transversely a \emph{Floer datum}.

More generally, in the situation of a cyclic {\glag}\ $\Lund$ as above, a \emph{quilted Floer datum} is a pair of sequences
\e
(\Hund_t, \Jund_t)\in \Hcal_t(\Mund)\times \Jcal_t(\Mund):= \prod_{i=1}^{k-1}\Hcal_t(M_i)\times \Jcal_t(M_i),
\e
for which $\Ical_{\Hund}(\Lund)$ is cut out transversely, i.e. $\left(\prod \Lund'\right) _0$ and $\left(\prod \Lund'\right) _1$ intersect transversely in $\prod \Mund$.
\end{defi}

Let $Z=\{s+it: 0\leq t \leq 1\}\subset \cc$ be the strip, and $\partial_0 Z= \lbrace t=0\rbrace$, $\partial_1 Z= \lbrace t=1\rbrace$ its two boundaries. For $J_t\in \Jcal_t(M_N)$, $H_t \in \Hcal_t(M_N)$ and   $x, y \in \Ical(L_0, L_1; H_t)$, let $\widetilde{\Mcal}(x, y; H_t, J_t)$ be the moduli space of perturbed $J_t$-holomorphic strips 
\e
u\colon Z\to M
\e
satisfying the \emph{Floer equation}
\e\label{eq:Floer_eq}
\partial_s u + J_t (\partial_t u - X_{H_t})=0,
\e
the Larangian boundary conditions $u_{|\partial_0 Z} \subset L_0^N$, $u_{|\partial_1 Z} \subset L_1^N$, and asymptotic to $x$ and $y$ when $s\to-\infty$ and $s\to+\infty$ respectively. Let then $\Mcal(x, y; H_t, J_t)$ be its quotient by $\rr$ (modulo translations in the $s$-direction).

For $i\in \zz$, let  $\widetilde{\Mcal}(x, y; H_t, J_t)_i$ and $\Mcal(x, y; H_t, J_t)_i$ denote the subsets of curves with Maslov index $I(u)=i+1$. 
\begin{prop}\label{prop:def_CF_N}
Assume that $(M;L_0,L_1)$ either satisfy the exact (\ref{ass:exact_setting}) or monotone (\ref{ass:monotone_setting}) assumptions. There exists a comeagre subset
\e
\Hcal\Jcal_t^{reg}(M_N)\subset \Hcal_t(M_N) \times \Jcal_t(M_N)
\e
of regular peturbations such that, for $\Fcal = (H_t, J_t)\in \Hcal\Jcal_t^{reg}(M_N)$, $\phi_{H_t}^1(L_0^N)$ intersects $L_1^N$ transversely; $\widetilde{\Mcal}(x, y; H_t, J_t)_i$ and $\Mcal(x, y; H_t, J_t)_i$ are smooth of dimension $i+1$ and i respectively. When $i=0$,  $\Mcal(x, y; H_t, J_t)_0$ is a finite set. In this case, define
\e
CF(M_N ; L_0^N,L_1^N; \Fcal)= \bigoplus_{x\in \Ical(L_0, L_1; H_t)} \Z{2}\, x, 
\e
with differential $\partial$ defined by 
\e
\partial x = \sum_{y\in \Ical(L_0, L_1; H_t)} \# \Mcal(x, y; H_t, J_t)_0 \,y
\e

Then $\partial^2=0$,  therefore one can define $HF(M_N ; L_0^N,L_1^N; \Fcal)$ as the homology group of this chain complex.

\end{prop}

\begin{proof}
The transversality statement is a standard argument \cite{FloerHoferSalamon}. In the monotone case, let $K\subset M_N$ be a compact subset containing $L_0^N,L_1^N$, and such that $H_t = 0$ and $J_t=J_{M_N}$ outside $K$. Consider a strip $u$ in $\widetilde{\Mcal}(x, y; H_t, J_t)$, and assume by contradiction that its image is not contained in $K$. Then outside $K$, composing with the projection $M_N\to B_N$, which is pseudo-holomorphic by Proposition~\ref{prop:fibration_homotopy_quotient}, one gets a pseudo-holomorphic curve in $B_N$ that should have no maximum, which is a contradiction.

Therefore curves are contained in $K$, and the monotonicity assumptions ensure  energy bounds. Therefore Gromov compactness applies to these moduli spaces, and the rest of the argument is standard.
\end{proof}

\paragraph{\textbf{The increment map $\alpha_N\colon CF_N\to CF_{N+1}$}}

For each $N$ fix a Floer datum $\Fcal_N$ for $(M_N;L_0^N,L_1^N )$, and let 
\e
CF_N = CF(M_N;L_0^N,L_1^N;\Fcal_N ).
\e

We quickly recall some notions from Wehrheim-Woodward's theory to introduce some notations, and refer to \cite{WWquilts} for the precise definitions.

\begin{defi} A \emph{quilted surface}  $\underline{S}$ consists of:
\begin{itemize}
\item A collection $\Pcal (\underline{S})$ of \emph{patches}: these are Riemann surfaces with boundary and strip-like ends.
\item A collection $\Scal (\underline{S})$ of \emph{seams}: these are pairwise disjoint identifications of boundary components of patches, satisfying a local real-analycity condition.
\item A collection $\Bcal (\underline{S})$ of \emph{true boundaries}: these are the boundary components of patches not belonging to any seam.
\end{itemize}

If $\underline{S}$ is a quilted surface, let its total space
\e
\abs{\underline{S}} = \left(\coprod_{P\in \Pcal (\underline{S})} P\right) /\sim
\e
be the Riemann surface obtained by gleing together all the patches along the seams. Its boundary is given by the union of components of $\Bcal (\underline{S})$, and the seams become real analytic curves in $\abs{\underline{S}}$. We will often define a quilted surface by giving its total space and seams.

\end{defi}

\begin{exam}Let $\delta_1,\ldots,\delta_k$ be positive numbers,

(Quilted half-strip) Let $\underline{Z}_{\pm}(\delta_1,\ldots,\delta_k)$ be the quilted surface whose total space is $\rr_{\pm}\times [0, \delta_1+\cdots+\delta_k]$ and seams the horizontal lines 
\e
\rr_{\pm}\times \{ \delta_1\}, \rr_{\pm}\times \{ \delta_1 + \delta_2\}, \ldots , \rr_{\pm}\times \{ \delta_1 + \cdots +\delta_{k-1}\}.
\e
Its patches consist of half-strips 
\e\label{eq:patches_P_i}
P_i = \rr_{\pm}\times [ \delta_1 + \cdots +\delta_{i-1}, \delta_1 + \cdots +\delta_{i}].
\e

(Quilted half-cylinder) Let $\underline{C}_{\pm}(\delta_1,\ldots,\delta_k)$ be the quilted surface whose total space is $\rr_{\pm} \times \big( \rr/( \delta_1 + \cdots +\delta_{k})\zz \big)$ and seams the horizontal lines 
\e
\rr_{\pm}\times \{ \delta_1\}, \rr_{\pm}\times \{ \delta_1 + \delta_2\}, \ldots , \rr_{\pm}\times \{ \delta_1 + \cdots +\delta_{k}\}.
\e
It has the same patches as $\underline{Z}_{\pm}(\delta_1,\ldots,\delta_k)$ (except that $P_1$ and $P_k$ are seamed together).

\end{exam}

\begin{defi}
A \emph{quilted strip-like end} on $\underline{S}$ is a quilted map
\e
\underline{\epsilon}\colon  \underline{Z}_{\pm}(\delta_1,\ldots,\delta_k) \to \underline{S},
\e

i.e. a collection of maps $\epsilon_1, \ldots , \epsilon_k$
\e
\epsilon_i \colon P_i \to P_i'
\e
where $P_i$ is as in (\ref{eq:patches_P_i}) and $P_i'\in \Pcal(\underline{S})$, compatible with the seams (i.e. it induces a continuous map on total spaces), pseudo-holomorphic, proper on the total spaces, and mapping true boundaries to true boundaries.

Likewise, a \emph{quilted cylindrical end} on $\underline{S}$ is a quilted map
\e
\underline{\epsilon}\colon  \underline{C}_{\pm}(\delta_1,\ldots,\delta_k) \to \underline{S}
\e
compatible with the seams, pseudo-holomorphic, and proper on total spaces.

A \emph{quilted surface with strip-like and cylindrical ends} is a quilted surface $\Sund$ together with a collection of quilted strip-like and cylindrical ends such that the complement of the images of the ends in the total space is compact. Furthermore, each end is labelled either as an \emph{incoming, outgoing}, or a \emph{free} end. We denote by 
\e
\Ecal(\Sund) = \Ecal_{\rm in}(\Sund) \cup\Ecal_{\rm out}(\Sund) \cup\Ecal_{\rm free}(\Sund).
\e
the set of quilted ends of $\Sund$.
\end{defi}

\begin{defi} A \emph{decoration} $(\underline{M},\underline{L})$ of a quilted surface $\underline{S}$ consists of 
\begin{itemize}
\item a symplectic manifold $M_P$ for each patch $P\in \Pcal(\underline{S})$,
\item a Lagrangian correspondence $L_\sigma \subset (M_P)^-\times M_{P'}$ for each seam $\sigma\in \Scal (\underline{S})$ between two patches $P,P'$,
\item a Lagrangian submanifold $L_b \subset M_P$  for each true boundary of $P$.
\end{itemize}

\end{defi}

\begin{defi}Let $(\underline{S},\underline{M},\underline{L})$ be a decorated quilted surface. A (topological) \emph{quilt} 
\e
\underline{u} \colon \underline{S} \to (\underline{M},\underline{L})
\e
consists of maps $u_P\colon P\to M_P$ for each patch $P$ of $\underline{S}$ satisfying the \emph{seam condition}:
\e
(u_P(x), u_{P'}(x)) \in L_\sigma ,
\e
for $x\in \sigma$, and $\sigma$ a seam between $P$ and $P'$, and the \emph{boundary condition}:
\e
u_P(x) \in L_b ,
\e
for $x\in b$, and $b$ a true boundary of $P$.
\end{defi}

To define the moduli space of quilts that will be involved in $\alpha_N$ we use  perturbations as in \cite[sec.~(8e)]{Seidel_book}, adapted to the quilted setting.

\begin{defi}[Perturbation datum]\label{def:perturbation_datum}
Let $(\underline{S},\underline{M},\underline{L})$ be a decorated quilted surface with strip like ends.  Let 
 \e
 \Kcal(\underline{S},\underline{M},\underline{L}) \subset \Omega^1(\underline{S}; \Hcal(\underline{M})) := \prod_{P\in \Pcal(\underline{S})} \Omega^1(P; \Hcal({M}_{P})),
 \e
consist of quilted 1-forms $\underline{K} = (K_P)_P$ such that 
\begin{itemize}
\item on a boundary $b$ of a patch $P$ (belonging or not belonging to a seam),
\e
K_P(\xi)_{|L_b}=0 \text{, }\forall\xi\in Tb.
\e
\item on a (quilted) incoming or outgoing end (or at least far enough in the end) $K_P$ is only $t$-dependent (i.e. independent of $s$).
\item on a free end, $\Kund$ vanishes.

\item the following transversality condition is satisfied on a quilted incoming or outgoing end:
suppose a quilted end is decorated by the cyclic \glag :
\e
\underline{L}  = \xymatrix{ M_0 \ar[r]^{L_{0}} & M_1 \ar[r]^{L_{1}} & \cdots \ar[r]^{L_{k-1}}& M_k = M_0,}
\e
with $M_0$ a point if we are considering a strip-like end. By the two conditions above, on the end, $\Kund$ is of the form $K_{P_i} = H_i(t)\d t $ on the patch $P_i$ decorated by $M_i$.

Then, we want $\left(\prod \Lund'\right) _0$ and $\left(\prod \Lund'\right) _1$ as in Definition~\ref{def:perturbed_gints} to intersect transversely in $\prod \Mund$.
\end{itemize}
 
Let $\Jcal(P, M_P)= C^\infty(P, \Jcal(M_P))$ be the space  of domain-dependent \acs s, and 
\e
\Jcal(\Sund, \underline{M}, \underline{L}) \subset \prod_{P\in\Pcal(\underline{S})} \Jcal(P, M_P)
\e
be those $\underline{J} = (J_P)_P$ that are only $t$-dependent on quilted ends. 

The set of \emph{perturbation data} is denoted 
\e
 \Kcal\Jcal(\underline{S},\underline{M},\underline{L}) := \Kcal(\underline{S},\underline{M},\underline{L}) \times  \Jcal(\underline{S},\underline{M},\underline{L}) .
\e
Notice that at an incoming (resp. outgoing) end, a perturbation datum $\Pcal$ is asymptotic to a Floer datum $\Fcal_{in}$ (resp. $\Fcal_{out}$). In this situation we will write $\Pcal\colon \Fcal_{in}\to \Fcal_{out}$. Denote by 
\e
 \Kcal\Jcal(\underline{S},\underline{M},\underline{L}, \Fcal_{in}, \Fcal_{out}) \subset  \Kcal\Jcal(\underline{S},\underline{M},\underline{L})
\e
the subset of such perturbation datum.

\end{defi}

\begin{remark}\label{rem:no-1-forms} In most moduli spaces we will be considering, it would be enough to consider 1-forms of the form $H_z \d t$. Except in Section~\ref{sec:module_str}, where these would not satisfy the boundary assumptions, when the boundary is not horizontal.
\end{remark}

\begin{remark}\label{rem:no_pert_free_ends} We don't perturb on the free ends; the reason is that we will need our actual Lagrangians there to rule out strip breaking. Our \glag s at the free ends  will generally not intersect transversely, but only cleanly in the sense of P\'ozniak \cite{Pozniak}.
\end{remark}

 Let $\underline{Z}$ be the quilted surface that will be involved in defining the map $\alpha_N$: its total space is $Z\setminus \{0, i\}$, with a vertical seam at $s=0$. We view it as a quilted surface with one incoming end at $s\to -\infty$, one outgoing end at $s\to +\infty$, and two free quilted strip-like ends near $0$ and $i$. 
The quilted surface $\underline{Z}$ has two patches 
 \ea
 Z_- &= \{s+it: 0\leq t \leq 1,\ s\leq 0 \}\text{, and} \\
 Z_+ &= \{s+it: 0\leq t \leq 1,\ s\geq 0 \},
 \ea
one seam $\sigma = \{ s=0\}$ and four true boundary components $\partial_i Z_\pm =\{ \pm s \geq 0,\ t = i \}$, $i=0,1$. 

Let $(\underline{M}_N, \underline{L}_N)$ be the following decoration of $\underline{Z}$: $Z_-$ and $Z_+$ are decorated respectively by $M_N$ and $M_{N+1}$, $\sigma$ by $\Lambda_N$,  $\partial_i Z_-$ by $L_i^N$  and  $\partial_i Z_+$ by $L_i^{N+1}$, see Figure~\ref{fig:map_from_N_to_N+1}.

\begin{figure}[!h]
    \centering
    \def\svgwidth{.50\textwidth}
\begingroup%
  \makeatletter%
  \providecommand\color[2][]{%
    \errmessage{(Inkscape) Color is used for the text in Inkscape, but the package 'color.sty' is not loaded}%
    \renewcommand\color[2][]{}%
  }%
  \providecommand\transparent[1]{%
    \errmessage{(Inkscape) Transparency is used (non-zero) for the text in Inkscape, but the package 'transparent.sty' is not loaded}%
    \renewcommand\transparent[1]{}%
  }%
  \providecommand\rotatebox[2]{#2}%
  \newcommand*\fsize{\dimexpr\f@size pt\relax}%
  \newcommand*\lineheight[1]{\fontsize{\fsize}{#1\fsize}\selectfont}%
  \ifx\svgwidth\undefined%
    \setlength{\unitlength}{473.38582677bp}%
    \ifx\svgscale\undefined%
      \relax%
    \else%
      \setlength{\unitlength}{\unitlength * \real{\svgscale}}%
    \fi%
  \else%
    \setlength{\unitlength}{\svgwidth}%
  \fi%
  \global\let\svgwidth\undefined%
  \global\let\svgscale\undefined%
  \makeatother%
  \begin{picture}(1,0.62874251)%
    \lineheight{1}%
    \setlength\tabcolsep{0pt}%
    \put(0,0){\includegraphics[width=\unitlength,page=1]{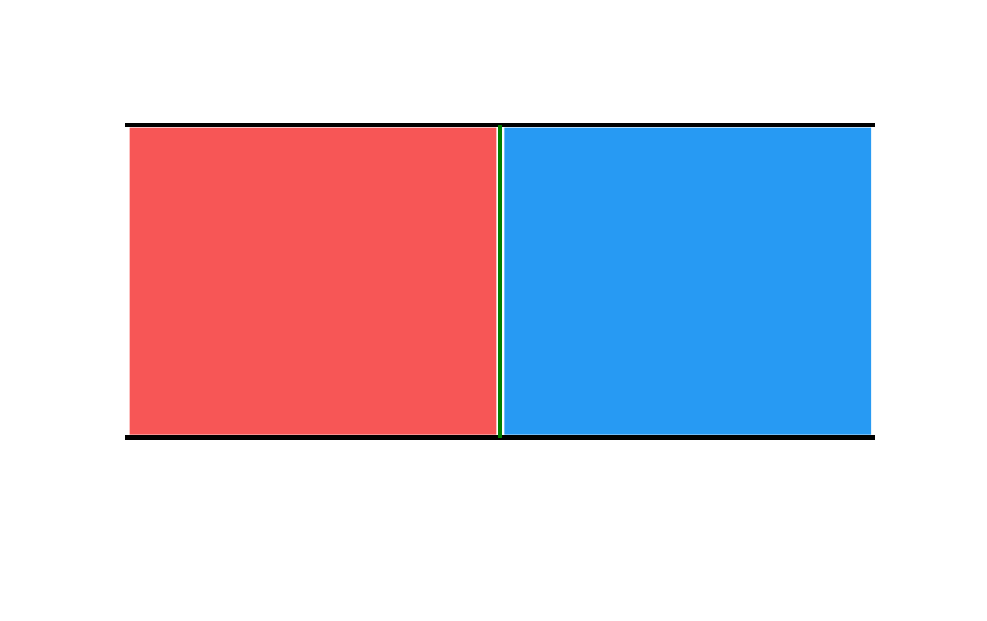}}%
    \put(0.26661809,0.33530147){\color[rgb]{0,0,0}\makebox(0,0)[lt]{\lineheight{1.25}\smash{\begin{tabular}[t]{l}$M_N$\end{tabular}}}}%
    \put(0.62148931,0.3422993){\color[rgb]{0,0,0}\makebox(0,0)[lt]{\lineheight{1.25}\smash{\begin{tabular}[t]{l}$M_{N+1}$\end{tabular}}}}%
    \put(0.61858019,0.53436215){\color[rgb]{0,0,0}\makebox(0,0)[lt]{\lineheight{1.25}\smash{\begin{tabular}[t]{l}$L_1^{N+1}$\end{tabular}}}}%
    \put(0.61731666,0.11028812){\color[rgb]{0,0,0}\makebox(0,0)[lt]{\lineheight{1.25}\smash{\begin{tabular}[t]{l}$L_0^{N+1}$\end{tabular}}}}%
    \put(0.26389099,0.11407915){\color[rgb]{0,0,0}\makebox(0,0)[lt]{\lineheight{1.25}\smash{\begin{tabular}[t]{l}$L_0^{N}$\end{tabular}}}}%
    \put(0.26235571,0.53368264){\color[rgb]{0,0,0}\makebox(0,0)[lt]{\lineheight{1.25}\smash{\begin{tabular}[t]{l}$L_1^{N}$\end{tabular}}}}%
    \put(0.46521307,0.10691977){\color[rgb]{0,0.50196078,0}\makebox(0,0)[lt]{\lineheight{1.25}\smash{\begin{tabular}[t]{l}$\Lambda_N$\end{tabular}}}}%
  \end{picture}%
\endgroup%

      \caption{The quilted surface $\underline{Z}$ and its decoration $(\underline{M}_N, \underline{L}_N)$.}
      \label{fig:map_from_N_to_N+1}
\end{figure}
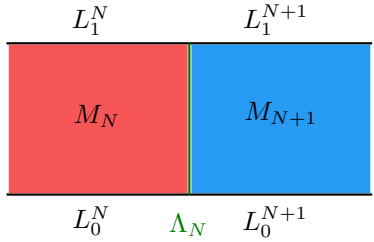

For an input $x\in \Ical (L_0^N, L_1^N, \Fcal_{N})$, an output $y\in \Ical (L_0^{N+1}, L_1^{N+1}, \Fcal_{N+1})$ and a perturbation datum 
\e
\Pcal_N = (\Kund, \Jund )\in \Kcal\Jcal(\underline{Z},\underline{M}_N,\underline{L}_N; \Fcal_{N},  \Fcal_{N+1}) ,
\e  
let the moduli space of perturbed quilted maps $\Lcal(x,y; \Pcal_N ) $ consist of quilts 
\e
\uund\colon \Zund\to (\Mund_N, \Lund_N)
\e 
satisfying the perturbed Cauchy-Riemann equation on a patch $P$: 
 \e
 (\d u_P - K_P)^{0,1} = 0,
 \e
with limits $x$ (resp. $y$) at $s\to -\infty$ (resp. $s\to +\infty$), and with unprescribed limits at the free quilted strip-like ends (or equivalently of finite energy). The superscript $0,1$ refers to the anti-holomorphic part, with respect to the \acs\ $J_P$.

Let $\Lcal(x,y; \Pcal_N )_i \subset \Lcal(x,y; \Pcal_N )$ denote the subspace of quilts of index $i$, corresponding to the index of the relevant Fredholm section $\overline{\partial}_{\Pcal_N}$ cutting out $\Lcal(x,y; \Pcal_N )$. Since we do not perturb at the free end, which is in clean intersection, one should use weighted Sobolev spaces there, as in \cite[sec.~2.3]{LekiliLipyanskiy}: we refer the reader to this for details. 

\begin{prop}\label{prop:Lcal_moduli_space}
 There exists a comeagre subset 
\e
 \Kcal\Jcal(\underline{Z},\underline{M}_N,\underline{L}_N; \Fcal_{N},  \Fcal_{N+1})^{\rm reg}\subset  \Kcal\Jcal(\underline{Z},\underline{M}_N,\underline{L}_N; \Fcal_{N},  \Fcal_{N+1})
\e
of regular perturbation data $\Pcal_N$ for which $\Lcal(x,y; \Pcal_N )_i$ is smooth and of dimension $i$. The zero-dimensional part $\Lcal(x,y; \Pcal_N )_0 $ is compact and can be used to define a chain map
\[
\alpha_N\colon CF_N \to CF_{N+1}.
\]
and the compactification $\overline{\Lcal}(x,y; \Pcal_N )_1 $ of its one-dimensional part $\Lcal(x,y; \Pcal_N )_1 $ can be used to show that it is actually a chain map.

Therefore, one can apply the telescope construction of Section~\ref{ssec:chain_cpx} and define
\e
CF^G(M; L_0, L_1; \lbrace  \Fcal_{N}, \Pcal_{N} \rbrace_N) = \Tel(CF_N, \alpha_N) .
\e
\end{prop}

\begin{proof}
The first part of the statement (smoothness and expected dimension) follows from standard transversality arguments \cite{FloerHoferSalamon}. Compactness of $\Lcal(x,y; \Pcal_{N} )_0 $ follows from Gromov compactness.

For the compactification $\overline{\Lcal}(x,y; \Pcal_{N} )_1 $, a priori the degenerations that one could observe  are (see Figure~\ref{fig:bubbling_f_N}):
\begin{enumerate}
\item sphere bubbling in the interior of the patches, or disc bubbling at the true boundary component,
\item quilted sphere bubbling at the vertical seam,
\item strip breaking at the free ends,
\item strip breaking at the incoming or outgoing end.
\end{enumerate}

\begin{figure}[!h]
    \centering
    \def\svgwidth{.50\textwidth}
\begingroup%
  \makeatletter%
  \providecommand\color[2][]{%
    \errmessage{(Inkscape) Color is used for the text in Inkscape, but the package 'color.sty' is not loaded}%
    \renewcommand\color[2][]{}%
  }%
  \providecommand\transparent[1]{%
    \errmessage{(Inkscape) Transparency is used (non-zero) for the text in Inkscape, but the package 'transparent.sty' is not loaded}%
    \renewcommand\transparent[1]{}%
  }%
  \providecommand\rotatebox[2]{#2}%
  \newcommand*\fsize{\dimexpr\f@size pt\relax}%
  \newcommand*\lineheight[1]{\fontsize{\fsize}{#1\fsize}\selectfont}%
  \ifx\svgwidth\undefined%
    \setlength{\unitlength}{419.52755906bp}%
    \ifx\svgscale\undefined%
      \relax%
    \else%
      \setlength{\unitlength}{\unitlength * \real{\svgscale}}%
    \fi%
  \else%
    \setlength{\unitlength}{\svgwidth}%
  \fi%
  \global\let\svgwidth\undefined%
  \global\let\svgscale\undefined%
  \makeatother%
  \begin{picture}(1,0.70945946)%
    \lineheight{1}%
    \setlength\tabcolsep{0pt}%
    \put(0,0){\includegraphics[width=\unitlength,page=1]{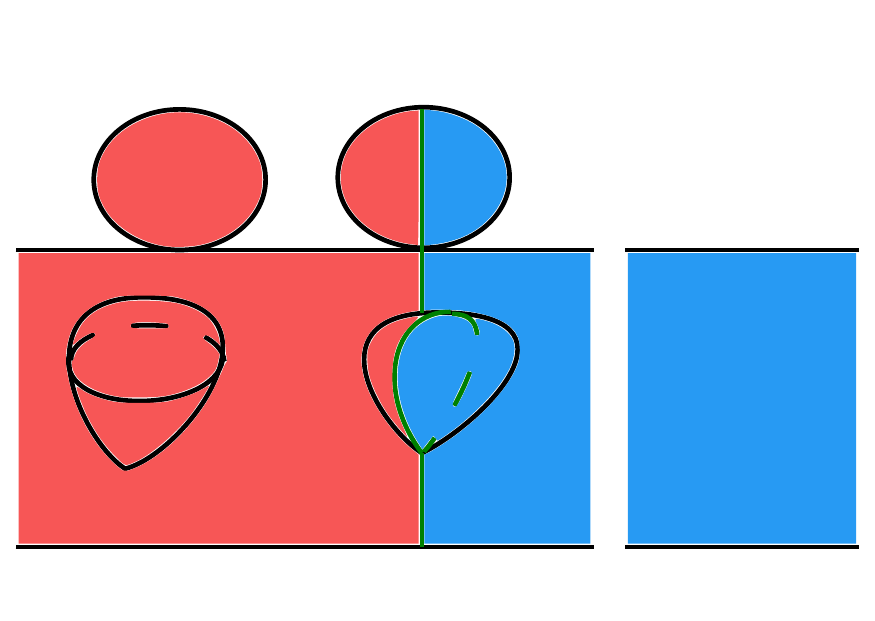}}%
    \put(0.10578638,0.61370367){\color[rgb]{0,0,0}\makebox(0,0)[lt]{\lineheight{1.25}\smash{\begin{tabular}[t]{l}$(1)$\end{tabular}}}}%
    \put(0.18516984,0.15475328){\color[rgb]{0,0,0}\makebox(0,0)[lt]{\lineheight{1.25}\smash{\begin{tabular}[t]{l}$(1)$\end{tabular}}}}%
    \put(0.41299224,0.61825765){\color[rgb]{0,0,0}\makebox(0,0)[lt]{\lineheight{1.25}\smash{\begin{tabular}[t]{l}$(3)$\end{tabular}}}}%
    \put(0.51160928,0.15936193){\color[rgb]{0,0,0}\makebox(0,0)[lt]{\lineheight{1.25}\smash{\begin{tabular}[t]{l}$(2)$\end{tabular}}}}%
    \put(0.76724471,0.4528084){\color[rgb]{0,0,0}\makebox(0,0)[lt]{\lineheight{1.25}\smash{\begin{tabular}[t]{l}$(4)$\end{tabular}}}}%
  \end{picture}%
\endgroup%

      \caption{A priori bubling in $\overline{\Lcal}(x,y; \Pcal_{N} )$.}
      \label{fig:bubbling_f_N}
\end{figure}

We show that the three first cases are ruled out by our assumptions. It will follow that the fourth case is the only one that can actually happen, and therefore
\e
\partial \overline{\Lcal}(x,y; \Pcal_{N} )_1 = \coprod_z \Mcal(x,z; \Fcal_{N})\times \Lcal(z,y; \Pcal_{N}) \cup \coprod_z \Lcal(x,z;\Pcal_{N})\times \Mcal(z,y; \Fcal_{N+1}),
\e
from which the identity $\partial \alpha_N + \alpha_N \partial = 0$ follows.

The first kind of degenerations are clearly ruled out by our assumptions, for the same reasons as they are ruled out from the moduli space of the differential.

We will rule out (2) as a special case of (3). Suppose that one has a strip breaking at the free end: by folding it one can view it as a strip
\e
u\colon (Z; \partial_0 Z, \partial_1 Z)\to (M_N^- \times M_{N+1}; \Lambda_N, L_i^{N}\times L_i^{N+1}).
\e

Lift it to a strip
\ea
\tilde{u} &= (u_M^1, u_N, u_M^2, u_{N+1})\colon \\  (Z; \partial_0 Z, \partial_1 Z) &\to ( M^-\times T_N^-\times M\times T_{N+1} ; I_N, L_i\times 0_N \times L_i \times 0_{N+1} ),
\ea
with limits $x= (x_M, x_N, x_M, x_{N+1})$ and $y= (y_M, y_N, y_M, y_{N+1})$ at the ends. This is always possible: first lift the $\partial_0 Z$ boundary to $I_N$ (using the fibration $I_N \to \Lambda_N$), then extend this lift to the strip $Z$, and the other boundary $\partial_1 Z$ will automatically be mapped to $L_i\times  0_N \times L_i \times 0_{N+1}$. 
 Notice that $\Delta_M \times N_{\Gamma(\iota_N)}$ and $L_i\times  0_N \times L_i \times 0_{N+1}$ intersect cleanly along $\Delta_{L_i} \times \Gamma(\iota_N)$.

Now we claim we can find another strip 
\e
\tilde{d} = (d_M^1, d_N, d_M^2, d_{N+1})\colon Z\to M^-\times T_N^-\times M\times T_{N+1}
\e 
that coincides with $\tilde{u}$ on $\partial_0 Z$, is entirely contained in $\Delta_M \times N_{\Gamma(\iota_N)}$, and such that $\partial_1 \tilde{d}$ is in the intersection $\Delta_{L_i} \times \Gamma(\iota_N)$. Indeed, take $d_M^1= d_M^2 = u_M^1$, and for  $d_N, d_{N+1}$ homotope $\partial_0 u_N, \partial_0 u_{N+1}$ to the zero section:
\ea
d_N (s,t) &=  (1-t) u_N (s,0) ,  \\
d_{N+1} (s,t) &= (1-t) u_{N+1} (s,0) .
\ea

Now,  glue $\tilde{d}$ and $\tilde{u}$ along their $\partial_0 Z$ boundary to get a disc $\tilde{u}_{cap}$ with boundary in $L_i\times 0_N \times L_i  \times 0_{N+1}$. By construction, $u$, $\tilde{u}$ and $\tilde{u}_{cap}$ have same symplectic area and Maslov index.

In the monotone setting, $u$ has index greater that $A$, which would force the principal component to live in a moduli space of negative dimension, empty for transversality reasons.

In the exact setting, $\tilde{u}_{cap}$ must have zero area, which contradicts $u$ being nonconstant.

Assume now we have a quilted sphere bubble of type (2): we can fold it to a disc 
\e
u\colon (D^2, \partial D^2) \to (M_N^- \times M_{N+1}, \Lambda_N).
\e
Pick a path in $\Lambda_N$ that connects a boundary point of $u$ to a point in $L_i^{N}\times L_i^{N+1}$. Such a path exists: for example one can follow the seam of the principal component of the quilt the bubble is attached to. This path allows one to view $u$ as a strip 
\e
(Z; \partial_0 Z, \partial_1 Z)\to (M_N^- \times M_{N+1}; \Lambda_N, L_i^{N}\times L_i^{N+1})
\e
of the kind (3) with same area and index, which therefore can be ruled out for the same reasons.

\end{proof}

\begin{remark}\label{rem:} If $G$ is not simply connected, lifting discs of  type (2) to discs with boundary in $\widehat{\Lambda}_N$ (and not $\widetilde{\Lambda}_N$) might not be possible, and for this reason $\Lambda_N$ might not be monotone. Fortunately, the above trick using $L_i^{N}\times L_i^{N+1}$ allows us to overcome this problem. 
\end{remark}

\begin{remark}\label{rem:comparison_KLZ} In \cite{KimLauZheng}, the authors define equivariant self-Floer homology of a Lagrangian using similar symplectic homotopy quotients: they consider $M\times_G \mu_{T_N}^{-1}(0)$, which fibers over $B_N$ with fibers $M$ and therefore admits a symplectic structure, by Thurston's theorem. We believe this space is equivalent to ours, i.e. the symplectic structure can be chosen so that it is symplectomorphic to $M_N$. However, the increment maps are constructed differently (they don't use quilts), and it is unclear to us whether these give equivalent constructions.
\end{remark}

\subsection{Gradings}
\label{ssec:gradings}

Depending on the setting, the Floer complex may admit some grading.  In the case when Lagrangians are oriented, comparing the orientation of the direct sum $T_x L_0 \oplus T_x L_1 $ with that of $T_x L$ at a transverse intersection point $x$ provides an absolute $\Z{2}$ grading on $CF(M;L_0, L_1)$.

If  $M, L_0$ and $L_1$ are simply connected (and hence orientable, but not necessarily oriented), then $CF(M;L_0, L_1)$ can be endowed with a \emph{relative} grading over $\Z{A}$, with $A\in \zz$ as in Assumption~\ref{ass:monotone_setting}, or $A= +\infty$ in the exact setting. Relative means that for $x,y$ generators of $CF(M;L_0, L_1)$, one has a number $I(x,y)\in \Z{A}$ corresponding to the difference in degrees.  This will be the case in Section~\ref{sec:Equivariant_HSI} (\MW's setting).

In the above two cases, it is clear that $M_N, L_0^N, L_1^N$ will satisfy similar assumptions, and hence $CF(M_N; L_0^N, L_1^N)$ and their telescope will inherit the same kind of absolute or relative gradings.

It is usually convenient to endow the Lagrangians with the structure of a \emph{grading} \cite{Kontsevich_ICM,Seidel_graded}, in order to get a refined absolute grading on the Floer complex:

\begin{defi}[\cite{Seidel_graded}]\label{def:gradings}For $n\geq 2$, or $n= +\infty$, an $n$-fold Maslov covering of $M$ is a $\Z{n}$-covering $\Lcal^n \to \Lcal$ of the Lagrangian Grassmannian bundle $\Lcal \to M$.

If $L\subset M$ is a Lagrangian, its tangent subspaces defines a section $s_L \colon L \to \Lcal$ of $\Lcal_{|L} \to L$.

Assume $M$ is given an $n$-fold Maslov covering $\Lcal^n$. A $\Lcal^n$-\emph{grading} on $L$ is a lift $\tilde{L}\colon L\to \Lcal^n$ of $s_L$.
\end{defi}

If in addition $(M, \omega, \mu)$ is $G$-Hamiltonian and $L\subset M$ is a $G$-Lagrangian, one can define equivariant analogues of Maslov coverings and gradings:

\begin{defi}\label{def:G-gradings} Assume $M$ is $G$-Hamiltonian and $L\subset M$ is a $G$-Lagrangian. Let $\Lcal_G \subset \Lcal_{|\mu^{-1}(0)} \to \mu^{-1}(0)$ consist of $G$-Lagrangian subspaces. An $n$-fold Maslov $G$-covering of $M$ is a $\Z{n}$-covering $\Lcal^n_G \to \Lcal_G$ of $\Lcal_G  \to \mu^{-1}(0)$, with a lift to $\Lcal^n_G$ of the $G$-action on $\Lcal_G$.

If $L\subset M$ is a $G$-Lagrangian, then $s_L$ as defined previously is a section of ${\Lcal_G}_{|L}$.  A $(\Lcal_G^n, G)$-\emph{grading} on $L$ is a $G$-equivariant lift $\tilde{L}\colon L\to \Lcal^n_G$ of $s_L$.
\end{defi}

This definition is motivated by the following immediate result:
\begin{prop}\label{prop:G-gradings_vs_quotient}Assume that the action of $G$ on $M$ is regular. Then the datum of a $n$-fold Maslov $G$-covering $\Lcal_G^n$ of $M$ is equivalent to a $n$-fold Maslov covering $\Lcal^n$ of $M\red G$. If such a datum is given, a $(\Lcal_G^n, G)$-\emph{grading} on $L$ is equivalent to a $\Lcal^n$-\emph{grading} on $L/G$.
\end{prop}

\begin{exam}\label{exam:grading}
Let $M = T^*X$ be a cotangent bundle. The bundle $\Lcal \to M$ admits a ``section by Maslov cycles'' $\mathfrak{m}$: for $m\in M$ let 
\ea
\mathfrak{m}_m &= \{l\in \Lcal_m \ : \ l\text{ is not transverse to the fiber at }m \} ,\\
\mathfrak{m} &= \bigcup_{m\in M} \mathfrak{m}_m ,\\
\Lcal^{\pitchfork} &= \Lcal\setminus \mathfrak{m} .
\ea

We claim that the existence of $\mathfrak{m}$ ensures $\Lcal$ admits a Maslov $\zz$-cover $c\colon \widehat{\Lcal} \to \Lcal$. Indeed, recall that $\Lcal_m$ can be identified with $U(n)/O(n)$, in such a way that Maslov cycles correspond to fibers of the map $det^2\colon U(n)/O(n) \to U(1)$. Having $\mathfrak{m}$ globally defined allows one to extend $det^2$ to a map $d\colon \Lcal \to U(1)$, by requiring $\mathfrak{m} = d^{-1}(1)$, say. Then the pullback $\widehat{\Lcal} = \exp^* \Lcal$, with $\exp\colon \rr \to U(1)$, will be a Maslov $\zz$-cover of $\Lcal$.

Fix a component $\widehat{\Lcal}^{\pitchfork}_0$ of $\widehat{\Lcal}^{\pitchfork} = c^{-1}(\Lcal^{\pitchfork})$. 
If $L\subset M$ is a Lagrangian transverse to the fibers at every point, then the section $s_L$ takes its values in $\Lcal^{\pitchfork}$, and admits a unique lift contained in $\widehat{\Lcal}^{\pitchfork}_0$, which is defined to be its canonical grading.

If now $X$ is acted on by $G$ (inducing a Hamiltonian action on $M$), then $\widehat{\Lcal}_G = c^{-1} (\Lcal_G)$ is a Maslov $G$-cover as defined previously. With
\ea
\Lcal^{\pitchfork}_G &= \Lcal_G\setminus \mathfrak{m} , \\
\widehat{\Lcal}^{\pitchfork}_G &= c^{-1} (\Lcal^{\pitchfork}_G ) ,
\ea
if $L\subset M$ is a $G$-Lagrangian transverse to the fibers, then its canonical grading takes its values in $\widehat{\Lcal}^{\pitchfork}_G$ and defines a $G$-grading.
\end{exam}

If $M$  has a Maslov cover, and $L_0, L_1\subset M$ are $G$-graded, by Example~\ref{exam:grading} $L_0\times 0_N, L_1\times 0_N\subset M\times T_N$ are $G$-graded, and by Proposition~\ref{prop:G-gradings_vs_quotient} $L_0^N, L_1^N \subset M_N$ are graded. Therefore:
\begin{prop}\label{prop:abs_grading_CF_G} If $L_0, L_1\subset M$ are $G$-graded in addition of satisfying Assumptions~\ref{ass:exact_setting},~\ref{ass:monotone_setting}, then $CF_G(M;L_0, L_1)$ is absolutely $\Z{n}$-graded.
\end{prop}
\section{Continuation maps}
\label{sec:continuation_maps}

In this section we aim to apply Corollary~\ref{cor:homotopy_equiv_telescopes} to show that the homotopy type of the equivariant Floer complex is independent of the choice of Floer and perturbation data.

For each $N$, with $i=0, 1$, let $\Fcal^i_N$ be a regular Floer datum for $(M_N; L_0^N, L_1^N)$, and
\ea
C_N &= CF(M_N; L_0^N, L_1^N; \Fcal^0_N) ,\\
D_N &= CF(M_N; L_0^N, L_1^N; \Fcal^1_N) .
\ea
Let also $\Pcal^i_N$ be regular perturbation datum for $(\Zund; \Mund_N, \Lund_N)$, defining increment morphisms:
\ea
\alpha_N\colon C_N\to C_{N+1}\text{, with }i=0, \\
\beta_N\colon D_N\to D_{N+1}\text{, with }i=1.
\ea

We are going to introduce perturbation data, as summarized in the following diagram:
\e\label{diag:pert_data_continuation_maps}
\xymatrix{ \Fcal^0_{N} \ar[d]^{\Qcal_{N}} \ar[r]^{\Pcal^0_{N}} \ar[dr]^{\Rcal^L_{N}} & \Fcal^0_{N+1} \ar[d]^{\Qcal_{N+1}} \\
\Fcal^1_{N}  \ar[r]^{\Pcal^1_{N}}  & \Fcal^1_{N+1}  \\ } .
\e
These will be used to define maps:
\e
\xymatrix{\cdots \ar[r] & C_N \ar[d]^{\varphi_N} \ar[r]^{\alpha_N} \ar@{-->}[dr]^{\kappa_N} & C_{N+1} \ar[d]^{\varphi_{N+1}}\ar[r]^{} & \cdots \\
\cdots \ar[r] & D_N  \ar[r]^{\beta_N}  & D_{N+1} \ar[r]^{} & \cdots \\
  } .
\e

Let $\Qcal_{N}$ be a perturbation datum on $(Z; M_N; L_0^N, L_1^N)$ going from $\Fcal^0_{N}$ to $\Fcal^1_{N}$. Use it to define a moduli space $\Ccal(x, y; \Qcal_{N})$ of $\Qcal_{N}$-perturbed pseudo-holomorphic maps. For generic $\Qcal_{N}$ the moduli space is smooth and of expected dimension. Its zero dimensional part $\Ccal(x, y; \Qcal_{N})_0$ defines a chain map
\e
\varphi_N\colon C_N \to D_N.
\e
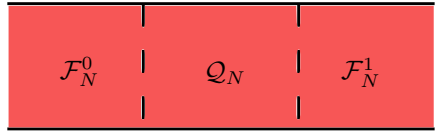
\begin{figure}[!h]
    \centering
    \def\svgwidth{.50\textwidth}
\begingroup%
  \makeatletter%
  \providecommand\color[2][]{%
    \errmessage{(Inkscape) Color is used for the text in Inkscape, but the package 'color.sty' is not loaded}%
    \renewcommand\color[2][]{}%
  }%
  \providecommand\transparent[1]{%
    \errmessage{(Inkscape) Transparency is used (non-zero) for the text in Inkscape, but the package 'transparent.sty' is not loaded}%
    \renewcommand\transparent[1]{}%
  }%
  \providecommand\rotatebox[2]{#2}%
  \newcommand*\fsize{\dimexpr\f@size pt\relax}%
  \newcommand*\lineheight[1]{\fontsize{\fsize}{#1\fsize}\selectfont}%
  \ifx\svgwidth\undefined%
    \setlength{\unitlength}{371.33858268bp}%
    \ifx\svgscale\undefined%
      \relax%
    \else%
      \setlength{\unitlength}{\unitlength * \real{\svgscale}}%
    \fi%
  \else%
    \setlength{\unitlength}{\svgwidth}%
  \fi%
  \global\let\svgwidth\undefined%
  \global\let\svgscale\undefined%
  \makeatother%
  \begin{picture}(1,0.38167939)%
    \lineheight{1}%
    \setlength\tabcolsep{0pt}%
    \put(0,0){\includegraphics[width=\unitlength,page=1]{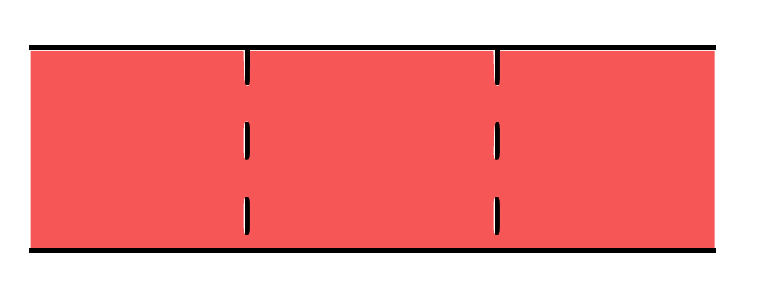}}%
    \put(0.14616661,0.16251564){\color[rgb]{0,0,0}\makebox(0,0)[lt]{\lineheight{1.25}\smash{\begin{tabular}[t]{l}$\Fcal_N^0$\end{tabular}}}}%
    \put(0.7314325,0.16199382){\color[rgb]{0,0,0}\makebox(0,0)[lt]{\lineheight{1.25}\smash{\begin{tabular}[t]{l}$\Fcal_N^1$\end{tabular}}}}%
    \put(0.44612808,0.15978004){\color[rgb]{0,0,0}\makebox(0,0)[lt]{\lineheight{1.25}\smash{\begin{tabular}[t]{l}$\Qcal_N$\end{tabular}}}}%
  \end{picture}%
\endgroup%

      \caption{The perturbation $\Qcal_{N}$  on $(Z; M_N; L_0^N, L_1^N)$.}
      \label{fig:cont_map}
\end{figure}

In order to prove that $\varphi_N$ commutes with $\alpha_N$ and $\beta_N$ up to homotopy we define a parametrized moduli space, involving a family $\lbrace \Rcal^L_{N}\rbrace_{L\in \rr}$ of perturbations on $(\Zund, \Mund_N, \Lund_N)$ such that (see Figure~\ref{fig:cont_map_htpy}): 
\begin{itemize}
\item if $L\ll 0$, then  $\Rcal^L_{N}$ corresponds to the superposition of $\Qcal_N$ shifted by $L$ in the region $s<L/2$ of $\Zund$ and $\Pcal_N^1$ in $s>L/2$. Notice that when $L$ is negatively large enough, these coincide with $\Fcal_N^1$ at $s=L/2$.
\item if $L\gg 0$, then  $\Rcal^L_{N}$ corresponds to the superposition of $\Pcal_N^0$ shifted by $L$ in the region $s<L/2$ of $\Zund$ and $\Qcal_{N+1}$ in $s>L/2$. Notice that when $L$ is large enough, these coincide with $\Fcal_{N+1}^0$ at $s=L/2$.
\end{itemize}

\begin{figure}[!h]
    \centering
    \def\svgwidth{.6\textwidth}
    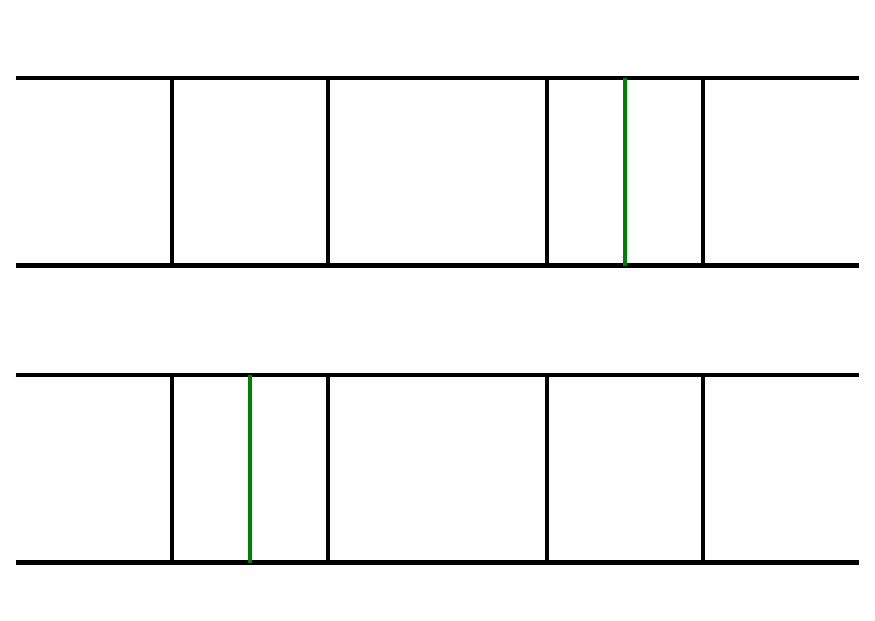
      \caption{The perturbation $\Rcal^L_{N}$ for large $\abs{L}$.}
      \label{fig:cont_map_htpy}
\end{figure}

For each $L$, let then  $\Hcal_L(x,y; \Rcal^L_{N})$ be the moduli space of $\Rcal^L_{N}$-perturbed pseudo-holomorphic quilts $\uund\colon \Zund\to (\Mund_N, \Lund_N)$, and the corresponding parametrized moduli space:
\e
\Hcal(x,y; \lbrace\Rcal^L_{N} \rbrace_L ) =  \bigcup_{L\in \rr} \Hcal_L(x,y; \Rcal^L_{N}),
\e

\begin{prop}\label{prop:phi_htpy_commutes_alpha_beta}
For generic families $ \lbrace\Rcal^L_{N} \rbrace_L$, $\Hcal(x,y; \lbrace\Rcal^L_{N} \rbrace_L )$ is smooth and of expected dimension (i.e. the index of the corresponding Fredholm operator).

Its zero-dimensional part 
\e
\Hcal(x,y; \lbrace\Rcal^L_{N} \rbrace_L )_0 =  \bigcup_{L\in \rr} \Hcal_L(x,y; \Rcal^L_{N})_{-1}
\e
can be used to define a map 
\e
\kappa_N\colon C_N\to D_{N+1}, 
\e
and its one-dimensional part  
\e
\Hcal(x,y; \lbrace\Rcal^L_{N} \rbrace_L )_1 =  \bigcup_{L\in \rr} \Hcal_L(x,y; \Rcal^L_{N})_{0}
\e 
can be used to show that $\kappa_N$ gives a homotopy 
\e
\varphi_{N+1} \alpha_N - \beta_N\varphi_N = \d \kappa_N + \kappa_N \d .
\e 

Therefore one gets a continuation morphism 
\e
\Tel (\varphi_N, \kappa_N ) \colon CF_G^0 \to CF_G^1,
\e 
with $CF_G^0 = \Tel(C_N, \alpha_N)$ and $CF_G^1 = \Tel(D_N, \beta_N)$.
\end{prop}
\begin{proof} Smoothness of these moduli spaces follows from parametrized transversality \cite[Def.~3.1.6, Th.~3.1.6]{McDuSal}.

The boundary of the compactification of $\Hcal(x,y; \lbrace\Rcal^L_{N} \rbrace_L )_1$ consists of three pieces:
\begin{itemize}
\item When $L\to -\infty$, the quilt breaks in a strip in $M_N$ perturbed by $\Qcal_N$, and a quilt perturbed by $\Pcal^1_N$, leading to a contribution for $\beta_N\varphi_N$
\item When $L\to +\infty$, the quilt breaks in a strip in $M_{N+1}$ perturbed by $\Qcal_{N+1}$, and a quilt perturbed by $\Pcal^0_N$, leading to a contribution for $\varphi_{N+1} \alpha_N$
\item At finite $L$, one can have strip breaking at either the incoming or outgoing end, contributing to $\d \kappa_N + \kappa_N \d$.
\end{itemize}
Indeed, all other possible bubbling or strip breaking is excluded, as in the proofs of Propositions~\ref{prop:def_CF_N} and  \ref{prop:Lcal_moduli_space}.
\end{proof}

Now suppose that two different choices $(\Qcal_N^0, \Rcal_N^{L, 0} )$ and $(\Qcal_N^1, \Rcal_N^{L, 1} )$ of regular perturbations were made, leading to maps 
$(\varphi_N^0 ,\kappa_N^0)$ and $(\varphi_N^1 ,\kappa_N^1)$. In order to construct homotopies $(\phi_N ,\kappa_N)$ as in Definition~\ref{def:Homotopies_between_Telescopes_alg} we introduce two kinds of families of perturbations $\Scal_N^v$,  $\Tcal_N^{L,v}$, with $L\in \rr$ and $v\in [0,1]$.

Let $\lbrace \Scal_N^v\rbrace_{v\in [0,1]}$ be a family of perturbations on  $(Z, M_N, L_0^N, L_1^N)$  from $\Fcal_N^0$ to $\Fcal_N^1$ such that, when $v=0$ or 1, $\Scal_N^v = \Qcal_N^v$. Let $\Ccal^v(x,y, \Scal_N^v)$ be the moduli space associated with $\Scal_N^v$, and let
\e
\Ccal(x,y; \lbrace \Scal_N^v\rbrace) = \bigcup_{v\in [0,1]} \Ccal^v(x,y; \Scal_N^v)
\e
be the corresponding parametrized moduli space.

\begin{prop}\label{prop:htpy_continuation}
For generic families $ \lbrace\Scal^v_{N} \rbrace_v$, $\Ccal(x,y; \lbrace \Scal_N^v\rbrace)$ is smooth and of expected dimension.

Its zero-dimensional part $\Ccal(x,y)_0 = \bigcup_{v\in [0,1]} \Ccal^v(x,y)_{-1}$ defines $\phi_N \colon C_N \to D_N$, and its one-dimensional part  $\Ccal(x,y)_1 = \bigcup_{v\in [0,1]} \Ccal^v(x,y)_{0}$ compactifies to a cobordism that gives the relation
\e
\varphi^1_N -\varphi^0_N = \d \phi_N + \phi_N \d .
\e
\end{prop}
\begin{proof}
The proof is analogous to the proof of Proposition~\ref{prop:phi_htpy_commutes_alpha_beta}
\end{proof}

Let $\lbrace\Tcal_N^{L,v}\rbrace_{(L,v)\in \rr\times [0,1]}$ be a family of perturbations on $(\Zund, \Mund_N, \Lund_N)$ such that:
\begin{itemize}
\item if $L\ll 0$, then  $\Tcal_N^{L,v}$ corresponds to the superposition of $\Scal_N^v$ shifted by $L$ in the region $s<L/2$ of $\Zund$ and $\Pcal_N^1$ in $s>L/2$. 
\item if $L\gg 0$, then  $\Tcal_N^{L,v}$ corresponds to the superposition of $\Pcal_N^0$ shifted by $L$ in the region $s<L/2$ of $\Zund$ and $\Scal_{N+1}^v$ in $s>L/2$. 
\item when $v=0$ or 1, $\Tcal_N^{L,v} = \Rcal_N^{L,v}$.
\end{itemize}

Let $\Kcal_L^v(x,y; \lbrace\Tcal_N^{L,v}\rbrace)$ be the associated moduli space of quilts, and 
\e
\Kcal(x,y;\lbrace\Tcal_N^{L,v}\rbrace) = \bigcup_{v\in [0,1],\ L\in \rr} \Kcal_L^v(x,y;\Tcal_N^{L,v}).
\e

\begin{prop}\label{prop:relation_kappa} For generic $\lbrace\Tcal_N^{L,v}\rbrace$, $\Kcal(x,y;\lbrace\Tcal_N^{L,v}\rbrace)$ is smooth and of expected dimension. 
Its zero-dimensional part $\Kcal(x,y)_0 = \bigcup_{v\in [0,1],\ L\in \rr} \Kcal_L^v(x,y)_{-2}$ defines the map $\kappa_N\colon C_N\to D_{N+1}$ and its one-dimensional part 
\e
\Kcal(x,y)_1 = \bigcup_{v\in [0,1],\ L\in \rr} \Kcal_L^v(x,y)_{-1}
\e 
compactifies to a cobordism that gives the relation:
\e\label{eq:relation_kappa}
\kappa^1_N- \kappa^0_N + \phi_{N+1} \alpha_N + \beta_N \phi_N = \d\kappa_N + \kappa_N \d.
\e
Therefore, by Proposition~\ref{prop:homotopy_telescopes}, different paths of perturbations  induce the same maps on telescopes, up to homotopy.
\end{prop}
\begin{proof}

The boundary of the compactification of $\Kcal(x,y)_1$ consists of five kind of degenerations (see Figure~\ref{fig:relation_kappa}):
\begin{enumerate}
\item at $v=0$, contribution for $\kappa_N^0$,
\item at $v=1$, contribution for $\kappa_N^1$,
\item at $L\to -\infty$, contribution for $\beta_N \phi_N$,
\item at $L\to +\infty$, contribution for $\phi_{N+1} \alpha_N$,
\item at finite $L$ and $v\in (0,1)$, strip breaking at incoming and outgoing ends, contributing to $\d\kappa_N +\kappa_N \d$.
\end{enumerate}
\end{proof}

\begin{figure}[!h]
    \centering
    \def\svgwidth{.80\textwidth}
    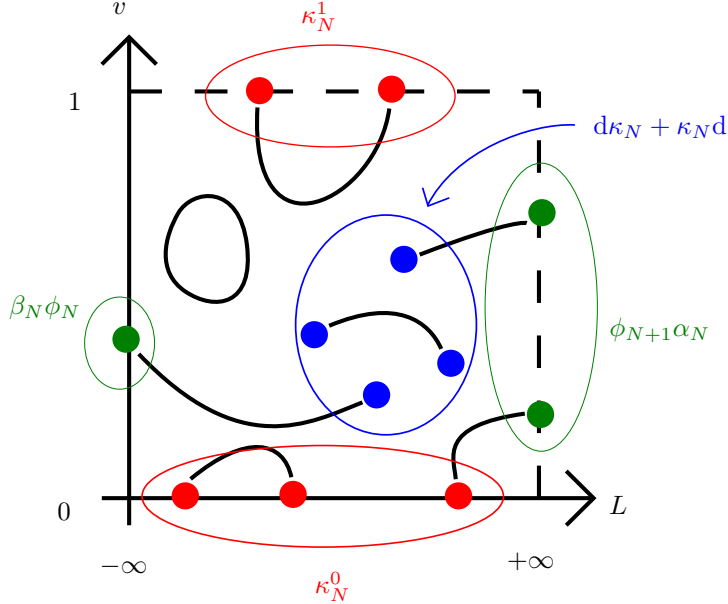
      \caption{The moduli space $\Kcal(x,y)_1$ and its compactification. The boundary components are labelled with their associated contribution to relation~(\ref{eq:relation_kappa})}
      \label{fig:relation_kappa}
\end{figure}

 To show that furthermore these are homotopy equivalences, by Corollary~\ref{cor:homotopy_equiv_telescopes}, one has to check  that if:
 \begin{itemize}
 \item $(\varphi_N, \kappa_N)$ are as in Proposition~\ref{prop:phi_htpy_commutes_alpha_beta} 
 from $(C_N, \alpha_N)$ to $(D_N, \beta_N)$,
 \item $(\tilde{\varphi}_N, \tilde{\kappa}_N)$  are   from $(D_N, \beta_N)$ to $(C_N, \alpha_N)$,
 \end{itemize}
then their composition  $(\tilde{\varphi}_N, \tilde{\kappa}_N) \circ(\varphi_N, \kappa_N)$ is homotopic to $(id_{C_N}, 0)$, i.e. there exists 
 \e
 (\phi_N, H_N)\colon (C_N, \alpha_N) \to (C_N, \alpha_N)
 \e 
 such that:
 
\ea
\tilde{\varphi}_N \varphi_N - id &= \d\phi_N + \phi_N \d ,  \label{eq:compo_htpic_id_1} \\
\tilde{\kappa}_N \phi_N + \tilde{\varphi}_{N+1} \kappa_N + \phi_{N+1}\alpha_N + \alpha_N \phi_N   &= \d H_N + H_N \d \label{eq:compo_htpic_id_2} .
\ea 
 
 Let $\Qcal_N$, $\Rcal_N^L$ (resp. $\tilde{\Qcal}_N$, $\tilde{\Rcal}_N^{\tilde{L}}$) be the (families of) perturbations as earlier involved in the definitions of $(\varphi_N, \kappa_N)$ (resp. $(\tilde{\varphi}_N, \tilde{\kappa}_N)$).

 Let $\lbrace\Ucal_N^v\rbrace_{v\in [0, +\infty )}$ be a family of perturbations on $(Z, M_N, L_0^N, L_1^N)$ going from $\Fcal_N^0$ to $\Fcal_N^0$, and such that:
 \begin{itemize}
 \item when $v=0$, the perturbation $\Ucal_N^v$ is just $\Fcal_N^0$ (unperturbed),
 \item when $v\gg 0$, the perturbation $\Ucal_N^v$ is a superposition of $\Qcal_N$ and $\tilde{\Qcal}_N$ shifted by $v$, as in Figure~\ref{fig:pert_Ucal}.
 \end{itemize}
 
\begin{figure}[!h]
    \centering
    \def\svgwidth{.50\textwidth}
\begingroup%
  \makeatletter%
  \providecommand\color[2][]{%
    \errmessage{(Inkscape) Color is used for the text in Inkscape, but the package 'color.sty' is not loaded}%
    \renewcommand\color[2][]{}%
  }%
  \providecommand\transparent[1]{%
    \errmessage{(Inkscape) Transparency is used (non-zero) for the text in Inkscape, but the package 'transparent.sty' is not loaded}%
    \renewcommand\transparent[1]{}%
  }%
  \providecommand\rotatebox[2]{#2}%
  \newcommand*\fsize{\dimexpr\f@size pt\relax}%
  \newcommand*\lineheight[1]{\fontsize{\fsize}{#1\fsize}\selectfont}%
  \ifx\svgwidth\undefined%
    \setlength{\unitlength}{419.52755906bp}%
    \ifx\svgscale\undefined%
      \relax%
    \else%
      \setlength{\unitlength}{\unitlength * \real{\svgscale}}%
    \fi%
  \else%
    \setlength{\unitlength}{\svgwidth}%
  \fi%
  \global\let\svgwidth\undefined%
  \global\let\svgscale\undefined%
  \makeatother%
  \begin{picture}(1,0.52702703)%
    \lineheight{1}%
    \setlength\tabcolsep{0pt}%
    \put(0,0){\includegraphics[width=\unitlength,page=1]{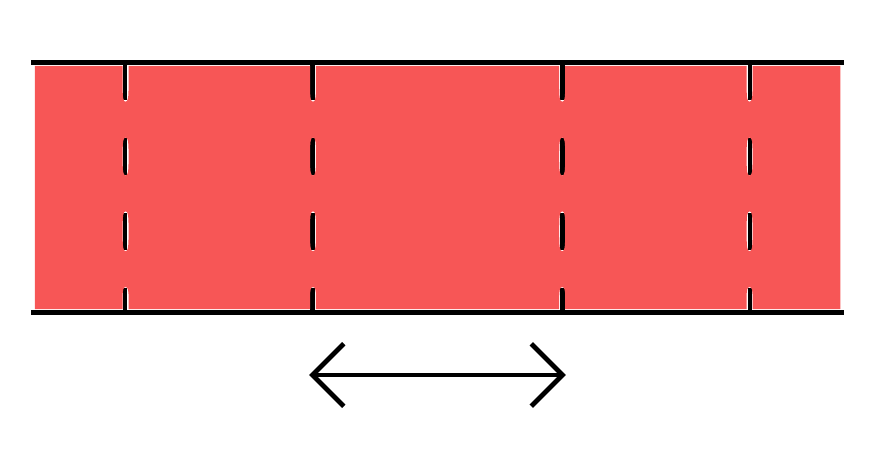}}%
    \put(0.46922151,0.01882658){\color[rgb]{0,0,0}\makebox(0,0)[lt]{\lineheight{1.25}\smash{\begin{tabular}[t]{l}$v$\end{tabular}}}}%
    \put(0.22295651,0.30285454){\color[rgb]{0,0,0}\makebox(0,0)[lt]{\lineheight{1.25}\smash{\begin{tabular}[t]{l}$\Qcal_N$\end{tabular}}}}%
    \put(0.70952653,0.30292448){\color[rgb]{0,0,0}\makebox(0,0)[lt]{\lineheight{1.25}\smash{\begin{tabular}[t]{l}$\tilde{\Qcal}_N$\end{tabular}}}}%
  \end{picture}%
\endgroup%

      \caption{The perturbation $\Ucal_N^v$ for large $v$.}
      \label{fig:pert_Ucal}
\end{figure}
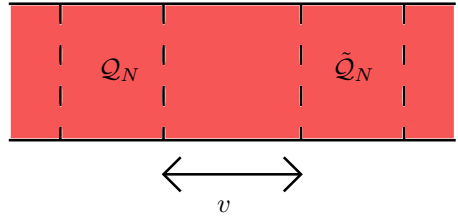
 
 These fit into our existing diagram of perturbations as follows:
 \e\label{diag:pert_data_htpy_inv}
\xymatrix{ \Fcal^0_{N} \ar[d]^{\Qcal_{N}} \ar[r]^{\Pcal^0_{N}} \ar[dr]^{\Rcal^L_{N}} \ar@/^-3pc/[dd]^{\Ucal_N^v} & \Fcal^0_{N+1} \ar[d]^{\Qcal_{N+1}} \ar@/^3pc/[dd]^{\Ucal_{N+1}^v} \\
\Fcal^1_{N} \ar[d]^{\tilde{\Qcal}_{N}} \ar[dr]^{\tilde{\Rcal}^{\tilde{L}}_{N}}  \ar[r]^{\Pcal^1_{N}}  & \Fcal^1_{N+1}\ar[d]^{\tilde{\Qcal}_{N+1}}  \\ 
\Fcal^0_{N} \ar[r]^{\Pcal^0_{N}} & \Fcal^0_{N+1}  \\} 
\e

 These define moduli spaces $\Dcal^v(x,y;\Ucal_N^v)$ and a corresponding parametrized  moduli space 
 \e
 \Dcal(x,y;\lbrace\Ucal_N^v\rbrace) = \bigcup_{v}\Dcal^v(x,y;\Ucal_N^v)
 \e 
whose zero dimensional component defines the map $\phi_N$, and  one-dimensional part serves to prove the Relation~(\ref{eq:compo_htpic_id_1}).

 To construct $H_N$ and prove the Relation~(\ref{eq:compo_htpic_id_2}) we define a two-parameter family of perturbations on $(\Zund, \Mund_N, \Lund_N)$ parametrized by a square, and obtained by patching together five families $\Vcal^1_N, \Vcal^2_N, \Vcal^3_N, \Vcal^4_N, \Vcal^5_N$ of perturbations, with $K>0$ large enough (see Figure~\ref{fig:cont_map_htpy_inverse}):
 \begin{enumerate}
 \item $\Vcal^1_N$ is parametrized by $(v, \delta_1)\in [0,+\infty)\times [K,+\infty)$ and corresponds to the superposition of $\Ucal_N^v$ and $\Pcal_N^0$ spaced by $\delta_1$,
 \item $\Vcal^2_N$ is parametrized by $(\tilde{L}, \delta_2)\in \rr\times [K,+\infty)$ and corresponds to the superposition of $\Qcal_N$ and $\tilde{\Rcal}_N^{\tilde{L}}$ spaced by $\delta_2$,
 \item $\Vcal^3_N$ is parametrized by $(L, \delta_3)\in \rr\times [K,+\infty)$ and corresponds to the superposition of $\Rcal_N^{L}$ and $\tilde{\Qcal}_{N+1}$   spaced by $\delta_3$,
 \item $\Vcal^4_N$ is parametrized by $(v, \delta_4)\in [0,+\infty)\times [K,+\infty)$ and corresponds to the superposition of $\Pcal_N^0$ and $\Ucal_{N+1}^v$ and  spaced by $\delta_4$,
 \end{enumerate}

\begin{remark}Intuitively, one can think about $\Vcal^1_N, \Vcal^2_N, \Vcal^3_N, \Vcal^4_N$ as "filling in" the four 2-cells in Diagram~(\ref{diag:pert_data_htpy_inv}) after removing the arrows ${\Rcal}^{{L}}_{N}$ and $\tilde{\Rcal}^{\tilde{L}}_{N}$, see Figure~\ref{fig:twocell}. 
\end{remark}

\begin{figure}[!h]
    \centering
    \def\svgwidth{.50\textwidth}
    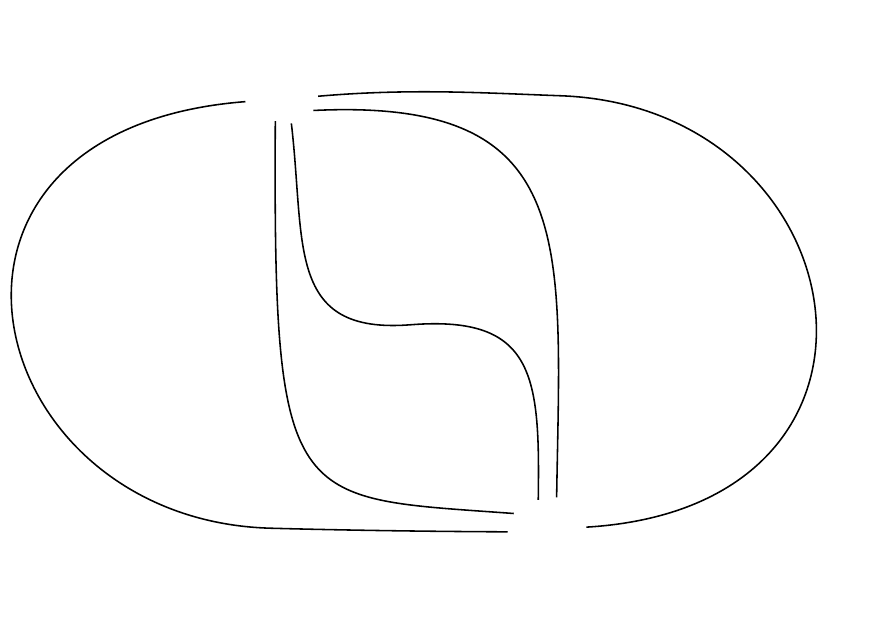
      \caption{$\Vcal^1_N, \Vcal^2_N, \Vcal^3_N, \Vcal^4_N$ as 2-cells.}
      \label{fig:twocell}
\end{figure}

Notice that these four families have some overlaps as indicated in Figure~\ref{fig:cont_map_htpy_inverse} : for example $\Vcal^1_N$ and  $\Vcal^2_N$ coincide when $v=\delta_2$ and $\delta_1 = -\tilde{L}$. Therefore one can patch together their parameter spaces as in  Figure~\ref{fig:cont_map_htpy_inverse_2} to obtain a square with a smaller square at its center removed. Let then $\Vcal^5_N$ be any family that extends smoothly the boundary of this central square: use it to fill the middle square. We then get a family $\lbrace\Vcal^w_N\rbrace_{w\in [0,1]^2}$ that allows us to define moduli spaces $\Ecal^w(x,y; \Vcal^w_N)$ and 
 \e
 \Ecal(x,y; \lbrace\Vcal^w_N\rbrace) = \bigcup_{w\in [0,1]^2}\Ecal^w(x,y; \Vcal^w_N).
 \e
Use their zero dimensional part to define $H_N$, and get (\ref{eq:compo_htpic_id_2}) from its one dimensional part: each side of the square giving respectively the summands $\alpha_N \phi_N $, $\tilde{\kappa}_N \phi_N$, $\tilde{\varphi}_{N+1} \kappa_N$, $\phi_{N+1}\alpha_N$, and the right hand side $\d H_N + H_N \d$ corresponds to strip breaking at interior points.
\begin{figure}[!h]
    \centering
    \def\svgwidth{1.0\textwidth}
    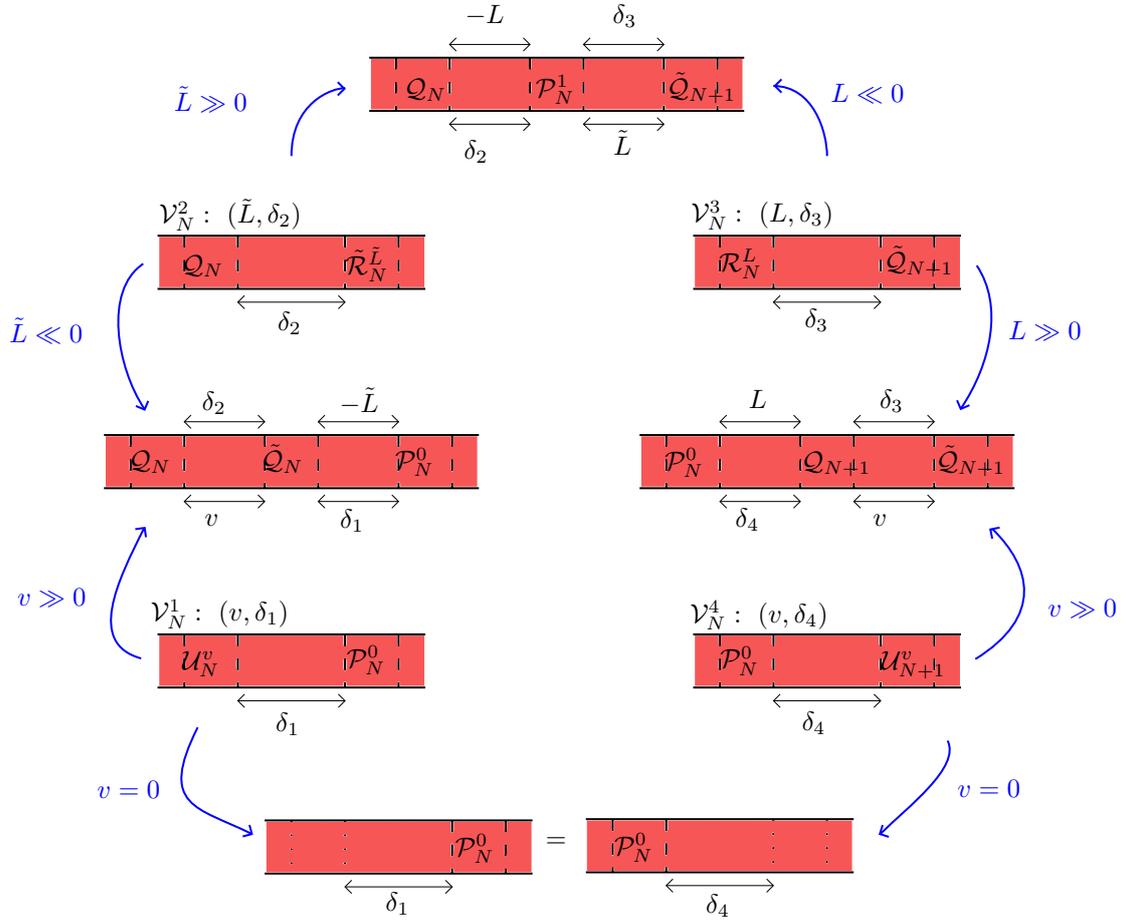
      \caption{The families $\Vcal^1_N, \Vcal^2_N, \Vcal^3_N, \Vcal^4_N$ are drawn in lines 2 and 4. The blue arrows indicate their overlaps, as in Figure~\ref{fig:cont_map_htpy_inverse_2}.}
      \label{fig:cont_map_htpy_inverse}
\end{figure}

\begin{figure}[!h]
    \centering
    \def\svgwidth{.50\textwidth}
\begingroup%
  \makeatletter%
  \providecommand\color[2][]{%
    \errmessage{(Inkscape) Color is used for the text in Inkscape, but the package 'color.sty' is not loaded}%
    \renewcommand\color[2][]{}%
  }%
  \providecommand\transparent[1]{%
    \errmessage{(Inkscape) Transparency is used (non-zero) for the text in Inkscape, but the package 'transparent.sty' is not loaded}%
    \renewcommand\transparent[1]{}%
  }%
  \providecommand\rotatebox[2]{#2}%
  \newcommand*\fsize{\dimexpr\f@size pt\relax}%
  \newcommand*\lineheight[1]{\fontsize{\fsize}{#1\fsize}\selectfont}%
  \ifx\svgwidth\undefined%
    \setlength{\unitlength}{538.58267717bp}%
    \ifx\svgscale\undefined%
      \relax%
    \else%
      \setlength{\unitlength}{\unitlength * \real{\svgscale}}%
    \fi%
  \else%
    \setlength{\unitlength}{\svgwidth}%
  \fi%
  \global\let\svgwidth\undefined%
  \global\let\svgscale\undefined%
  \makeatother%
  \begin{picture}(1,1.00526316)%
    \lineheight{1}%
    \setlength\tabcolsep{0pt}%
    \put(0,0){\includegraphics[width=\unitlength,page=1]{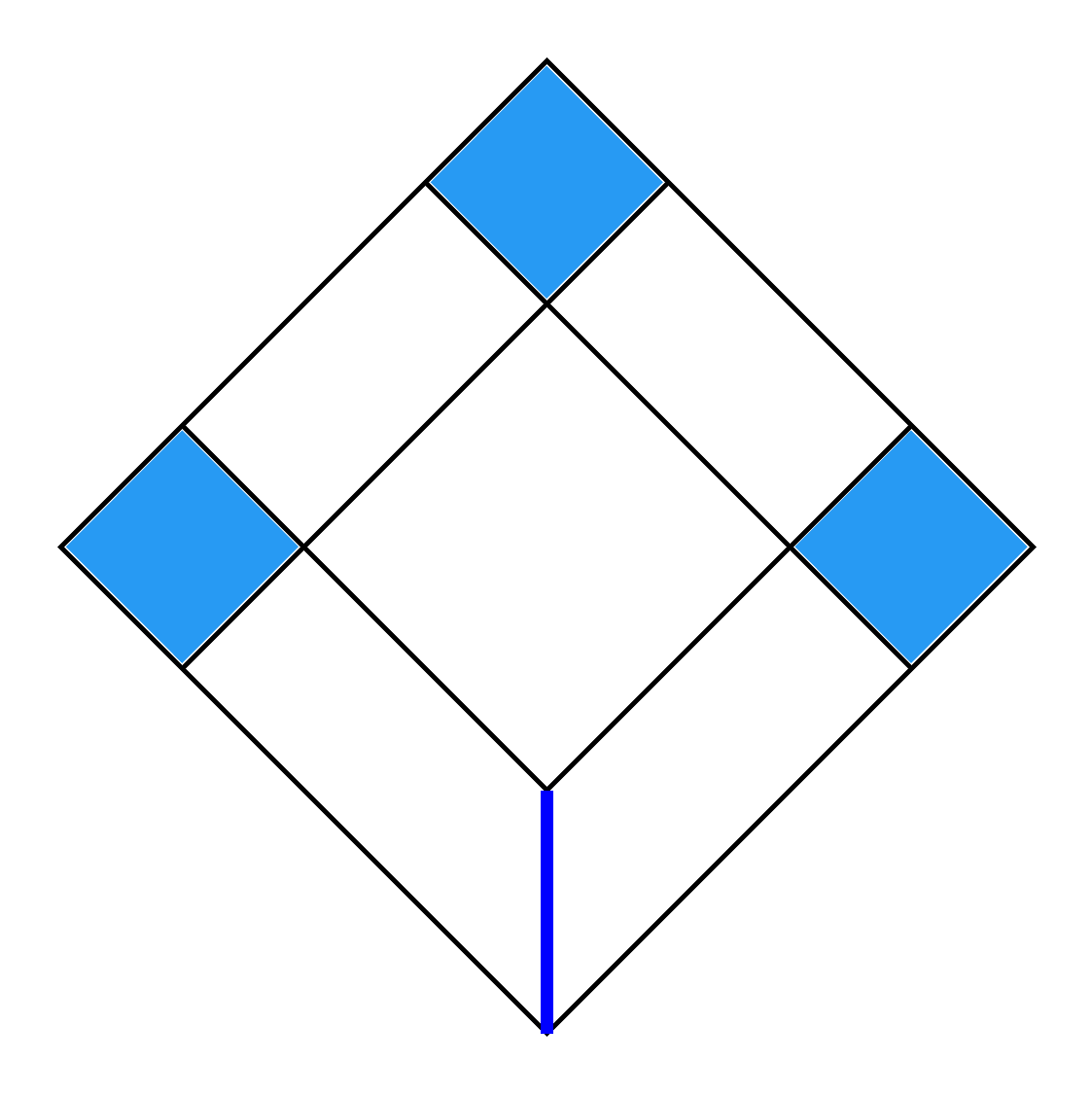}}%
    \put(0.31436489,0.2881904){\color[rgb]{0,0,0}\makebox(0,0)[lt]{\lineheight{1.25}\smash{\begin{tabular}[t]{l}$\Vcal_N^1$\end{tabular}}}}%
    \put(0.28650474,0.64161428){\color[rgb]{0,0,0}\makebox(0,0)[lt]{\lineheight{1.25}\smash{\begin{tabular}[t]{l}$\Vcal_N^2$\end{tabular}}}}%
    \put(0.62021153,0.2919202){\color[rgb]{0,0,0}\makebox(0,0)[lt]{\lineheight{1.25}\smash{\begin{tabular}[t]{l}$\Vcal_N^4$\end{tabular}}}}%
    \put(0.62984829,0.64474759){\color[rgb]{0,0,0}\makebox(0,0)[lt]{\lineheight{1.25}\smash{\begin{tabular}[t]{l}$\Vcal_N^3$\end{tabular}}}}%
    \put(0.46020022,0.47673739){\color[rgb]{0,0,0}\makebox(0,0)[lt]{\lineheight{1.25}\smash{\begin{tabular}[t]{l}$\Vcal_N^5$\end{tabular}}}}%
  \end{picture}%
\endgroup%

      \caption{The square parametrizing $\Vcal^1_N, \Vcal^2_N, \Vcal^3_N, \Vcal^4_N, \Vcal^5_N$.}
      \label{fig:cont_map_htpy_inverse_2}
\end{figure}

\section{Self-Floer homology and Morse homology}
\label{sec:Floer_vs_Morse}

In this section we prove Theorem~\ref{th:Floer_vs_Morse_intro} (except for the $H^*(BG)$-bimodule part, which will be proved in Proposition~\ref{prop:PSS_mod_str}). To do so, we use its well-known non-equivariant analogue, involving the Piunikhin-Salamon-Schwarz isomorphism, and show that these maps commute with the increments 
\ea
\alpha_N&\colon CF_N \to CF_{N+1}, \\
j_N&\colon CM_N \to CM_{N+1}
\ea
 up to homotopy. Such isomorphisms first appeared in \cite{PSS} for Hamiltonian Floer homology, and were extended to the Lagrangian setting in \cite{Albers_PSS}. We briefly review their construction, and refer to the later reference for more details.

\subsection{PSS isomorphisms}
\label{ssec:PSS_isomorphisms}
 
 Suppose $G=1$ and $L\subset M$ satisfies either Assumption~\ref{ass:exact_setting} or \ref{ass:monotone_setting}, with $L_0 = L_1 = L$ (i.e. no $G$-action). 
 
 Let $\Fcal$ be a perturbation datum for $(M; L, L)$, $f$ a Morse function on $L$, with a pseudo-gradient $v$. Then two chain morphisms
 \ea
 PSS&\colon CF(M; L, L; \Fcal) \to CM(L; f,v), \\
 SSP&\colon CM(L; f,v) \to CF(M; L, L; \Fcal),
 \ea
can be defined, and are inverses from each other up to homotopy, in the sense that 
\ea
PSS\circ SSP -Id_{CM} = \d H+H\d , \\
SSP\circ PSS -Id_{CF} = \d K+K\d .
\ea

Let $x$ and $y$ be generators of $CF(M; L, L; \Fcal)$ and $CM(L; f,v)$ respectively. The morphism $PSS$ is defined by counting the zero-dimensional part of a moduli space $\Pcal\Scal\Scal(x,y)$ consisting of pairs of strips and flow half-lines (see Figure~\ref{fig:PSS_SSP}):
\ea
u &\colon Z\to M , \\
\gamma &\colon \rr_{\geq 0} \to L , 
\ea
such that: 
\begin{itemize}
\item $u$ is a $\Qcal$-holomorphic curve, with $\Qcal$ a regular perturbation datum on $Z$ from $\Fcal$ to $(H= 0, J= J_0)$, with $J_0$ a constant \acs .
\item $\gamma$ is a flowline for $v$.
\item $\lim_{s\to -\infty} u(s+it) = x$, $\lim_{s\to +\infty} \gamma(s) = y$, and $(u,\gamma)$ satisfy the matching condition $\lim_{s\to +\infty} u(s+it) = \gamma(0)$.
\end{itemize}
The moduli space $\Scal\Scal\Pcal(y,x)$ involved in defining $SSP$ is defined analogously, but the other way around.
\begin{figure}[!h]
    \centering
    \def\svgwidth{.50\textwidth}
    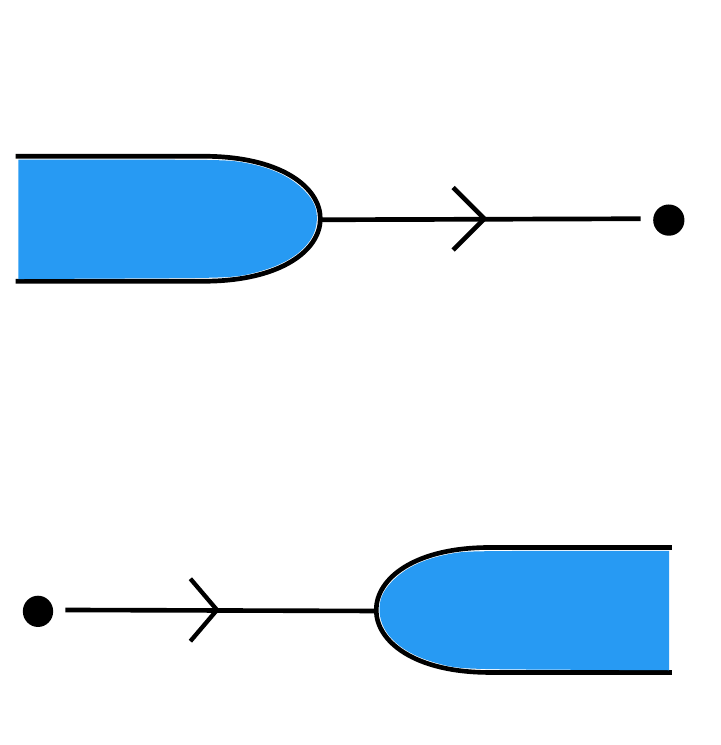
      \caption{The moduli spaces $\Pcal\Scal\Scal(x,y)$ and $\Scal\Scal\Pcal(y,x)$.}
      \label{fig:PSS_SSP}
\end{figure}

Now for each $N$, let 
\begin{itemize}

\item $\Fcal_N$ be a Floer datum for $(M_N; L^N, L^N)$, and 
\e
CF_N = CF(M_N; L^N, L^N;\Fcal_N)
\e

\item $\Pcal_N$ a regular perturbation datum on $\Zund$, inducing 
\e
\alpha_N\colon CF_N \to CF_{N+1}
\e

\item  $f_N$ be a Morse function on $L^N$, $v_N$ a pseudo-gradient for $f_N$, and 
\e
CM_N= CM(L^N; f_N, v_N)
\e

\item $\Qcal_N$ a regular perturbation datum on $ \rr_{\leq 0} \times [0,1]$, inducing a $PSS$-morphism
\e
PSS_N\colon CF_N\to CM_N
\e
\end{itemize}
Proving commutativity with $\alpha_N, j_N$, i.e.
\e\label{eq:PSS_commutes_N_to_N+1}
PSS_{N+1}\alpha_N -j_N PSS_N = \d\kappa_N+ \kappa_N \d ,
\e
involves a parametrized moduli space (Figure~\ref{fig:PSS_commmutes_i_N})

\e
\Xcal (x,y) = \bigcup_{L\in \rr} \Xcal^L (x,y) ,
\e
where, for $L<0$, $\Xcal^L (x,y)$ consists of a quilted strip and a flow line, i.e. triples 
\ea
u_N &\colon (-\infty, \psi(L)] \times [0, 1] \to M_N ,\\
u_{N+1} &\colon [ \psi(L), +\infty) \times [0, 1] \to M_{N+1} ,\\
\gamma_{N+1} &\colon \rr_{\geq 0} \to L_{N+1},
\ea
with $\psi\colon \rr_{<0}\to \rr$ some fixed increasing diffeomorphism. These triples satisfy the conditions below. 
\begin{itemize}
\item Fix first a one parameter family of perturbations $\Rcal^L_N$, $L<0$, on $\Zund [\psi(L)]$ (i.e. $\Zund$ with the vertical seam at $s=\psi(L)$) such that when $L\ll 0$, $\Rcal^L_N$ is as in Figure~\ref{fig:perturbation_R_N_L_PSS}
\item $(u_N ,u_{N+1})$ is a $\Rcal^L_N$-holomorphic quilt, with the obvious boundary and seam conditions
\item $\gamma_{N+1}$ is a flow line for $v_N$, 
\item $\lim_{s\to -\infty} u_N = x$, $\lim_{s\to +\infty} \gamma_{N+1} = y$, $\lim_{s\to +\infty} u_{N+1} = \gamma_{N+1}(0)$.
\end{itemize}

\begin{figure}[!h]
    \centering
    \def\svgwidth{.70\textwidth}
    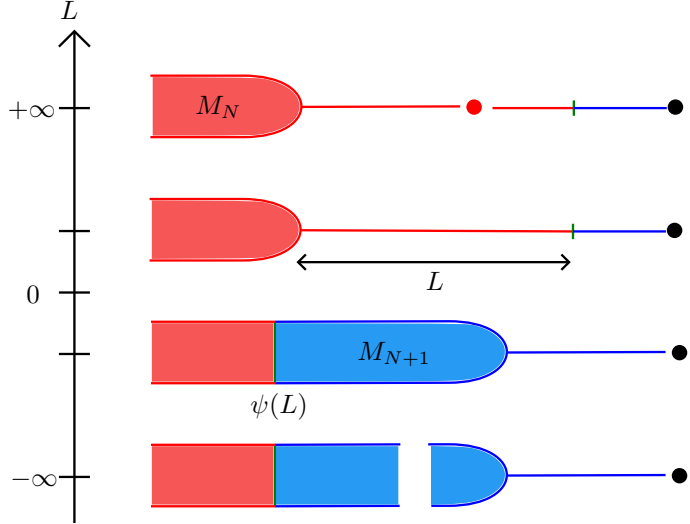
      \caption{PSS commutes with the increments: the deformation parametrizing the moduli space $\Xcal (x,y)$.}
      \label{fig:PSS_commmutes_i_N}
\end{figure}

And for $L\geq 0$, $\Xcal^L (x,y)$ consists of a strip and a grafted line, i.e. triples 
\ea
u_N &\colon Z \to M_N ,\\
\gamma_N &\colon [0, L]  \to L_N ,\\
\gamma_{N+1} &\colon [L, +\infty ) \to L_{N+1},
\ea
such that
\begin{itemize}
\item $u_N$ is $\Fcal_N$-holomorphic, 
\item $\gamma_N$ (resp. $\gamma_{N+1}$) is a flow line for $v_N$ (resp. $v_{N+1}$),
\item $\lim_{s\to -\infty} u_N = x$, $\lim_{s\to +\infty} u_N = \gamma_N (0)$, $\gamma_N(L) = \gamma_{N+1}(L)$, $\lim_{s\to +\infty}\gamma_{N+1} = y$.
\end{itemize}

\begin{figure}[!h]
    \centering
    \def\svgwidth{.50\textwidth}
\begingroup%
  \makeatletter%
  \providecommand\color[2][]{%
    \errmessage{(Inkscape) Color is used for the text in Inkscape, but the package 'color.sty' is not loaded}%
    \renewcommand\color[2][]{}%
  }%
  \providecommand\transparent[1]{%
    \errmessage{(Inkscape) Transparency is used (non-zero) for the text in Inkscape, but the package 'transparent.sty' is not loaded}%
    \renewcommand\transparent[1]{}%
  }%
  \providecommand\rotatebox[2]{#2}%
  \newcommand*\fsize{\dimexpr\f@size pt\relax}%
  \newcommand*\lineheight[1]{\fontsize{\fsize}{#1\fsize}\selectfont}%
  \ifx\svgwidth\undefined%
    \setlength{\unitlength}{419.52755906bp}%
    \ifx\svgscale\undefined%
      \relax%
    \else%
      \setlength{\unitlength}{\unitlength * \real{\svgscale}}%
    \fi%
  \else%
    \setlength{\unitlength}{\svgwidth}%
  \fi%
  \global\let\svgwidth\undefined%
  \global\let\svgscale\undefined%
  \makeatother%
  \begin{picture}(1,0.40540541)%
    \lineheight{1}%
    \setlength\tabcolsep{0pt}%
    \put(0,0){\includegraphics[width=\unitlength,page=1]{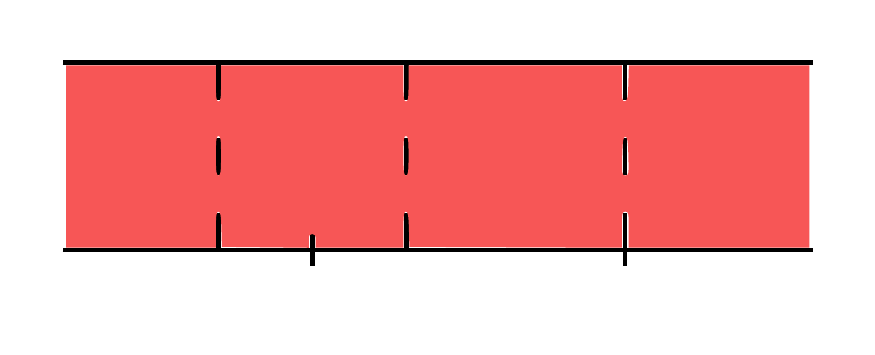}}%
    \put(0.73636335,0.20342857){\color[rgb]{0,0,0}\makebox(0,0)[lt]{\lineheight{1.25}\smash{\begin{tabular}[t]{l}$(0, J_0)$\end{tabular}}}}%
    \put(0.12292833,0.20520719){\color[rgb]{0,0,0}\makebox(0,0)[lt]{\lineheight{1.25}\smash{\begin{tabular}[t]{l}$\Fcal_{N}$\end{tabular}}}}%
    \put(0.31921193,0.20102392){\color[rgb]{0,0,0}\makebox(0,0)[lt]{\lineheight{1.25}\smash{\begin{tabular}[t]{l}$\Pcal_{N}$\end{tabular}}}}%
    \put(0.50796384,0.20127597){\color[rgb]{0,0,0}\makebox(0,0)[lt]{\lineheight{1.25}\smash{\begin{tabular}[t]{l}$\Fcal_{N+1}$\end{tabular}}}}%
    \put(0.30136152,0.02452533){\color[rgb]{0,0,0}\makebox(0,0)[lt]{\lineheight{1.25}\smash{\begin{tabular}[t]{l}$\psi(L)$\end{tabular}}}}%
    \put(0.6963612,0.02998311){\color[rgb]{0,0,0}\makebox(0,0)[lt]{\lineheight{1.25}\smash{\begin{tabular}[t]{l}0\end{tabular}}}}%
  \end{picture}%
\endgroup%

      \caption{$\Rcal_N^L$ when $L \ll 0$.}
      \label{fig:perturbation_R_N_L_PSS}
\end{figure}

\begin{prop}\label{prop:equiv_PSS}
For regular perturbations, $\Xcal(x,y)$ is smooth of expected dimension, its zero dimensional part defines a map $\kappa_N$, and its one-dimensional part allows us to prove (\ref{eq:PSS_commutes_N_to_N+1}).

Therefore one has a well-defined map between telescopes
\e
PSS_G = \Tel(PSS_N, \kappa_N) \colon CF_G(M;L,L) \to CM_G(L) .
\e
Likewise, one can define analogously maps  $SSP_N$ commuting with $\alpha_N, j_N$ up to $\d \kappa_N' + \kappa_N' \d $, and a map
\e
SSP_G = \Tel(SSP_N, \kappa_N ') \colon CM_G(L) \to CF_G(M;L,L) .
\e
These two maps are inverse of each other up to homotopy.

\end{prop}
\begin{proof}
The proof is the same proof that $PSS - SSP = Id+\d H + H\d$, upgraded to the telescopic setting, in an analogous way to what we did in Section~\ref{sec:continuation_maps}.

\end{proof}

\section{Bimodule structure}
\label{sec:module_str}

Recall that equivariant cohomology $H_G^*(X)$ has a $H^*(BG)$-module structure. By combining the (pair of pants) product structure on Floer cohomology with the Lagrangian correspondence between $M_N$ and $B_N$ we define a chain level bimodule structure analogous to it. Let 
\ea
CF_N &= CF(M_N; L_0^N,  L_1^N) , \\
A_N &= CF(B_N; 0_{B_N}, 0_{B_N}) \simeq CM(BG_N).
\ea

For $i=0,1$, let $P_i^N \subset M_N^-  \times B_N$ be the quotient of
\e
 L_i \times \Delta_{T_N} \subset  M^- \times T_N^- \times T_N
\e
by the action of $G\times G$.

One can then define maps 
\ea
L_N&\colon A_N \otimes CF_N \to   CF_N  \\
R_N&\colon CF_N  \otimes A_N \to   CF_N ,
\ea
defined by counting quilts as in the left side of Figure~\ref{fig:bimod_str}. 

\begin{remark}\label{rem:module_str_compo_pair_of_pants} By stretching the quilt as in the right side of Figure~\ref{fig:bimod_str}, one can show that $L_N$ (resp. $R_N$) is obtained by composing the morphism 
\e
A_N \to CF(M_N; L_0^N,  L_0^N)
\e 
(resp. $A_N \to CF(M_N; L_0^N,  L_0^N) $) induced by $P_0^N$ (resp. $P_1^N$) with the pair of pants product 
\e
CF(M_N; L_0^N,  L_0^N) \otimes CF_N \to CF_N
\e 
(resp. $CF_N \otimes CF(M_N; L_1^N,  L_1^N) \to CF_N$).
\end{remark}

\begin{remark}\label{rem:module_str_Morse} In Morse homology, these quilts correspond to grafted trees as in Figure~\ref{fig:module_str_Morse}.
\end{remark}

\begin{figure}[!h]
    \centering
    \def\svgwidth{1.0\textwidth}
    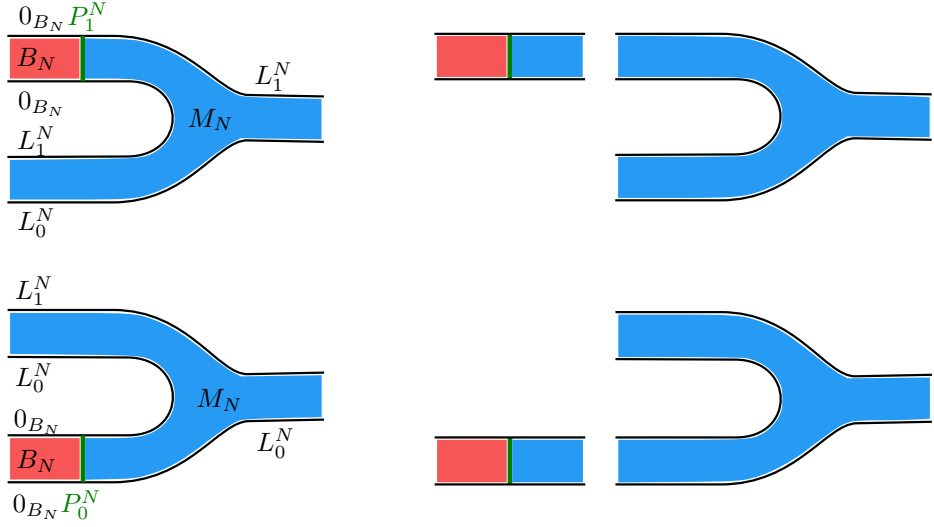
      \caption{Quilts defining $L_N$ and $R_N$.}
      \label{fig:bimod_str}
\end{figure}

\begin{figure}[!h]
    \centering
    \def\svgwidth{1.0\textwidth}
    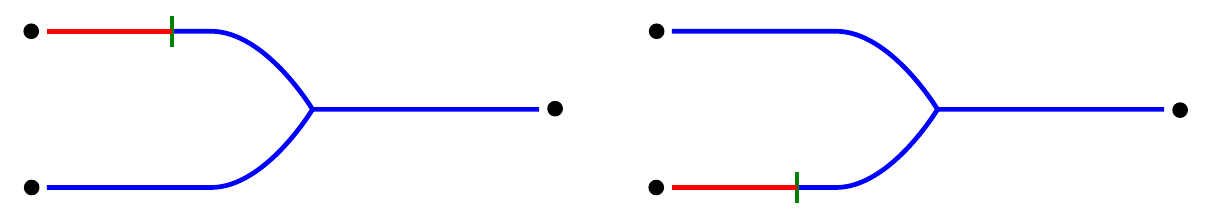
      \caption{Grafted trees analogues in Morse homology.}
      \label{fig:module_str_Morse}
\end{figure}

\begin{prop}\label{prop:bimod_str}

The morphisms $L_N$ and $R_N$ are associative up to homotopy, i.e. there exists
\ea
\lambda^L_N &\colon A_N \otimes A_N \otimes CF_N \to CF_N , \\
\lambda^R_N &\colon   CF_N \otimes A_N \otimes A_N \to CF_N  
\ea
such that
\ea
L_N \circ (id_{A_N} \otimes L_N) + L_N \circ (m_{A_N} \otimes id_{CF_N}) &= \d \lambda^L_N + \lambda^L_N \d , \label{eq:module_relation_left}\\
R_N\circ (R_N \otimes id_{A_N}) + R_N\circ (id_{CF_N} \otimes m_{A_N}  ) &= \d \lambda^R_N + \lambda^R_N \d , \label{eq:module_relation_right}
\ea
with $ m_{A_N}\colon A_N \otimes A_N \to A_N$ the pair of pants product.

Furthermore, these maps commute with the increment maps, respectively up to homotopies $\d\kappa^L_N + \kappa^L_N\d$ and  $\d \kappa^R_N + \kappa^R_N \d$, and therefore induce maps between telescopes
\ea
\Tel(L_N, \kappa^L_N) &\colon  \Tel ( A_N \otimes CF_N) \to \Tel(  CF_N)  \\
\Tel(R_N, \kappa^R_N)&\colon \Tel ( CF_N  \otimes A_N ) \to \Tel(  CF_N) ,
\ea
Combined with the maps of Proposition~\ref{prop:product_telescopes_chain_cpx} one gets product maps
\ea
L&\colon  \Tel ( A_N) \otimes \Tel ( CF_N) \to \Tel(  CF_N),  \\
R&\colon \Tel ( CF_N) \otimes \Tel (A_N) \to \Tel( CF_N) .
\ea
At the homology level, these induce a $H^*(BG)$-bimodule structure on $HF_G(L_0, L_1)$.
\end{prop}
\begin{remark}\label{rem:A_infty_bimodule} By combining the $A_\infty$-structures on $CF(M_N; L_i^N,  L_i^N)$ with the morphism $A_N \to CF(M_N; L_i^N,  L_i^N)$ induced by $P_i^N$, one should be able to show that the chain complex $CF_G$ is an $A_\infty$-bimodule over $CM_G(BG)$.
\end{remark}
\begin{proof}
Let us first prove (\ref{eq:module_relation_left}) (the proof of (\ref{eq:module_relation_right}) is similar). This can be done, at first glance, by considering the deformation suggested in  Figure~\ref{fig:bimod_assoc_1}: one defines a moduli space 
\e
\Mcal = \bigcup_{t\in \rr} \Mcal_t
\e
 that interpolates from $L_N \circ (m_{A_N} \otimes id_{CF_N})$ to $L_N \circ (id_{A_N} \otimes L_N)$.
 
\begin{figure}[!h]
    \centering
    \def\svgwidth{.350\textwidth}
\begingroup%
  \makeatletter%
  \providecommand\color[2][]{%
    \errmessage{(Inkscape) Color is used for the text in Inkscape, but the package 'color.sty' is not loaded}%
    \renewcommand\color[2][]{}%
  }%
  \providecommand\transparent[1]{%
    \errmessage{(Inkscape) Transparency is used (non-zero) for the text in Inkscape, but the package 'transparent.sty' is not loaded}%
    \renewcommand\transparent[1]{}%
  }%
  \providecommand\rotatebox[2]{#2}%
  \newcommand*\fsize{\dimexpr\f@size pt\relax}%
  \newcommand*\lineheight[1]{\fontsize{\fsize}{#1\fsize}\selectfont}%
  \ifx\svgwidth\undefined%
    \setlength{\unitlength}{850.39370079bp}%
    \ifx\svgscale\undefined%
      \relax%
    \else%
      \setlength{\unitlength}{\unitlength * \real{\svgscale}}%
    \fi%
  \else%
    \setlength{\unitlength}{\svgwidth}%
  \fi%
  \global\let\svgwidth\undefined%
  \global\let\svgscale\undefined%
  \makeatother%
  \begin{picture}(1,1.66666667)%
    \lineheight{1}%
    \setlength\tabcolsep{0pt}%
    \put(0,0){\includegraphics[width=\unitlength,page=1]{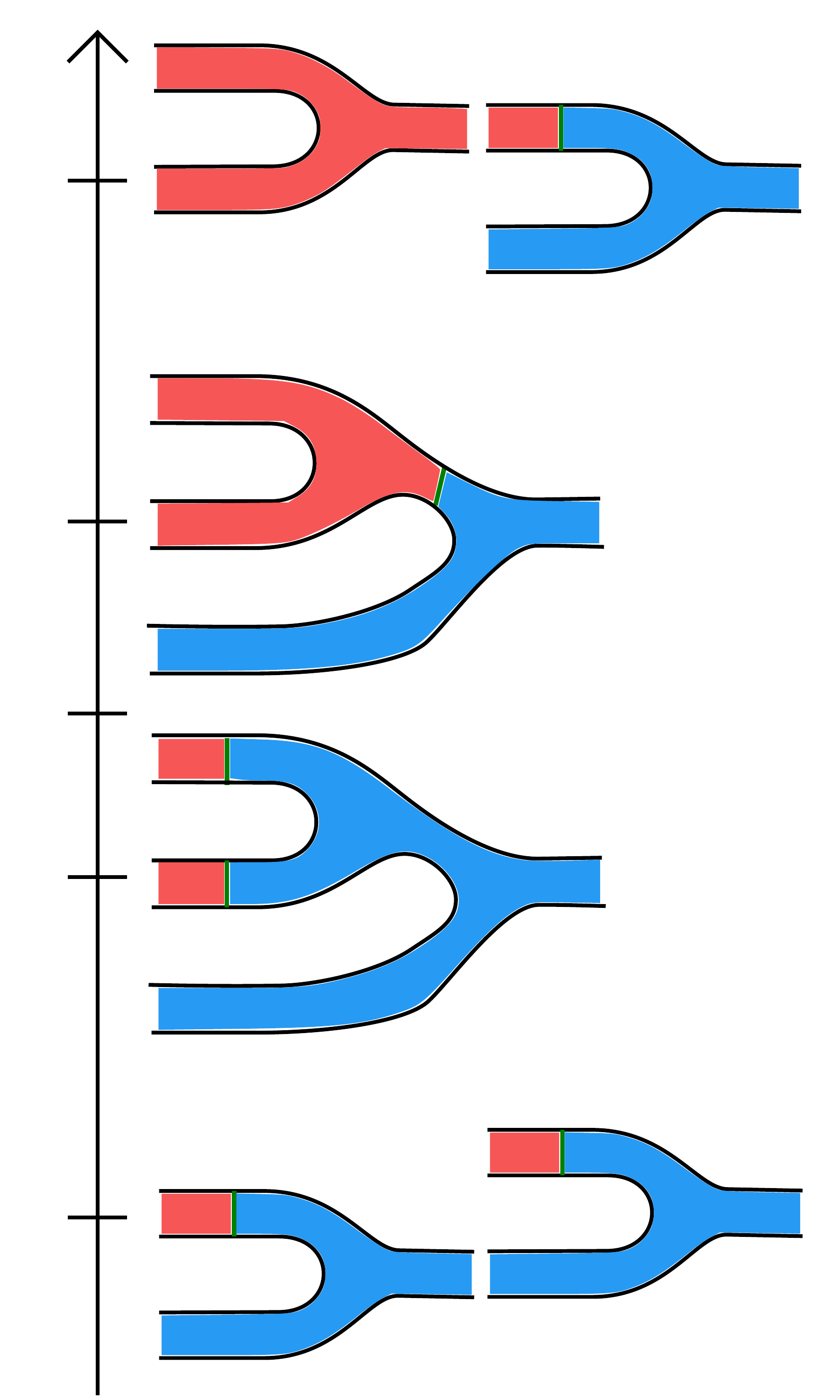}}%
    \put(-0.07777497,1.43768019){\color[rgb]{0,0,0}\makebox(0,0)[lt]{\lineheight{1.25}\smash{\begin{tabular}[t]{l}$+\infty$\end{tabular}}}}%
    \put(0.02211082,1.62493989){\color[rgb]{0,0,0}\makebox(0,0)[lt]{\lineheight{1.25}\smash{\begin{tabular}[t]{l}$t$\end{tabular}}}}%
    \put(0.00481329,0.8054832){\color[rgb]{0,0,0}\makebox(0,0)[lt]{\lineheight{1.25}\smash{\begin{tabular}[t]{l}0\end{tabular}}}}%
    \put(-0.0693653,0.20856389){\color[rgb]{0,0,0}\makebox(0,0)[lt]{\lineheight{1.25}\smash{\begin{tabular}[t]{l}$-\infty$\end{tabular}}}}%
  \end{picture}%
\endgroup%

      \caption{The moduli space $\Mcal = \bigcup_{t\in \rr} \Mcal_t$.}
      \label{fig:bimod_assoc_1}
\end{figure}

 Care must be taken however, since when $t=0$ one has a seam condition tangent to the boundary condition, so the puncture is not a strip like end. To remedy this we instead consider
\e
\Mcal_{ \leq -\epsilon} = \bigcup_{t\leq -\epsilon} \Mcal_t
\e
and at $t=-\epsilon$ we continue the moduli space by stretching away the puncture as in  Figure~\ref{fig:bimod_assoc_2}, i.e. we glue to it a parametrized moduli space
\e
\Acal = \bigcup_{L\geq 0} \Acal_L
\e
It compactifies to a moduli space $\overline{\Acal}$ by adding strip breaking at the ends at finite $L$ (giving some homotopy terms) and when  $L\to +\infty$ by adding $\Acal_\infty$ consisting of the picture drawn in Figure~\ref{fig:bimod_assoc_2}, with possibly a quilted disc bubble attached to it.

\begin{figure}[!h]
    \centering
    \def\svgwidth{.50\textwidth}
\begingroup%
  \makeatletter%
  \providecommand\color[2][]{%
    \errmessage{(Inkscape) Color is used for the text in Inkscape, but the package 'color.sty' is not loaded}%
    \renewcommand\color[2][]{}%
  }%
  \providecommand\transparent[1]{%
    \errmessage{(Inkscape) Transparency is used (non-zero) for the text in Inkscape, but the package 'transparent.sty' is not loaded}%
    \renewcommand\transparent[1]{}%
  }%
  \providecommand\rotatebox[2]{#2}%
  \newcommand*\fsize{\dimexpr\f@size pt\relax}%
  \newcommand*\lineheight[1]{\fontsize{\fsize}{#1\fsize}\selectfont}%
  \ifx\svgwidth\undefined%
    \setlength{\unitlength}{566.92913386bp}%
    \ifx\svgscale\undefined%
      \relax%
    \else%
      \setlength{\unitlength}{\unitlength * \real{\svgscale}}%
    \fi%
  \else%
    \setlength{\unitlength}{\svgwidth}%
  \fi%
  \global\let\svgwidth\undefined%
  \global\let\svgscale\undefined%
  \makeatother%
  \begin{picture}(1,0.37)%
    \lineheight{1}%
    \setlength\tabcolsep{0pt}%
    \put(0,0){\includegraphics[width=\unitlength,page=1]{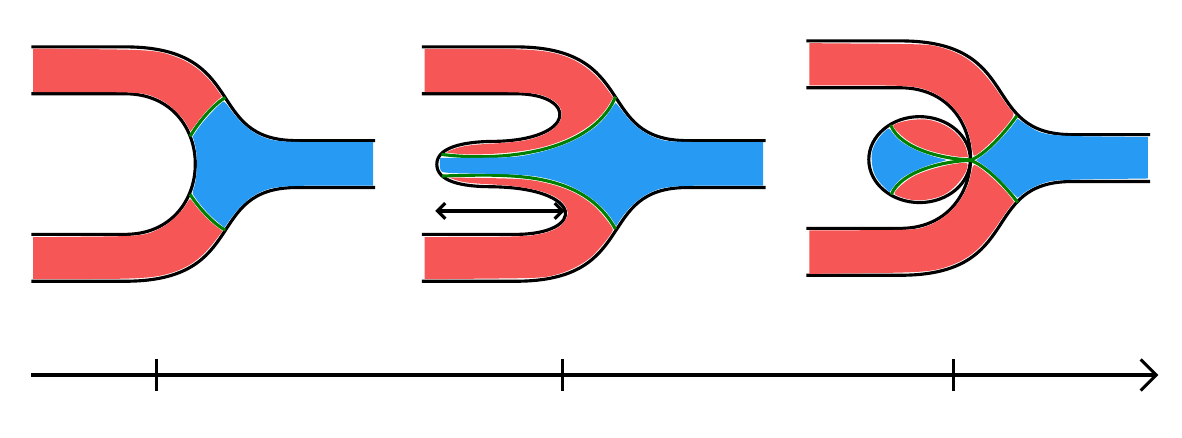}}%
    \put(0.12274789,-0.00309782){\makebox(0,0)[lt]{\lineheight{1.25}\smash{\begin{tabular}[t]{l}0\end{tabular}}}}%
    \put(0.45743405,-0.0037047){\makebox(0,0)[lt]{\lineheight{1.25}\smash{\begin{tabular}[t]{l}$L$\end{tabular}}}}%
    \put(0.40352463,0.1427407){\makebox(0,0)[lt]{\lineheight{1.25}\smash{\begin{tabular}[t]{l}$L$\end{tabular}}}}%
    \put(0.78396692,-0.00014427){\makebox(0,0)[lt]{\lineheight{1.25}\smash{\begin{tabular}[t]{l}$+\infty$\end{tabular}}}}%
  \end{picture}%
\endgroup%

      \caption{The moduli space $\overline{\Acal}$.}
      \label{fig:bimod_assoc_2}
\end{figure}

\begin{figure}[!h]
    \centering
    \def\svgwidth{.50\textwidth}
\begingroup%
  \makeatletter%
  \providecommand\color[2][]{%
    \errmessage{(Inkscape) Color is used for the text in Inkscape, but the package 'color.sty' is not loaded}%
    \renewcommand\color[2][]{}%
  }%
  \providecommand\transparent[1]{%
    \errmessage{(Inkscape) Transparency is used (non-zero) for the text in Inkscape, but the package 'transparent.sty' is not loaded}%
    \renewcommand\transparent[1]{}%
  }%
  \providecommand\rotatebox[2]{#2}%
  \newcommand*\fsize{\dimexpr\f@size pt\relax}%
  \newcommand*\lineheight[1]{\fontsize{\fsize}{#1\fsize}\selectfont}%
  \ifx\svgwidth\undefined%
    \setlength{\unitlength}{566.92913386bp}%
    \ifx\svgscale\undefined%
      \relax%
    \else%
      \setlength{\unitlength}{\unitlength * \real{\svgscale}}%
    \fi%
  \else%
    \setlength{\unitlength}{\svgwidth}%
  \fi%
  \global\let\svgwidth\undefined%
  \global\let\svgscale\undefined%
  \makeatother%
  \begin{picture}(1,0.37)%
    \lineheight{1}%
    \setlength\tabcolsep{0pt}%
    \put(0,0){\includegraphics[width=\unitlength,page=1]{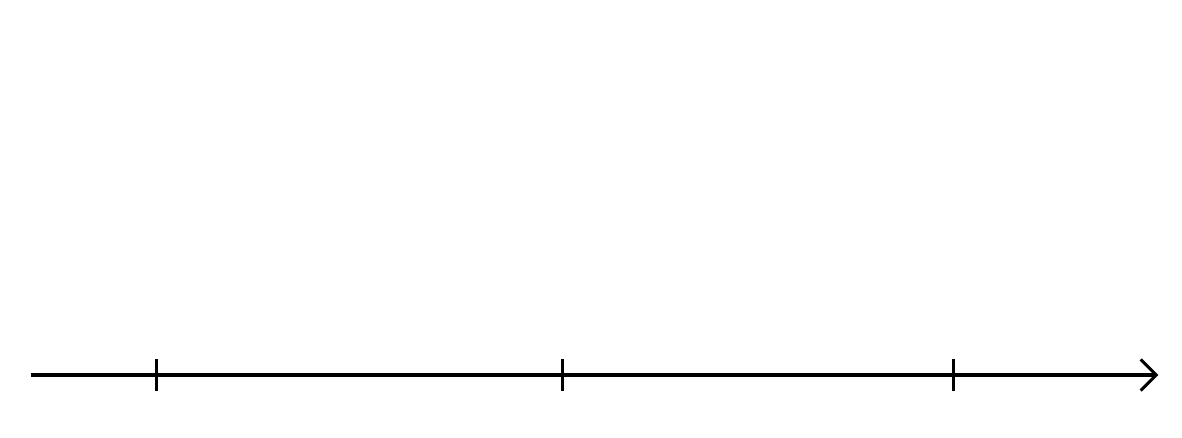}}%
    \put(0.12274789,-0.00309782){\makebox(0,0)[lt]{\lineheight{1.25}\smash{\begin{tabular}[t]{l}0\end{tabular}}}}%
    \put(0.45743405,-0.0037047){\makebox(0,0)[lt]{\lineheight{1.25}\smash{\begin{tabular}[t]{l}$L$\end{tabular}}}}%
    \put(0.78396692,-0.00014427){\makebox(0,0)[lt]{\lineheight{1.25}\smash{\begin{tabular}[t]{l}$+\infty$\end{tabular}}}}%
    \put(0,0){\includegraphics[width=\unitlength,page=2]{bimod_assoc_3.pdf}}%
    \put(0.40352463,0.1427407){\makebox(0,0)[lt]{\lineheight{1.25}\smash{\begin{tabular}[t]{l}$L$\end{tabular}}}}%
  \end{picture}%
\endgroup%

      \caption{The moduli space $\overline{\Bcal}$.}
      \label{fig:bimod_assoc_3}
\end{figure}

Likewise, at $t=\epsilon$ we do a similar stretching (see Figure~\ref{fig:bimod_assoc_3}) and get another parametrized moduli space
\e
\Bcal = \bigcup_{L\geq 0} \Bcal_L
\e
which compactifies to $\overline{\Bcal}$ by adding strip breaking at the ends at finite $L$  and when  $L\to +\infty$ by adding $\Bcal_\infty$ as drawn in Figure~\ref{fig:bimod_assoc_3}, with possibly a quilted disc bubble attached to it.

Now it remains to notice that both bubblings in $\Acal_\infty$ and $\Bcal_\infty$ are ruled out by our assumptions: in both cases after lifting and unfolding one gets a pair of discs in $(T_N, 0_N)$ and  $(M, L_1)$, as in Figure~\ref{fig:bimod_assoc_bubbling}. These discs, if nonconstant, would have a Maslov index too large. Therefore  $\Acal_\infty =\Bcal_\infty$ and one can form the moduli space 
\e
\overline{\Mcal}_{ \leq -\epsilon}   \bigcup_{\Mcal_{ -\epsilon} = \Acal_0} \overline{\Acal} \bigcup_{\Acal_\infty = \Bcal_\infty} \overline{\Bcal} \bigcup_{\Bcal_0 = \Mcal_{ \epsilon} } \overline{\Mcal}_{ \geq \epsilon} 
\e
whose zero dimensional part defines $\lambda^L_N$ and one dimensional part proves Equation~(\ref{eq:module_relation_left}).

\begin{figure}[!h]
    \centering
    \def\svgwidth{.80\textwidth}
    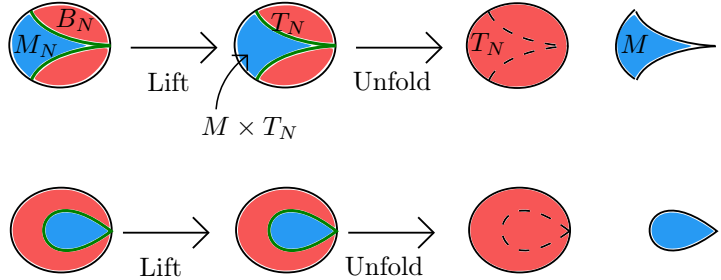
      \caption{Bubbling in $\Acal_\infty$ and $\Bcal_\infty$. After lifting and unfolding, one gets a pair of discs in $(T_N, 0_N)$ and  $(M, L_1)$, which must be constant.}
      \label{fig:bimod_assoc_bubbling}
\end{figure}

The fact that the left and right actions commute up to homotopy is a straightforward deformation argument suggested in Figure~\ref{fig:left_right_action_commute}.

\begin{figure}[!h]
    \centering
    \def\svgwidth{.40\textwidth}
\begingroup%
  \makeatletter%
  \providecommand\color[2][]{%
    \errmessage{(Inkscape) Color is used for the text in Inkscape, but the package 'color.sty' is not loaded}%
    \renewcommand\color[2][]{}%
  }%
  \providecommand\transparent[1]{%
    \errmessage{(Inkscape) Transparency is used (non-zero) for the text in Inkscape, but the package 'transparent.sty' is not loaded}%
    \renewcommand\transparent[1]{}%
  }%
  \providecommand\rotatebox[2]{#2}%
  \newcommand*\fsize{\dimexpr\f@size pt\relax}%
  \newcommand*\lineheight[1]{\fontsize{\fsize}{#1\fsize}\selectfont}%
  \ifx\svgwidth\undefined%
    \setlength{\unitlength}{209.76377953bp}%
    \ifx\svgscale\undefined%
      \relax%
    \else%
      \setlength{\unitlength}{\unitlength * \real{\svgscale}}%
    \fi%
  \else%
    \setlength{\unitlength}{\svgwidth}%
  \fi%
  \global\let\svgwidth\undefined%
  \global\let\svgscale\undefined%
  \makeatother%
  \begin{picture}(1,1.13513514)%
    \lineheight{1}%
    \setlength\tabcolsep{0pt}%
    \put(0,0){\includegraphics[width=\unitlength,page=1]{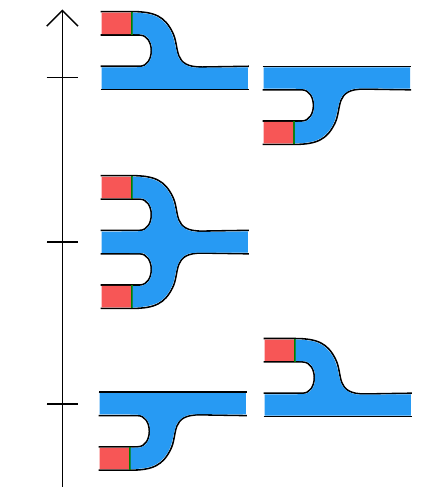}}%
    \put(-0.05350574,0.93341353){\color[rgb]{0,0,0}\makebox(0,0)[lt]{\lineheight{1.25}\smash{\begin{tabular}[t]{l}$+\infty$\end{tabular}}}}%
    \put(-0.00912747,0.55799101){\color[rgb]{0,0,0}\makebox(0,0)[lt]{\lineheight{1.25}\smash{\begin{tabular}[t]{l}$0$\end{tabular}}}}%
    \put(-0.04666247,0.19005056){\color[rgb]{0,0,0}\makebox(0,0)[lt]{\lineheight{1.25}\smash{\begin{tabular}[t]{l}$-\infty$\end{tabular}}}}%
  \end{picture}%
\endgroup%

      \caption{The left and right actions commute.}
      \label{fig:left_right_action_commute}
\end{figure}
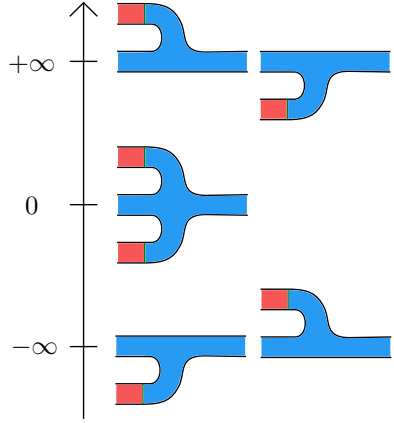

\begin{figure}[!h]
    \centering
    \def\svgwidth{1.0\textwidth}
    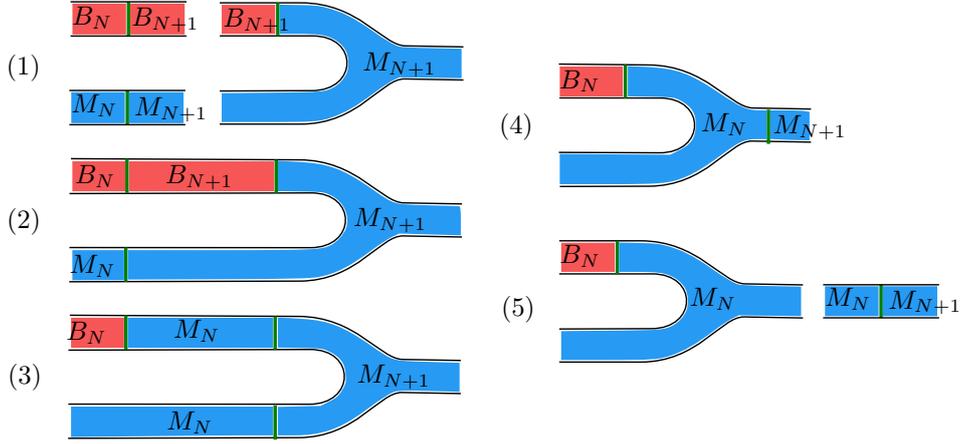
      \caption{$L_N$ commutes with the increments. Details about the strip-shrinking from Step (2) to Step (3) are indicated in Figure~\ref{fig:bimod_commutes_f_N_3}. Details for going from Step (3) to Step (4) will be shown in Figure~\ref{fig:bubbling_bimod_comm_3}. }
      \label{fig:bimod_commutes_f_N}
\end{figure}

The proof that $L_N$ and $R_N$ commute with the increment maps up to homotopies is a similar deformation argument, suggested in Figure~\ref{fig:bimod_commutes_f_N}. Again, the following problems can occur, which we will indicate further below how we handle them.
\begin{itemize}
\item Between steps (2) and (3) when the two seams come together: different strip-like ends have to come together, so the strip-like end structure wouldn't be fixed. Therefore, as in the previous argument, we first stretch the free ends away as in Figure~\ref{fig:bimod_commutes_f_N_3}. 
Then one can shrink the width between the two seams and replace the two seams by a single one decorated by the composition $\Lambda_{B_N} \circ P_1^{N+1}$, with $\Lambda_{B_N} = (N_{\Gamma(i_N)}\cap \mu_{T_N^- \times T_{N+1}}^{-1}(0))/G^2$. As this composition is embedded and equal to $P_1^{N}\circ \Lambda_{N}$, one can continue the moduli space as drawn in the picture.

\item The transition from step (3) to (4), which is similar to the breaking observed in the proof of equation (\ref{eq:module_relation_left}), and handled in the same way.

\end{itemize}

\begin{figure}[!h]
    \centering
    \def\svgwidth{1.0\textwidth}
    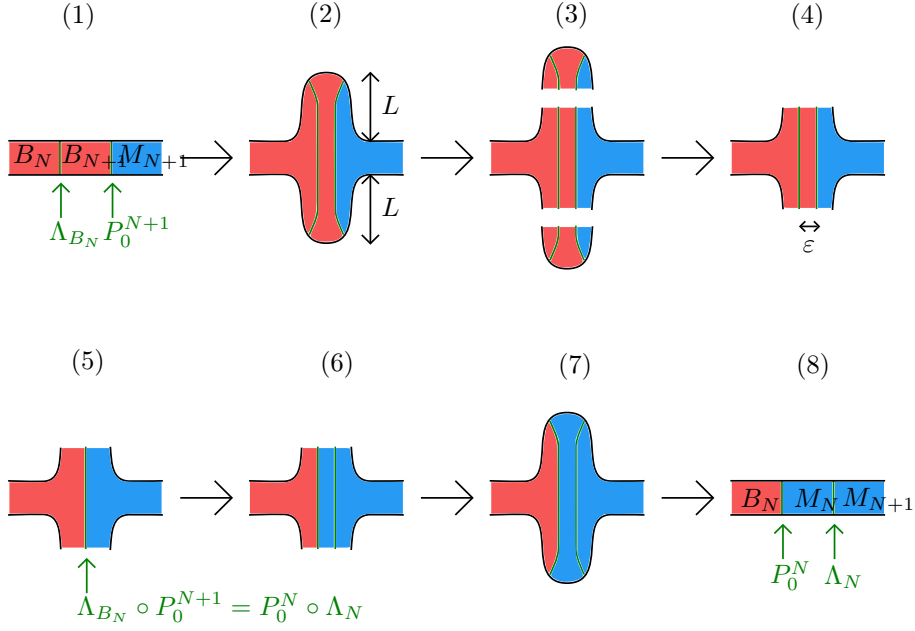
      \caption{From Step (2) to Step (3) in Figure~\ref{fig:bimod_commutes_f_N}. Figure 8 bubbling appearing at Step (5) will be ruled out, as indicated in Figure~\ref{fig:bubbling_bimod_comm_1}. Likewise, breakings at Steps (3) and (7) will be ruled out, as indicated in Figures~\ref{fig:bubbling_bimod_comm_2} and \ref{fig:bubbling_bimod_comm_2_}.}
      \label{fig:bimod_commutes_f_N_3}
\end{figure}

Along this process, several kinds of bubblings/breakings might a priori occur, in addition to the ones that we already ruled out. We now detail each of these:
\begin{itemize}
\item At step (5) of  Figure~\ref{fig:bimod_commutes_f_N_3}, some figure 8 bubbling can occur in the strip-shrinking process, from Bottman's removal of singularity theorem \cite{Bottman_remov_sing}. These are drawn in Figure~\ref{fig:bubbling_bimod_comm_1}: the lefthand column corresponds to the transition from step (4) to step (5), and the righthand column corresponds to the transition from step (6) to step (5). 
As drawn in Figure~\ref{fig:bubbling_bimod_comm_1}, such bubbles, after lifting and unfolding, give rise to a quilted sphere (i.e. a disc in $T_N\times T_{N+1}$ with boundary in $N_{\Gamma(i_N)}$), and a disc in $M$ with boundary in $L_1$. For example, the top-right figure 8 bubble in $(B_N, B_{N+1}, M_{N+1})$ can first be lifted to a similar figure 8 bubble in $(T_N, T_{N+1}, M\times T_{N+1})$. Then the patch in $M\times T_{N+1}$ can be unfolded, giving rise to a quilted sphere in $(T_N, T_{N+1})$ and a disc in $M$. And finally this last quilted sphere can be folded to a disc in $T_N^-\times  T_{N+1}$.

\begin{figure}[!h]
    \centering
    \def\svgwidth{.80\textwidth}
    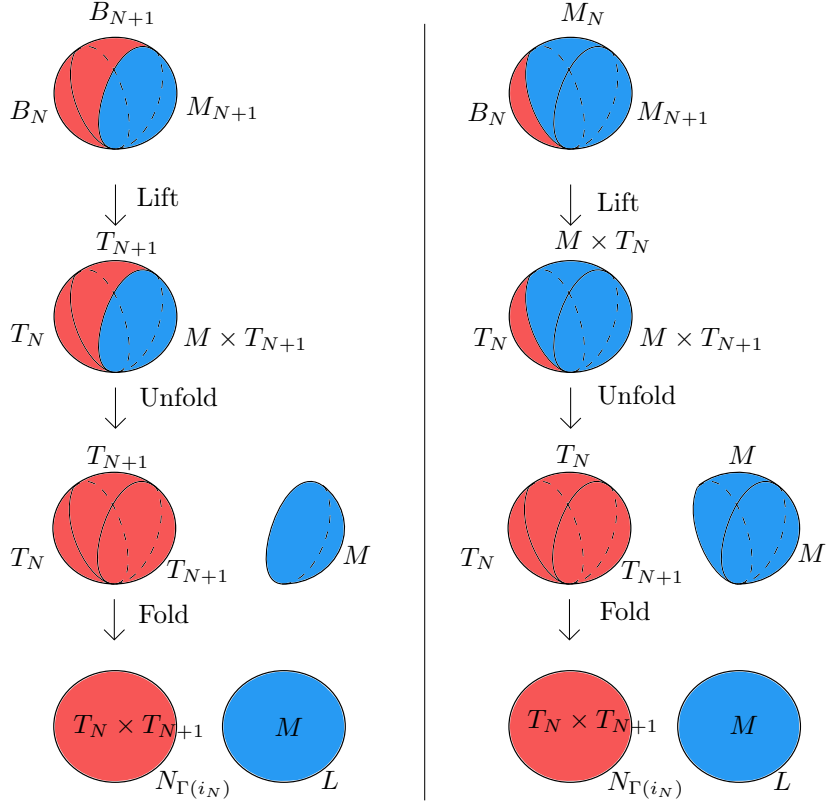
      \caption{Bubbling at Step (5) of Figure~\ref{fig:bimod_commutes_f_N_3}.  After lifting and unfolding, one gets a quilted sphere and a disc in $M$ with boundary in $L_1$. The quilted sphere can be folded to a disc in $(T_N\times T_{N+1},0_N \times 0_{N+1})$.}
      \label{fig:bubbling_bimod_comm_1}
\end{figure}

\item At steps (3) and (7) of  Figurs~\ref{fig:bimod_commutes_f_N_3}, one can observe the breakings of  Figures~\ref{fig:bubbling_bimod_comm_2} and \ref{fig:bubbling_bimod_comm_2_}: these are disc analogues of the previous figure 8 bubbles. Consider the breaking of Figure~\ref{fig:bubbling_bimod_comm_2}, corresponding to step (3). After lifting and unfolding, one gets a quilted sphere and a disc in $M$ with boundary in $L_1$. By folding the quilted sphere along the seam between $T_N$ and $T_{N+1}$, one gets a disc in $T_N\times T_{N+1}$ with one boundary in $0_N \times 0_{N+1}$, and another boundary in $N_{\Gamma(i_N)}$). As in the proof of Proposition~\ref{prop:Lcal_moduli_space}, one can cap the part of the boundary in $N_{\Gamma(i_N)}$ with a disc in  $N_{\Gamma(i_N)}$, and get a disc with the whole boundary in $0_N \times 0_{N+1}$.

\begin{figure}[!h]
    \centering
    \def\svgwidth{1.0\textwidth}
    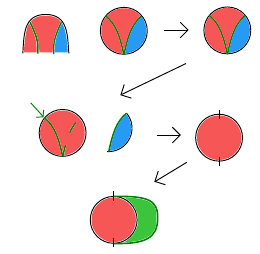
      \caption{Strip breaking at Step (3) of Figure~\ref{fig:bimod_commutes_f_N_3}. After lifting and unfolding, one gets a quilted disc and a disc in $M$ with boundary in $L_1$. The quilted disc can be folded to a disc in $T_N\times T_{N+1}$ with boundary in $0_N \times 0_{N+1}$ and $N_{\Gamma(i_N)}$. Capping at the second boundary with a disc in $N_{\Gamma(i_N)}$ yields a disc in $T_N\times T_{N+1}$ with boundary in $0_N \times 0_{N+1}$.}
      \label{fig:bubbling_bimod_comm_2}
\end{figure}

\begin{figure}[!h]
    \centering
    \def\svgwidth{1.0\textwidth}
    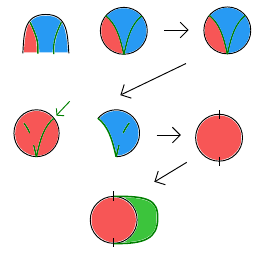
      \caption{Strip breaking at Step  (7) of Figure~\ref{fig:bimod_commutes_f_N_3}.}
      \label{fig:bubbling_bimod_comm_2_}
\end{figure}

\item Between steps (3) and (4) of Figure~\ref{fig:bimod_commutes_f_N}, there can be breaking as in  Figure~\ref{fig:bubbling_bimod_comm_3}. After lifting and unfolding, in both cases one gets a quilt on $T_N$, $T_{N+1}$, and a disc in $M$ with boundary in $L_1$. 

For case (A), we cap the quilt with a strip in $0_N$ as shown in the picture, so that the true boundary is mapped to a constant value (this is possible since $0_N$ is simply connected). This gives a quilted sphere, which folds to a disc in $T_N \times T_{N+1}$ with boundary in $N_{\Gamma(i_N)}$.

For case (B), the quilted part is different, with one boundary in $0_N\times 0_{N+1}$ and another in $N_{\Gamma(i_N)}$. We cap the part in $N_{\Gamma(i_N)}$ with a strip in $N_{\Gamma(i_N)}$, such that the true boundary gets mapped in ${\Gamma(i_N)} \subset 0_N\times 0_{N+1}$. We therefore end up with a disc in $T_N \times T_{N+1}$ with boundary in $ 0_N\times 0_{N+1}$.

\begin{figure}[!h]
    \centering
    \def\svgwidth{1.0\textwidth}
    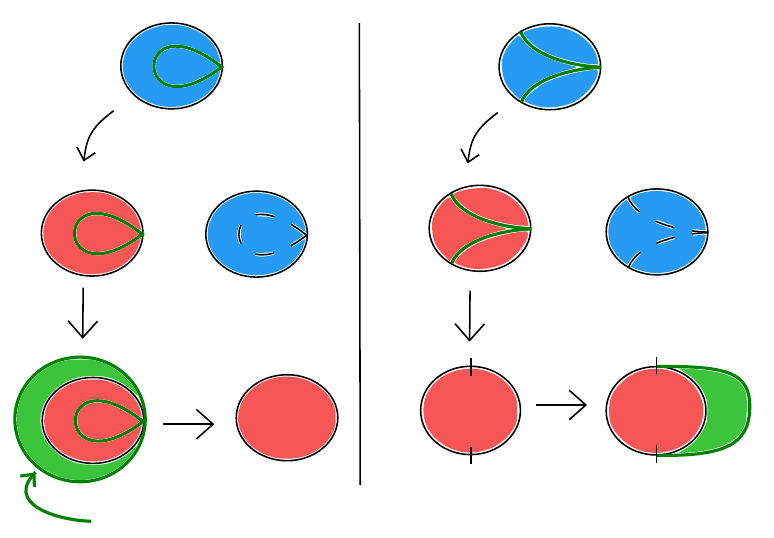
      \caption{Strip breaking between Steps (3) and (4) of Figure~\ref{fig:bimod_commutes_f_N}. The red quilted disc in (A) line 2 can be capped, so that we can view it as a quilted sphere, and fold it to a disc in $(T_N\times T_{N+1}, N_{\Gamma(i_N)})$. For the red quilted disc in (B) line 2, we concatenate both patches in $T_N$ to a single one. This permits to fold it to a bigon. Capping the part of the boundary in $N_{\Gamma(i_N)}$ yields a disc in $T_N\times T_{N+1}$ with boundary in $0_N\times 0_{N+1}$.}
      \label{fig:bubbling_bimod_comm_3}
\end{figure}
\end{itemize}

The same arguments as in the proof of Proposition~\ref{prop:Lcal_moduli_space} allows us to show that all these bubblings/breakings are excluded by our assumptions. This ends the proof.
\end{proof}

\begin{prop}\label{prop:PSS_mod_str}
The isomorphisms $PSS_G$ and $SSP_G$ defined in Section~\ref{sec:Floer_vs_Morse} preserve the $H^*(BG)$-bimodule structures.
\end{prop}

\begin{proof}
This is a routine deformation argument, using the deformation of Figure~\ref{fig:PSS_mod_str}.
\end{proof}

\begin{figure}[!h]
    \centering
    \def\svgwidth{1.0\textwidth}
    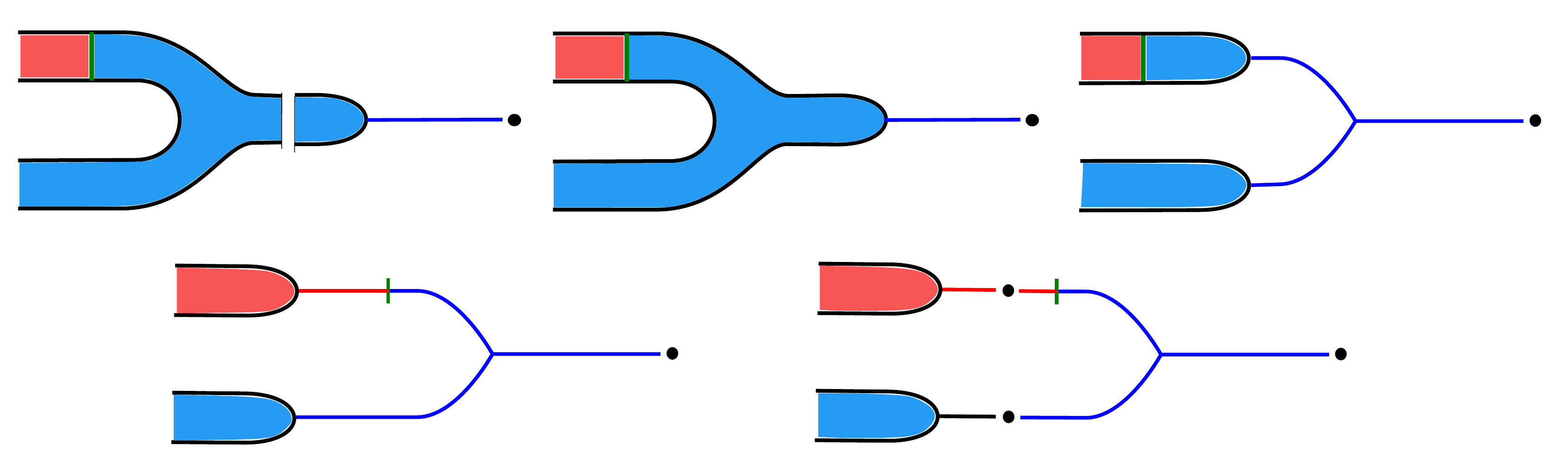
      \caption{The deformation showing that the PSS and SSP isomorphisms preserve the bimodule structures.}
      \label{fig:PSS_mod_str}
\end{figure}

\section{Relations with the symplectic quotient (Kirwan maps)}
\label{sec:Kirwan_maps}

Assume in this section that the action of $G$ on $M$ is regular in the sense of Definition~\ref{def:reduction}, so that $M\red G$ is smooth and symplectic, $L_0/G, L_1/G \subset M\red G $ are smooth Lagrangians satisfying either Assumption~\ref{ass:exact_setting} or Assumption~\ref{ass:monotone_setting} (without Hamiltonian actions). The group $HF(M\red G; L_0/G, L_1/G )$ is then well-defined, and we aim to compare it with $HF_G(M; L_0, L_1)$.

For $N\geq 1$, one has a sequence of Lagrangian correspondences:
\e
\xymatrix{ M_N \ar[r]^{\mu_N^{-1}(0)} & M\times T_N \ar[r]^{\Delta_M \times 0_N} & M \ar[r]^{\mu^{-1}(0)}& M\red G , }
\e
with $\mu_N\colon M\times T_N \to \g^*$ defined by $\mu_N(m,t)=\mu(m)+ \mu_{T_N}(t)$. This sequence of correspondences induces  morphisms
\ea
K_N &\colon CF(M_N;L_0^N,L_1^N ) \to CF(M\red G;L_0/G, L_1/G ) , \\ 
K_N' &\colon CF(M\red G;L_0/G, L_1/G ) \to  CF(M_N;L_0^N,L_1^N ), 
\ea
defined by counting quilts as in the left of Figure~\ref{fig:Kirwan_maps}: the vertical seams are decorated by the corresponding correspondences, while the boundaries in $M_N,  M\times T_N ,  M , M\red G$, are decorated respectively by $L_i^N, L_i \times 0_N, L_i, L_i/G$ (with $i=0, 1$). Notice that by unfolding the patches in $M\times T_N$, one can alternatively view these quilts as ``\emph{foams}'' as in the right side of Figure~\ref{fig:Kirwan_maps}.

\begin{figure}[!h]
    \centering
    \def\svgwidth{1.0\textwidth}
    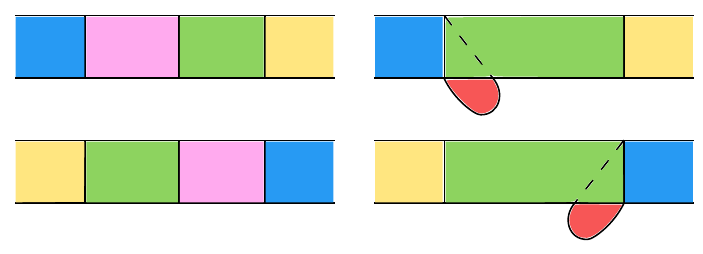
      \caption{Quilted strips defining $K_N$ and $K_N'$. The pink patches in $M\times T_N$ can be unfolded, yielding the foams on the right.}
      \label{fig:Kirwan_maps}
\end{figure}
 As usual, one can show that these commute with the increment maps up to homotopies $\d \kappa_N + \kappa_N \d$ and $\d \kappa_N' + \kappa_N' \d$, and therefore induce maps between telescopes:
\ea
K = \Tel(K_N, \kappa_N) &\colon CF_G(M;L_0,L_1 ) \to CF(M\red G;L_0/G, L_1/G ) , \\ 
K' = \Tel(K_N', \kappa_N')&\colon CF(M\red G;L_0/G, L_1/G ) \to  CF_G(M;L_0,L_1 ) . 
\ea

\begin{remark}If $M$ is a compact symplectic manifold with a regular $G$-action, then its equivariant cohomology is related to the cohomology of the symplectic quotient by the Kirwan map:
\e
H^*_G(M) \to H^*(M\red G) .
\e
(obtained by composing the pullback of the inclusion $\mu^{-1}(0)\subset M$ with the Cartan isomorphism $H^*_G(\mu^{-1}(0)) \simeq H^*(M\red G)$.) Kirwan showed in \cite{Kirwan_thesis} that this map is surjective. However it is not injective in general.
\end{remark}

\begin{remark} In the setting of Section~\ref{sec:Floer_vs_Morse}, where $G$ acts on a smooth manifold $X$, $M= T^* X$ and $L_0 = L_1 = 0_X$, if furthermore the action of $G$ on $X$ is free, then the action on $M$ is regular. And under the PSS isomorphisms, the morphisms $K, K'$ correspond at the homology level to the Cartan isomorphisms $H^G(X) \simeq H(X/G)$. One can ask whether $K, K'$  induce isomorphisms more generally. Notice that this would not contradict the non-injectivity of the classical Kirwan maps (at least in the obvious way) since under the PSS isomorphisms $K$ does not correspond to a Kirwan map.
\end{remark}

\section{Equivariant Manolescu-Woodward's Symplectic Instanton homology}
\label{sec:Equivariant_HSI}

We now apply our construction to the setting of \emph{Symplectic Instanton Homology} defined by Manolescu and Woodward \cite{MW}, which is a slightly more involved one from the monotone setting of Section~\ref{ssec:monotone_setting}. We start by quickly reviewing their construction, and refer to \cite{MW} for more details.

\subsection{Manolescu-Woodward's Symplectic Instanton homology}
\label{ssec:HSI}

Let $\Sigma'$ be a connected oriented  surface of genus $g$ with one boundary component, and let $G= SU(2)$ throughout this section. Associated to it is the \emph{extended moduli space} defined by Jeffrey \cite{jeffrey} $(M, \omega) = \Mg(\Sigma')$. This space is a (singular, degenerate) symplectic manifold with a Hamiltonian $G$-action with moment map $\mu$. The important features of it is that it is smooth and nondegenerate near $\mu^{-1}(0)$, and its symplectic quotient is identified with the (singular) Atiyah-Bott flat moduli space $\M(\Sigma)$ of $\Sigma$, the closed surface obtained by capping $\Sigma'$ with a disc. If $Y = H_0 \cup_{\Sigma} H_1$ is a closed oriented connected 3-manifold with $\Sigma$ as a Heegaard splitting, then associated to the two handlebodies is a pair of smooth $G$-Lagrangians $L_0, L_1 \subset \Mg(\Sigma')$.

The Atiyah-Floer conjecture states that the instanton homology of $Y$ should correspond to $HF(\M(\Sigma); L_0/G, L_1/G)$, wich is ill-defined since $\M(\Sigma),  L_0/G$ and $L_1/G$ are singular. However, (after a cutting construction on $\Mg(\Sigma')$ that we will outline) Manolescu and Woodward succeeded in defining $HF(\Mg(\Sigma'); L_0, L_1)$ and suggested defining $HF_G(\Mg(\Sigma'); L_0, L_1)$ as an equivariant symplectic side for the Atiyah-Floer conjecture.

We now outline their cutting construction. Equip $\g \subset \hh$ with the standard inner product of $\hh$, and let $\tilde{\mu}\colon M\to \rr$ be defined by $\tilde{\mu}(m) = \abs{\mu(m)}$, it is the moment map of the circle action on $M\setminus \mu^{-1}(0)$ defined by 
\e
 m\mapsto e^{u \frac{\mu(m)}{2\tilde{\mu}(m)}}m , \text{ for }u\in U(1)\simeq \rr/ 2\pi \zz .
\e 
It turns out that $\tilde{\mu}^{-1}( [0, 1) )$ is smooth, and that $\omega$ is non-degenerate on 
\e
\N := \tilde{\mu}^{-1}( [0, 1/2) ) \subset M
\e 
(but is degenerate on $\tilde{\mu}^{-1}( 1/2) $). Furthermore, $M$ is $1/4$-monotone, and its minimal Chern number is a multiple of 4.

Morally speaking, $HSI$ is defined in $\N$, but this space is noncompact, and a priori not convex at infinity. To get a compact space, \MW \  consider the symplectic cutting 
\e
M_{\leq \lambda} = \Phi^{-1}( [0, \lambda) ) \cup (\Phi^{-1}( \lambda ))/ U(1) .
\e
It turns out that $\lambda = 1/2$ is the only value for which $M_{\leq \lambda} $ is monotone, but unfortunately the symplectic form is degenerate on $(\tilde{\mu}^{-1}( 1/2 ))/ U(1)$.  To remedy this, they also consider, for small $\epsilon>0$, the cutting $M_{\leq 1/2 - \epsilon}$. This cutting is symplectic (but not monotone) and diffeomorphic to $M_{\leq 1/2}$ via a diffeomorphism 
\e
\phi_\epsilon \colon M_{\leq 1/2} \to M_{\leq 1/2 - \epsilon}
\e 
that is a symplectomorphism away from a neighborhood of the cut locus. In the end, they get two 2-forms $\omega_{\leq 1/2}$ and $\phi_\epsilon ^* \omega_{\leq 1/2 -\epsilon}$ on 
\e
\N ^c =  \colon M_{\leq 1/2} = \N \cup R,
\e 
with $R = \tilde{\mu}^{-1}(0)/ U(1)$, and with the interplay of these two 2-forms and a good understanding of the way $\omega_{\leq 1/2}$ degenerates, they are able to show that the Floer homology in $\N^c$ \emph{relative} to $R$ (i.e. counting curves not intersecting $R$) is well-defined.

\subsection{Equivariant Symplectic Instanton homology}
\label{ssec:HSI_G}

Our strategy will be to apply \MW's construction to the symplectic homotopy quotients of $M$, rather than to $M$.

Let then $M_N = (M\times T_N)\red G$. It has a projection $\pi_N \colon M_N \to B_N$ as in Section~\ref{ssec:Symplectic_homotopy_quotients}.

The circle $U(1)$ acts on $M_N$ by $u. [m,t] = [u. m,t]$, with moment map $\tilde{\mu}_N\colon M_N \to \rr$ given by $\tilde{\mu}_N([m,t])= \tilde{\mu}(m)$. Let then 
\ea 
\N_N &= \tilde{\mu}_N^{-1}([0, 1/2) ) \text{, and} \\ 
\N^c_N &= (M_N)_{\leq 1/2} =\N_N \cup R_N \text{, with }R_N = \tilde{\mu}_N^{-1}( 1/2 )/ U(1).
\ea
This last moduli space still has a projection $\pi_N^c\colon \N^c_N \to B_N$, with fibers $\N^c$, and a fiber-preserving action of $G$. 

Let $\epsilon >0$, and suppose we are given $\phi_\epsilon \colon  M_{\leq 1/2} \to M_{\leq 1/2 - \epsilon}$ as above, and such that it is the identity on $\Wcal := M_{< 1/2 - 2\epsilon}$, and $G$-equivariant. Such a diffeomorphism can be constructed using the gradient flow of $\tilde{\mu}_N$ with respect to a $G$-invariant metric. Define then
\e
\phi_\epsilon^N \colon \N^c_N \to (M_N)_{\leq 1/2-\epsilon}
\e
by $\phi_\epsilon^N ([m,t]) = [\phi_\epsilon (m),t]  $. This extends to the cut locus by U(1)-equivariance of $\phi_\epsilon$. Then, with $\omega_{M_N}$ the symplectic form of $M_N$, the two forms
\ea
\tilde{\omega}_N &= (\omega_{M_N})_{\leq 1/2} \\
{\omega}_N &= (\phi_\epsilon^N)^* (\omega_{M_N})_{\leq 1/2 - \epsilon} 
\ea
coincide on $\Wcal_N  := (M_N)_{< 1/2 - 2\epsilon}$.

Fix a ground almost complex structure $J_N$ on $\N_N$ such that $R_N$ is an almost complex submanifold w.r.t. $J_N$, and such that the projection $\pi_N$ is $J_N$-holomorphic (w.r.t a fixed \acs\ on $B_N$). Let $\Jcal_N$ be the set of \acs s on  $\N_N$ that are equal to $J_N$ outside a compact subset of $\Wcal_N$.

To the handlebodies of the Heegaard splitting $Y = H_0 \cup_{\Sigma} H_1$ is associated a pair of $G$-Lagrangians $L_0, L_1 \subset \mu^{-1}(0) \subset M$. These are simply connected (and therefore monotone),  spin, and their minimal Maslov number is equal to $2 N_M$. So we get two Lagrangians in $\N^c_N$ disjoint from $R_N$:
\e
L_0^N = (L_0 \times 0_N)/G,\ L_1^N = (L_1 \times 0_N)/G \subset \Wcal_N \subset \N^c_N .
\e

In the end we are given $(\N_N^c , R_N, \Wcal_N, \omega_N, \tilde{\omega}_N, J_N, L_0^N, L_1^N )$, and restricted to a fiber of $\pi_N^c\colon \N^c_N \to B_N$, this is exactly the setting of \MW, therefore it satisfies the list of assumptions \cite[Assumption~2.5]{MW} in restriction to any fiber.

Let then $CF_N = CF( \N_N^c ; L_0^N, L_1^N ; R_N)$ denote the Floer complex of $L_0^N, L_1^N$ \emph{relative to} $R_N$, i.e. its Floer differential is defined by counting strips $u$ with intersection number $u. R_N =0 $, for a generic \acs\ $J$ in $\Jcal_N$ (see \cite[Section~2.2]{MW} for more about relative Lagrangian Floer homology). By positivity of intersection, such curves are actually disjoint from $R_N$.

The constructions in the previous sections carry through this setting:

\begin{theo}\label{th:equiv_HSI} As defined above, $CF_N$ is a chain complex (i.e. $\partial ^2 = 0$). Moreover, the increment $\alpha_N \colon CF_N \to CF_{N+1}$ defined as in Section~\ref{ssec:construction} (and counting quilts not intersecting $R_N$ and $R_{N+1}$) are chain morphisms.

Let then $CSI_G(Y) = \Tel (CF_N, \alpha_N)$, and $HSI_G(Y)$ its homology group. It is relatively $\Z{8}$-graded, and in the case when $Y$ is a rational homology sphere, an absolute $\Z{8}$-grading can be fixed canonically. Furthermore, it has the structure of an $H^*(BG)$-module, as defined in Section~\ref{sec:module_str}.

\end{theo}

\begin{proof} We need to make sure that moduli spaces of curves behave the same way as in the previous sections, i.e. no new bubbling or breaking arises.

Suppose $b$ is a bubble/breaking arising in a moduli space relevant to the theorem (either a strip for $\partial$, a quilted strip for $\alpha_N$, or a quilted pair of pants for the module structure). If $b$ has nonzero area for the monotone form $\tilde{\omega}_N$, then it is ruled out for exactly the same reasons as previously. Namely, it would have an index too large, forcing the principal component to live in a moduli space of negative dimension. Therefore, assume that $b$ has zero $\tilde{\omega}_N$-area, and is therefore contained in $R_N$. 

Since the Lagrangians $L_0^N, L_1^N$ are disjoint from $R_N$, this rules out any degeneration having boundary or seam conditions in these: disc bubbling, quilted strip breaking, or quilted sphere bubbling at the seam for the module structure (which is decorated by $(\Delta_{T_N}\times L_i)/G^2$). One is left with sphere bubbling in $M$, and quilted sphere bubbling at the seam for the increment maps moduli spaces.

Assume first that $b$ is a sphere bubbling. Since $b$ is in $R_N$ and that the \acs\ equals $J_N$, its image by $\pi_N^c$ in $B_N$ is a pseudo-holomorphic sphere, and therefore is constant. Therefore $b$ is included in a fiber of $\pi_N^c$, and the same reasoning as in \cite[Prop.~2.10]{MW} applies, which we briefly sketch for the reader's convenience. The bubble $b$ must have intersection number with $R_N$ at most $-2$, this would force the principal component (which generically intersects $R_N$ transversely) to intersect $R_N$ at points where no bubbles are attached, and this is impossible for a limit of curves disjoint from $R_N$. 

Assume then that $b$ is a quilted sphere bubble at the seam of the increment, i.e. equal to  $(d_N, d_{N+1})$ consists of two discs in $R_N, R_{N+1}$ satisfying a seam condition in the correspondence $\Lambda_N$. This projects to a quilted sphere in $B_N, B_{N+1}$, i.e. a disc in $(B_N\times  B_{N+1}, N_{\Gamma(i_N)}/G^2)$, which is constant. Therefore $d_N$ and $d_{N+1}$ are contained in fibers of $\pi_N^c$, $\pi_{N+1}^c$. These two fibers can both be identified to $\N^c$, and under this identification and possibly after multiplying by an element in $G$, $d_N$ and $d_{N+1}$ glue together to a sphere in $R$, which must have intersection with $R$ less than $-2$, which leads to the same contradiction as for sphere bubbling.

About gradings, \MW\ showed \cite[Corollary~3.6]{MW} that the minimal Chern number of $\Ncal$ is a positive multiple of 4, which implies that the minimal Maslov numbers of $L_0$ and $L_1$  (and therefore of $L_0^N$ and $L_1^N$) are positive multiples of 8. If $Y$ is a rational homology sphere, by \cite[Prop.~III.1.1.(c)]{cassonpup} $L_0$ and $L_1$ intersect transversely at the trivial representation: this fact was used by \MW\ to define an absolute grading on $HSI(Y)$. It follows that $L_0^N$ and $L_1^N$ intersect cleanly along a copy of $0_{B_N}$. One can then perturb this intersection by a Morse function of $0_{B_N}$, and take the Morse indices as an absolute grading for the corresponding intersection points.
\end{proof}

\begin{remark}\label{rem:indep_Heegaard_splitting} By following the same proof as in \cite[Section~6]{MW}, one should be able to show that $HSI_G(Y)$ is independent of the choice of Heegaard splitting $\Sigma$. This would amount to defining a quilted version of equivariant Lagrangian Floer homology, and proving an equivariant analogue of \WW 's geometric composition theorem (which, for our construction, should be a straightforward consequence of the non-equivariant version). 
Of course, another way of proving independence would be through an equivariant version of the Atiyah-Floer conjecture: we expect $HSI_G(Y)$ to be isomorphic to a version of equivariant instanton homology defined in \cite{Miller_equiv} (or something similar, Miller Eismeier defined SO(3)-equivariant versions, while ours is SU(2)-equivariant), and \cite{DaemiMiller} prove topological invariance of it. Indeed, Miller Eismeier's construction considers the framed instanton complex $\widetilde{CI}(Y)$, on which a certain chain complex associated with SO(3) acts. He then defines equivariant instanton homology using a Bar construction. The Atiyah-Floer counterpart of the framed instanton complex $\widetilde{CI}(Y)$ is the Floer complex $CF(\Ncal; L_0, L_1)$ considered by \MW, see \cite[Conjecture~1.2]{MW}. Therefore, it is expected that their equivariant counterparts will be related as well. 
\end{remark}

\bibliographystyle{alpha}
\bibliography{biblio}

\end{document}